\documentclass[reqno,12pt]{amsart}
\usepackage[margin=1in,footskip=0.5in]{geometry}
\usepackage{enumitem}

\usepackage[latin1]{inputenc}
\usepackage[english]{babel}
\usepackage{subfigure}
\usepackage{amsmath,amssymb,amsthm}
\usepackage{ae}
\usepackage[T1]{fontenc}
\usepackage{graphicx}
\usepackage{epstopdf}
\usepackage{color}
\usepackage{amsaddr}
\newcommand{\e}{{\mathrm e}}

\newcommand{\be}{\begin{eqnarray}}
\newcommand{\ee}{\end{eqnarray}}
\newcommand{\ba}{\begin{align}}
\newcommand{\ea}{\end{align}}
\newcommand{\bi}{\begin{itemize}}
\newcommand{\ei}{\end{itemize}}
\usepackage{fullpage}

\newcommand{\eqlab}[1]{\label{eq:#1}}
\renewcommand{\eqref}[1]{(\ref{eq:#1})}

\newcommand{\figref}[1]{Fig.~\ref{fig:#1}}
\newcommand{\figlab}[1]{\label{fig:#1}}

\newcommand{\remref}[1]{Remark~\ref{remark:#1}}
\newcommand{\remlab}[1]{\label{remark:#1}}

\newcommand{\appref}[1]{Appendix~\ref{app:#1}}

\newtheorem{theorem}{Theorem}[section]
\newtheorem{proposition}[theorem]{Proposition}

\newtheorem{lemma}[theorem]{Lemma}
\newtheorem{cor}[theorem]{Corollary}

\newtheorem{remark}[theorem]{Remark}

\newcommand{\assumptionref}[1]{Assumption~\ref{assumption:#1}}
\newcommand{\assumptionlab}[1]{\label{assumption:#1}}
\newtheorem{assumption}{Assumption}
\numberwithin{equation}{section}

\begin{document}
\title{A geometric approach to exponentially small splitting: Zero-Hopf bifurcations of arbitrary co-dimension}
%\title{Exponentially small splitting in zero-Hopf bifurcations of arbitrary co-dimension}
\author{K. Uldall Kristiansen}
\address{Department of Applied Mathematics and Computer Science,
Technical University of Denmark,
2800 Kgs. Lyngby,
Denmark }

 \begin{abstract}
  In this paper, we present a geometric approach to exponentially small splitting in zero-Hopf bifurcations of arbitrary co-dimension. In further details, we consider a family of problems that generalizes the third order Michelsen/Kuramoto-Sivashinsky-type equations $\epsilon^{2(\kappa-1)} f'''+f'={Q}(f)$, where ${Q}$ is an arbitrary real polynomial with $\kappa=\operatorname{degree}Q\ge 2$ simple real roots. For $\epsilon=0$, the system has $(\kappa-1)$-many heteroclinic connections and we describe the exponentially small splitting for each connection for all $0<\epsilon\ll 1$ under a separate nondegeneracy condition. In particular, we find that the $j$th-splitting is of the form $\epsilon^{-\frac{3\kappa}{2}}\exp\left({-\epsilon^{1-\kappa}T^j}\right)(C^j+\mathcal O(\epsilon))$, where $T^j>0$ can be calculated explicitly and be interpreted as the blowup time of special unbounded solutions of the $\epsilon=0$-limiting system in imaginary time $f'=iQ(f)$. Our approach extends a similar geometric method developed by the present author for the generic zero-Hopf bifurcation of co-dimension two, which does not rely on explicit time-parametrizations of the unperturbed heteroclinic connections and their singularities in the complex plane. Instead, we work exclusively in the complexified phase space and relate the exponentially small splitting to the lack of analyticity of center-like invariant manifolds of associated generalized saddle-nodes. %We believe that our approach applies to general analytic unfoldings of zero-Hopf bifurcations in $\mathbb R^3$ with higher co-dimension. We also expect that our approach can be used in related problems, including discrete problems. %In fact, our approach applies to general analytic unfoldings of zero-Hopf bifurcations in $\mathbb R^3$ with higher co-dimension. %We find that the third order system provides a nice forum to present the method.
%  In this paper, we extend the slow divergence-integral from slow-fast systems, due to De Maesschalck, Dumortier and Roussarie, to smooth systems that limit onto piecewise smooth ones as $\epsilon\rightarrow 0$. In slow-fast systems, the slow divergence-integral  is an integral of the divergence along a canard cycle with respect to the slow time and it has proven very useful in obtaining good lower and upper bounds of limit cycles in planar polynomial systems. In this paper, our slow divergence-integral is based upon integration along a generalized canard cycle for a piecewise smooth two-fold bifurcation (of type visible-invisible called $VI_3$). We use this framework to show that the number of limit cycles in regularized piecewise smooth polynomial systems is unbounded.  
% tbd
% \end{abstract}

\noindent \textbf{keywords.} Zero-Hopf bifurcation, blowup, exponentially small splitting, GSPT, generalized saddle-nodes.
 \end{abstract}
 \bigskip
\smallskip

\maketitle

\tableofcontents
% \note{Why not consid¨er general $k+1$ co-dimensional zero-Hopf? Need more conditions on unperturbed manifolds and conditions on $(y,z)$-subsystem that ensure that the hyperbolic points are saddles.}
% \note{Need to think about notation: Should $\psi$ be $X$?}
\section{Introduction}
In this paper, we consider the following class of singularly perturbed systems
\begin{equation}\label{eq:x2y2z2new0}
\begin{aligned}
 { x}'_2  &={Q}(x_2)+\epsilon^{-\kappa} F(\epsilon x_2, \epsilon^{\kappa} y_2, \epsilon^{\kappa}z_2,\epsilon),\\%\epsilon^{\kappa-1} \sum_{n=0}^{\lfloor \frac{\kappa}{2}\rfloor} \frac{1}{2^{2n} n!^2}{Q}^{(2n)}( x_2) ((\epsilon^{\kappa-1} y)^2+(\epsilon^{\kappa-1} z)^2)^n+\epsilon^{2(\kappa-1)}\widetilde R_{2\kappa-1,2},\\
 \epsilon^{\kappa-1} y_2'  &=  z_2 -  \frac12 y_2 \epsilon^{\kappa-1} {Q}'(x_2) +\epsilon^{-\kappa} G(\epsilon x_2, \epsilon^{\kappa} y_2, \epsilon^{\kappa} z_2,\epsilon),\\
 \epsilon^{\kappa-1} z_2'  &= - y_2 -  \frac12 z_2 \epsilon^{\kappa-1} {Q}'(x_2) +\epsilon^{-\kappa} H( x_2, \epsilon^{\kappa} y_2, \epsilon^{\kappa} z_2,\epsilon),
%   \dot{ y}_2  &=  z_2 -{ y_2} \epsilon^{\kappa-1}\sum_{n=0}^{\lfloor \frac{\kappa-1}{2}\rfloor } \frac{1}{2^{2n+1}(n+1)! n!}{Q}^{(2n+1)}((\epsilon^{\kappa-1} y)^2+(\epsilon^{\kappa-1} z)^2)^n+\epsilon^{\kappa-1}\widetilde kjkoiS_{2\kappa-1,2},\\
%   \dot{ z}_2  &= - y_2 -{ z}_2\epsilon^{\kappa-1}\sum_{n=0}^{\lfloor \frac{\kappa-1}{2}\rfloor } \frac{1}{2^{2n+1}(n+1)! n!}{Q}^{(2n+1)}((\epsilon^{\kappa-1} y)^2+(\epsilon^{\kappa-1} z)^2)^n+\epsilon^{\kappa-1} \widetilde T_{2\kappa-1,2},
%   \dot \epsilon &=0.
\end{aligned}
\end{equation}
where $\kappa\in \mathbb N$, $\kappa\ge 2$, is the degree of the polynomial ${Q}$, which we write in the following form:
\begin{align}
 {Q}(f)  = -f^{\kappa}+ \sum_{\alpha=0}^{\kappa-1} a_\alpha f^\alpha.\label{eq:Qkdef}
 \end{align}
 The functions $F,G,H$ are all assumed to be real-analytic functions, satisfying
\begin{align}
 W(r\breve x,r^\kappa \breve y,r^\kappa\breve z,r\breve \epsilon)=\mathcal O(r^{3\kappa-2}),\quad W=F,G,H,\label{eq:Xcond}
\end{align}
as $r\to 0$
uniformly for $(\breve x,\breve y,\breve z,\breve \epsilon)\in \mathbb S^3$, where
$$\mathbb S^3 : = \left\{(\breve x,\breve y,\breve z,\breve \epsilon )\in \mathbb C^4\,:\,\vert \breve x\vert^2+\vert \breve y\vert^2+\vert \breve z\vert^2+\vert \breve \epsilon\vert^2=1\right\}.$$
% Moreover,

 %, see \lemmaref{} below, under the following assumption on ${Q}$:
The system \eqref{x2y2z2new0} is partly motivated by the following singularly perturbed, third order differential equation:
\begin{align}
 \epsilon^{2(\kappa-1)} f''' + f' ={Q}(f),\quad f=f(t),\quad 0<\epsilon\ll 1, \label{eq:3rd}
\end{align}
see Appendix \ref{app:A} for details on bringing \eqref{3rd} into the form \eqref{x2y2z2new0}. Perturbations of \eqref{3rd} by functions of the form
\begin{align}\eqlab{Wpert}
\epsilon^{-\kappa} W(\epsilon f,\epsilon^{\kappa} f',\epsilon^{2\kappa-1}f'',\epsilon),
\end{align}
with $W$ being real-analytic satisfying \eqref{Xcond}, are also allowed, see \remref{apprem}.
% and extensions hereof, 
% where ${Q}$ is a polynomial with $\kappa=\operatorname{deg}Q\in \mathbb N$, $\kappa\ge 2$. By appropriate scalings, it is without loss of generality to write
% \begin{align}
%  {Q}(f)  = -f^{\kappa}+ \sum_{j=0}^{\kappa-1} a_j f^j.\label{eq:Qkdef}
%  \end{align}
If 
\begin{align}
\label{eq:Q2}
{Q}(f) = -f^2+1,
\end{align}
  so that $\kappa=2$, then \eqref{3rd} is known as the Michelsen system; it appears through a travelling wave ansatz of the Kuramoto-Sivashinsky equation, see \cite{michelson1986a}. For $\epsilon=0$, we obtain a reduced first-order system
\begin{align}\label{eq:fred}
 f' ={Q}(f).
\end{align}
If ${Q}$ has $j\in \{2,\ldots, \kappa\}$ real simple roots, then there are $j-1$ many heteroclinic connections. In the case of the Michelsen system, it is well-known (see e.g. \cite{raghavan1997a}) that the single connection does not persist for $\epsilon>0$, but instead there is an exponentially small splitting with respect to $\epsilon\rightarrow 0$:
% \begin{theorem}
Let $f^\pm (\cdot, \epsilon)$, denote the unique solutions of \eqref{3rd} for \eqref{Q2} with $f^\pm (0,\epsilon)=0$ and $f^\pm (t,\epsilon)\rightarrow \pm 1$ as $t\rightarrow \pm \infty$, respectively. Then there is a nonzero constant $C\in \mathbb R \setminus \{0\}$ so that
\begin{align}\label{eq:michaelson}
 \epsilon \frac{\partial^2 f^+}{\partial t^2}(0,\epsilon)-\epsilon \frac{\partial^2 f^-}{\partial t^2}(0,\epsilon) = \epsilon^{-3} \e^{-\frac{\pi}{2\epsilon}}C( 1+o(1)),
\end{align}
cf. \cite{raghavan1997a}, see also \cite{byatt-smith1991a,grimshaw1991a}.
% see e.g. \cite{}.
% \end{theorem}
Due to the symmetry $(f,t)\mapsto (-f,-t)$ of the Michelsen problem, the fact that $$\frac{\partial^2 f^+}{\partial t^2}(0,\epsilon)-\frac{\partial^2 f^-}{\partial t^2}(0,\epsilon)\ne 0\quad \mbox{for $\epsilon>0$},$$ suffices to conclude that the heteroclinic connection for $\epsilon=0$ does not persist (the connection would have to be odd with respect to $t$). %To prove the asymptotic expansion \eqref{eq:michaelson}, the authors in \cite{raghavan1997a} use that the
We will throughout assume the following:
\begin{assumption}
	\assumptionlab{ass1}
	%The vector field $Z : \mathbb R^2 \times \mathbb R \times I \to \mathbb R^2$ satisfies the following constraints:
	${Q}$ is real and has $\kappa$-many simple real roots $\{q^j\}_{j=1}^\kappa$, that we order as follows:
$$q^{\kappa}<q^{\kappa-1}<\cdots<q^{1}.$$
	% 	$q^k<q^{k-1}<\cdots<q^1$ of ${Q}$.
\end{assumption}
Under this assumption, the equation \eqref{fred}, which is also the reduced problem of \eqref{x2y2z2new0}, has $(\kappa-1)$-many heteroclinic connections and in this paper we are concerned with the asymptotic splitting of each of these for all $0<\epsilon\ll 1$.

From the perspective of Geometric Singular Perturbation Theory (GSPT) \cite{fen3,jones_1995}, the splitting problem is complicated by the fact that $(y_2,z_2)=(0,0),x_2\in \mathbb R$ defines a normally elliptic (not hyperbolic) critical manifold for \eqref{x2y2z2new0} in the singular limit $\epsilon\to 0$.

The system \eqref{x2y2z2new0} also relates to the real-analytic normal form of the zero-Hopf bifurcation in $\mathbb R^3$:
% We believe that it is possibly to generalize our results to a normal form of the zero-Hopf bifurcation of higher co-dimension $k$:
\begin{equation}\label{eq:normalformn}
\begin{aligned}
 \dot x &=L_n(x,y^2+z^2,\mu)+\mathcal O(n+1),\\
 \dot y &=K_n(x,y^2+z^2,\mu)y+z+\mathcal O(n+1),\\
 \dot z &=K_n(x,y^2+z^2,\mu)z-y+\mathcal O(n+1),
\end{aligned}
\end{equation}
see e.g. \cite[Chapter 3]{haragus2011a}.
Here $\mu\in \mathbb R^m$, with $\vert\mu \vert\ge 0$ small enough, denote the unfolding parameters. Moreover, $L_n$ and $K_n$, are polynomials of degree $n$ whose $1$-jet vanish, i.e. $W_n(0,0,0)=0$ and $DW_n(0,0,0)=0$ for $W=L,K$.  Finally, $\mathcal O(n+1)$ indicate terms whose $n$-jet with respect to $(x,y,z,\mu)=(0,0,0,0)$ vanish. In the derivation of \eqref{normalformn}, we have divided the right hand side by a nonzero quantity (corresponding to a reparametrization of time).
Now, if we take $n=m=2$ in \eqref{normalformn}, and suppose that
\begin{align*}
\frac{\partial^2 }{\partial x^2}L_2(0,0,0)\ne 0,\quad \operatorname{det} D_{(\mu_1,\mu_2)} \begin{pmatrix}
                           L_2\\
                           K_2
                          \end{pmatrix}(0,0,0)\ne 0,
\end{align*}
then it is elementary to bring the system into the following form
\begin{equation}\label{eq:zeroHopf0}
%  \begin{equation}\eqlab{modelapp1}
\begin{aligned}
 \dot x&=-x^2+\mu_1+a (y^2 +z^2) + F(x,y,z,\mu_1,\mu_2),\\
 \dot y&= z+(\mu_2+b x ) y+G(x,y,z,\mu_1,\mu_2),\\
 \dot z&= -y+(\mu_2+b x) z +H(x,y,z,\mu_1,\mu_2),
\end{aligned}
\end{equation}
with $F,G,H$ each being real-analytic and third order with respect to $(x,y,z,\mu_1,\mu_2)\rightarrow 0$,
% \end{align}
see e.g. \cite[Eq. (1.1)]{bkt}. This is a normal form for the unfolding of the generic co-dimension two zero-Hopf bifurcation. %ere $\mu=(\mu_1,\mu_2)$ denotes the unfolding parameters.
For $(\mu_1,\mu_2)=(0,0)$, we have
\begin{equation}\label{eq:zeroHopf00}
\begin{aligned}
 \dot x&=-x^2+a (y^2 +z^2) + F(x,y,z,0,0),\\
 \dot y&= z+b x  y +G(x,y,z,0,0),\\
 \dot z&= -y+b xz+H(x,y,z,0,0),
\end{aligned}
\end{equation}
with the linearization about the singularity $(x,y,z)=(0,0,0)$ having eigenvalues $0,\pm i$. Let
\begin{align}\label{eq:mu12}
\begin{cases}\mu_1= \epsilon^2, \\
\mu_2=b \sigma\epsilon,
\end{cases}\quad 0\le \epsilon\ll 1,
\end{align}
 and define $(x_2,y_2,z_2)$ by
\begin{align}\label{eq:buhopf}
\begin{cases} x=\epsilon x_2, \\ y=\epsilon^2 y_2, \\ z=\epsilon^2 z_2.
\end{cases}
 \end{align}
 This brings \eqref{zeroHopf0}  into the following slow-fast form
% \begin{align}
 \begin{equation}\label{eq:zeroHopf}
\begin{aligned}
 x_2' &=-x_2^2+1+\epsilon^2 a (y_2^2+z_2^2)+\epsilon F_2(x_2,\epsilon y_2,\epsilon z_2,\epsilon),\\
 \epsilon y_2'&= z_2+\epsilon b(x_2+\sigma) y_2+\epsilon G_2(x_2,\epsilon y_2,\epsilon z_2,\epsilon),\\
 \epsilon z_2'&= -y_2+\epsilon b(x_2+ \sigma) z_2 +\epsilon H_2(x_2,\epsilon y_2,\epsilon z_2,\epsilon),
\end{aligned}
\end{equation}
where $\frac{d}{dt}=()'=\epsilon^{-1} \dot{()}$
and
\begin{align}
 W_2(x_2,\epsilon y_2,\epsilon z_2,\epsilon):=\epsilon^{-3} W(\epsilon x_2,\epsilon (\epsilon y_2),\epsilon (\epsilon z_2),\epsilon^2,\epsilon b\sigma ),\quad W=F,G,H.\label{eq:W2defn}
\end{align}
Here $F_2,G_2,H_2$ all have analytic extensions to $\epsilon=0$.
% \end{align}
For $\epsilon=0$, we therefore obtain the same limiting system as for the Michelsen problem: $$ x_2' = -x_2^2+1,$$ with a heteroclinic connection of the form $x_2(t)=\tanh(t)\rightarrow \pm 1$ as $t\rightarrow \pm \infty$. Notice that \eqref{x2y2z2new0} with $Q(f)=-f^2+1$ and $\kappa=2$ takes the form \eqref{zeroHopf}  with $a=\sigma=0$ and $b=1$.

It is well-known that a similar result to \eqref{michaelson} holds true for \eqref{zeroHopf}:
\begin{theorem}\label{thm:baldom2013a}
 \cite{baldom2013a}
Consider \eqref{zeroHopf} and suppose that
\begin{align}\label{eq:condbc}
 b>0 \quad \mbox{and}\quad -1 <\sigma<1,
\end{align}
both fixed.
Then for all $0<\epsilon\ll 1$ there exist a local one-dimensional unstable manifold $\mathbf{W}_{loc}^u$ for a saddle-focus equilibrium near $(x_2,y_2,z_2)=(-1,0,0)$ and a local one-dimensional stable $\mathbf{W}_{loc}^s$ for a saddle-focus equilibrium near $(x_2,y_2,z_2)=(1,0,0)$. Moreover, $\mathbf{W}_{loc}^u$ and $\mathbf{W}_{loc}^s$ are both transverse to $\{x_2=0\}$ and there exists a constant $C\ge 0$ such that $d(\epsilon):=\operatorname{dist}(\mathbf{W}_{loc}^u\cap \{x_2=0\},\mathbf{W}_{loc}^s\cap \{x_2=0\})$ has the following expansion:
\begin{align}
d(\epsilon) = \epsilon^{-b-2}\e^{-\frac{\pi}{2\epsilon}} (C+\mathcal O(1/\log \epsilon)),\label{eq:dist}
\end{align}
for all $0<\epsilon\ll 1$.
% Here $\operatorname{dist}$ denotes the Hausdorff distance in the $(x_2,y_2,z_2)$-space.
\end{theorem}
Notice in comparison with \cite{baldom2013a}, that their distance is described in terms of the $(x,y,z)$-coordinates, see also \cite{bkt}.

The proof of Theorem \ref{thm:baldom2013a} in \cite{baldom2013a} follows a well-known approach: The invariant manifolds are parameterized by $t\in \mathbb C$, defined through the unperturbed solution $x_2=\tanh(t)$, up close to the poles $t=\pm i\frac{\pi}{2}$. Upon setting $t=\pm i\frac{\pi}{2}+\epsilon t_{2}$, $t_2=\mathcal O(1)$, a so-called inner problem is defined for $\epsilon\to 0$ and it is shown that the stable and unstable manifolds lie close to special invariant manifold solutions of this inner problem. In this way, with the distance between the stable and unstable manifolds known near the poles, the desired separation at $x_2=0$ can be obtained through the solution of a certain linear boundary value problem, obtained by writing an equation for the difference. This functional analytic approach has been successful in many different problems with exponentially small splitting, see e.g. \cite{MR4940205,MR4743478,MR4892796} and references therein.

% The constant $C$ in \eqref{dist} (known as a Stokes constant)  is essentially a measure of the separation of the invariant manifold solutions of the inner problem. Generically, $C\ne 0$ and \eqref{dist} therefore determines an asymptotic formula for the splitting.

In the concurrent paper \cite{bkt}, the present author presents an alternative dynamical systems-based method for proving Theorem \ref{thm:baldom2013a}. In \cite{bkt}, we do not use complex time and the explicit solution of the $\epsilon=0$ limit, but instead work exclusively in the complexified phase space $(x,y,z)\in \mathbb C^3$ and view \eqref{buhopf} as a local version of the blowup transformation
\begin{align}
 r\ge 0,\,(\breve x,\breve y,\breve z,\breve \epsilon )\in \mathbb S^3\mapsto \begin{cases}
                                                x = r\breve x,\\
                                                y = r^2\breve y, \\
                                                z = r^2\breve z,\\
                                                \epsilon =r\breve \epsilon.
                                               \end{cases}\eqlab{blowupthis}
\end{align}
In particular, we use the associated chart $\breve x=1$, with chart-specific coordinates $(r_1,y_1,z_1,\epsilon_1)$, defined by
\begin{align}\eqlab{brevex1}
 \begin{cases}
                                                x = r_1,\\
                                                y = r_1^2y_1, \\
                                                z = r_1^2z_1,\\
                                                \epsilon =r_1\epsilon_1,
                                               \end{cases}
% \end{align*}
\end{align}
to extend the stable and unstable manifolds, which are naturally parameterized in compact subsets of the $(x_2,y_2,z_2)$-space as graphs over $x_2$, to order $x=\mathcal O(1)\in \mathbb C$, $\vert x\vert>0$ sufficiently small. For the extension, we apply the flow in appropriate normal form coordinates, obtained by following the approach in \cite{uldall2024a}, using two separate notions of formal invariant manifolds: (i) formal series with respect to $r_1$ with $\epsilon_1$-dependent coefficients, and (ii) formal series with respect to $\epsilon_1$ with $r_1$-dependent coefficients.
In particular, we use that within $\epsilon_1=0$, we have two invariant manifolds defined over two separate sectors (with nonempty intersection along the imaginary axis) of the complex $x$-plane; these manifolds, which we refer to as the unperturbed manifolds, correspond to the invariant manifolds of the inner problem in \cite{baldom2013a} and define stable and unstable sets of \eqref{zeroHopf00} for $\epsilon=0$. Our results show that the unperturbed manifolds give the leading order approximation of the stable and unstable manifolds for $x=\mathcal O(1)\in \mathbb C$, $\vert x\vert>0$, as $\epsilon\to 0$. Therefore whenever the unperturbed manifolds are distinct on the common domain, we obtain an $\mathcal O(1)$-separation of the stable and unstable manifolds for  $x=\mathcal O(1)\in \mathbb C$, $\vert x\vert>0$, for $0<\epsilon\ll 1$. To determine the separation of the invariant manfiolds at $x_2=0$, we derive a linear equation for the difference (as in \cite{baldom2013a}) for $x$ on the imaginary axis $x\in i \mathbb R$, $\vert x\vert\ge 0$ small enough. As a novel aspect, \cite{bkt} identifies  this problem as a slow-fast system with a normally hyperbolic critical manifold (of saddle-type) for $\epsilon\to 0$, and therefore use GSPT to set up a Fenichel-type normal form, see \cite{fen3,jones_1995}. From this normal form, the difference can easily be integrated from $x=\pm i \delta$, $\delta>0$ fixed small enough, to $x=0$. Here we also use that the equations are real-analytic.

In this paper, we demonstrate the power of the approach in \cite{bkt} by applying it to the system \eqref{x2y2z2new0} under \assumptionref{ass1} and \assumptionref{ass2} (a nondegeneracy condition) below. In this way, we determine an asymptotic expansion for the splitting of the $j$th connection as
\begin{align}\eqlab{dj}
 d^j(\epsilon) = \epsilon^{-\frac{3\kappa}{2}}\e^{-\epsilon^{1-\kappa}T^j } \left(C^j+\mathcal O(\epsilon)\right),\quad T^j= \left| \sum_{l=1}^j \frac{\pi}{Q'(q^l)}\right|>0,
\end{align}
for $j\in \{1,\ldots,\kappa-1\}$. We delay the precise statement to Theorem \ref{thm:main} below.

We believe that our approach also applies to the equations \eqref{normalformn} under the assumption
\begin{align}\label{eq:condL}
\begin{cases}\frac{\partial^\alpha }{\partial x^\alpha}L_\kappa(0,0,0) = 0, &\forall \alpha \in \{0,\ldots,\kappa-1\},\\
 \frac{\partial^\kappa }{\partial x^\kappa}L_\kappa(0,0,0) \ne 0,
\end{cases}
% \quad \mbox{and} \quad ,.
\end{align}
with $n=\kappa\ge 3$.
This corresponds to a zero-Hopf bifurcation of higher co-dimension $\kappa\ge 3$.
However, in order to state our result on \eqref{normalformn} for $n=\kappa$ satisfying \eqref{condL}, we would have to restrict the general unfolding parameters $\mu\in \mathbb R^\kappa$ to an appropriate $\epsilon$-family (similar to \eqref{mu12}) and apply additional assumptions on the system parameters (similar to \eqref{condbc}) in order to have co-existing one-dimensional heteroclinic connections in the singular limit $\epsilon=0$. Since we think of our method (generalizing the approach of \cite{bkt} to higher co-dimension) as our most significant contribution, we have therefore chosen to work with the family \eqref{x2y2z2new0} instead, where our results on the exponentially small splitting can be stated in a more straightforward manner. We believe that the general case is similar, but leave such extensions to the interested reader.

Although our approach follows \cite{bkt}, we believe that the extension to \eqref{x2y2z2new0} with arbitrary $\kappa\in \mathbb N$, $\kappa\ge 2$, is nontrivial. Firstly, we will have to work with invariant manifolds over $2(\kappa-1)$-many sectors. As a consequence, the integration of the differences require extra work. We cannot just perform the integration along the imaginary axis as in \cite{bkt}. Here our main insight is that this integration can be accurately described along $(\kappa-1)$-many  special unbounded solutions of \eqref{fred} in imaginary time $\dot f = i Q(f)$ that correspond to heteroclinic connections (that we refer to as $\mathcal H^j$, $j\in \{1,\ldots,\kappa-1\}$, below) on the associated Poincar\'e sphere. In fact, \textit{$T^j>0$ in \eqref{dj} is the finite blowup time along $\mathcal H^j$ for $f'(s)=iQ(f(s))$, $s\ge 0$, with $f(0)=p^j:=\mathcal H^j\cap \{\operatorname{Im}(f)=0\}$}. For $\kappa=2$ and $Q(f)=-f^2+1$, one easily finds that $\mathcal H^1=i\mathbb R$ and $T^1 = \pi/2$, in agreement with \eqref{michaelson} and Theorem \ref{thm:baldom2013a} (see \remref{rem_michaelson} below). We note that the blowup in imaginary time is intrinsically related to singularities of time parametrizations of connecting orbits, see \cite{fiedler2025a}, but we do not use explicit time-parametrizations in the present manuscript. Instead, our construction is purely based upon phase space methods of dynamical systems theory.

To further emphasize the connection between \eqref{normalformn}, with $n=\kappa$ satisfying \eqref{condL}, and \eqref{x2y2z2new0}, we notice that that the change of coordinates, defined by
\begin{align}\label{eq:xyzx2y2z20}
(x,y,z)\mapsto \begin{cases}
 x_2 =\epsilon^{-1} x,\\
 y_2 = \epsilon^{-\kappa} y,\\
 z_2 =\epsilon^{-\kappa} z,
\end{cases}
 \end{align}
 brings \eqref{x2y2z2new0} into the form \eqref{normalformn}:
 \begin{equation}\label{eq:XYZfirst}
\begin{aligned}
 \dot x&={P}(x,\epsilon)+F(x,y,z,\epsilon),\\
 \dot y&=z-\frac12 P'_x(x,\epsilon) y+G(x,y,z,\epsilon),\\
 \dot z &=-y -\frac12 P'_x(x,\epsilon) z+ H(x,y,z,\epsilon),
\end{aligned}
\end{equation}
with
\begin{align}
 {P} (x,\epsilon) := \epsilon^\kappa {Q}(\epsilon^{-1} x) = -x^\kappa + \sum_{\alpha=0}^{\kappa-1} \epsilon^{\kappa-\alpha} a_\alpha x^\alpha.\label{eq:PkQkfirst}
\end{align}
$L_\kappa=P$ clearly satisfies the conditions in \eqref{condL} for $\epsilon=0$. The system \eqref{XYZfirst} is therefore a one-dimensional ``sub-unfolding'' (through the small parameter $\epsilon$) of a zero-Hopf bifurcation of co-dimension $\kappa$. Moreover, the fact that $K_\kappa=-\frac12 P'_x$ guarantees that there are $(\kappa-1)$-many co-existing one-dimensional heteroclinic connections in the singular limit (for details see Lemma \ref{lem:Wuspj} below).

We finally emphasize that the factor $1/2$ (see \eqref{x2y2z2new0} and \eqref{XYZfirst}) is not important; it can be replaced by any positive number with only minor changes to the document. We leave this to the interested reader. %We believe that this is straightforward and leave the details of this to the interested reader. %In conclusion, we find that \eqref{x2y2z2new0} provides a suitable form to generalize the approach \cite{bkt} to arbitrary co-dimension.

In future work, it is my objective to study complex roots. The belief is that this can be handled by the same approach after some modifications. %We believe that this requires some rethinking.% see \secref{conclusion} for further discussion of this direction.

\subsection{Overview}
The paper is structured as follows. In Section \ref{sec:basics}, we first describe some basic properties of \eqref{x2y2z2new0}, including the existence of $2(\kappa-1)$-many one-dimensional invariant manifolds under the \assumptionref{ass1}. Next, we study the reduced problem \eqref{fred} in real and imaginary time on the Poincar\'e sphere, see Lemmata \ref{lem:realtime} and \ref{lem:imagtime}, and describe invariant manifolds of \eqref{XYZfirst} with $\epsilon=0$ that are graphs over different sectors of the complex $x$-plane, see Lemma \ref{lem:Sj}; we refer to the invariant manifolds of \eqref{XYZfirst} with $\epsilon=0$ as the unperturbed invariant manifolds. These results play an important role for our approach and for the precise statement of our main result, Theorem \ref{thm:main}, in Section \ref{sec:main}. Here Theorem \ref{thm:main} gives the asymptotic expansion (recall \eqref{dj}) of the exponentially small splitting of the invariant manifolds under a separate nondegeneracy condition, see \assumptionref{ass2}. We numerically verify the nondegeneracy condition for the family \eqref{3rd} in Remark \ref{rem:phin}.

The remainder of the paper is then devoted to the proof of Theorem \ref{thm:main}. For this we follow \cite{bkt}, by first describing the stable and unstable manifolds as graphs over large (but fixed) compact sets of the complex $x_2$-plane, see Proposition \ref{prop:Wuspj2} in Section \ref{sec:exit}. Next, in Section \ref{sec:exitt}, we continue these invariant manifolds (by application of the forward and backward flow) to $x=\mathcal O(1)\in \mathbb C$, $\vert x\vert>0$, by working in certain chart-specific coordinates $(r_1,y_1,z_1,\epsilon_1)$ associated with a blow-up transformation (that is induced by \eqref{xyzx2y2z20}). The main result of this section,
Proposition \ref{prop:WlO1}, shows that the stable and unstable manifolds are given by the unperturbed invariant manifolds in Lemma \ref{lem:Sj} over appropriate sectors with $x=\mathcal O(1)\in \mathbb C$, $\vert x\vert>0$, for $\epsilon\to 0$. Unperturbed invariant manifolds that are defined over overlapping domains are in general different (see \assumptionref{ass2} and Lemma \ref{lem:Psil}) and this leads to an $\mathcal O(1)$-separation of the stable and unstable manifolds for $x=\mathcal O(1)\in \mathbb C$, $\vert x\vert>0$. To determine the actual separation for $x\in \mathbb R$, we derive a linear equation for the difference in Section \ref{sec:diff}, which we diagonalize along the special trajectories $\mathcal H^j$ of \eqref{fred} in imaginary time by using classical GSPT. By integrating the resulting equations, we complete the proof in Section \ref{sec:complete}. %We conclude the paper in Section \ref{sec:conclusion} with a discussion section. % completes the proof. % We will assume the following condition throughout:
We illustrate our approach in \figref{blowup}. %Here we use that the system is real-analytic so that $(x,y,z)(t)$ is a solution if only if $(\overline x,\overline y,\overline z)(t)$ is a solution, where the bar denotes the complex conjugate.

 \begin{figure}[h!]
\begin{center}
{\includegraphics[width=.85\textwidth]{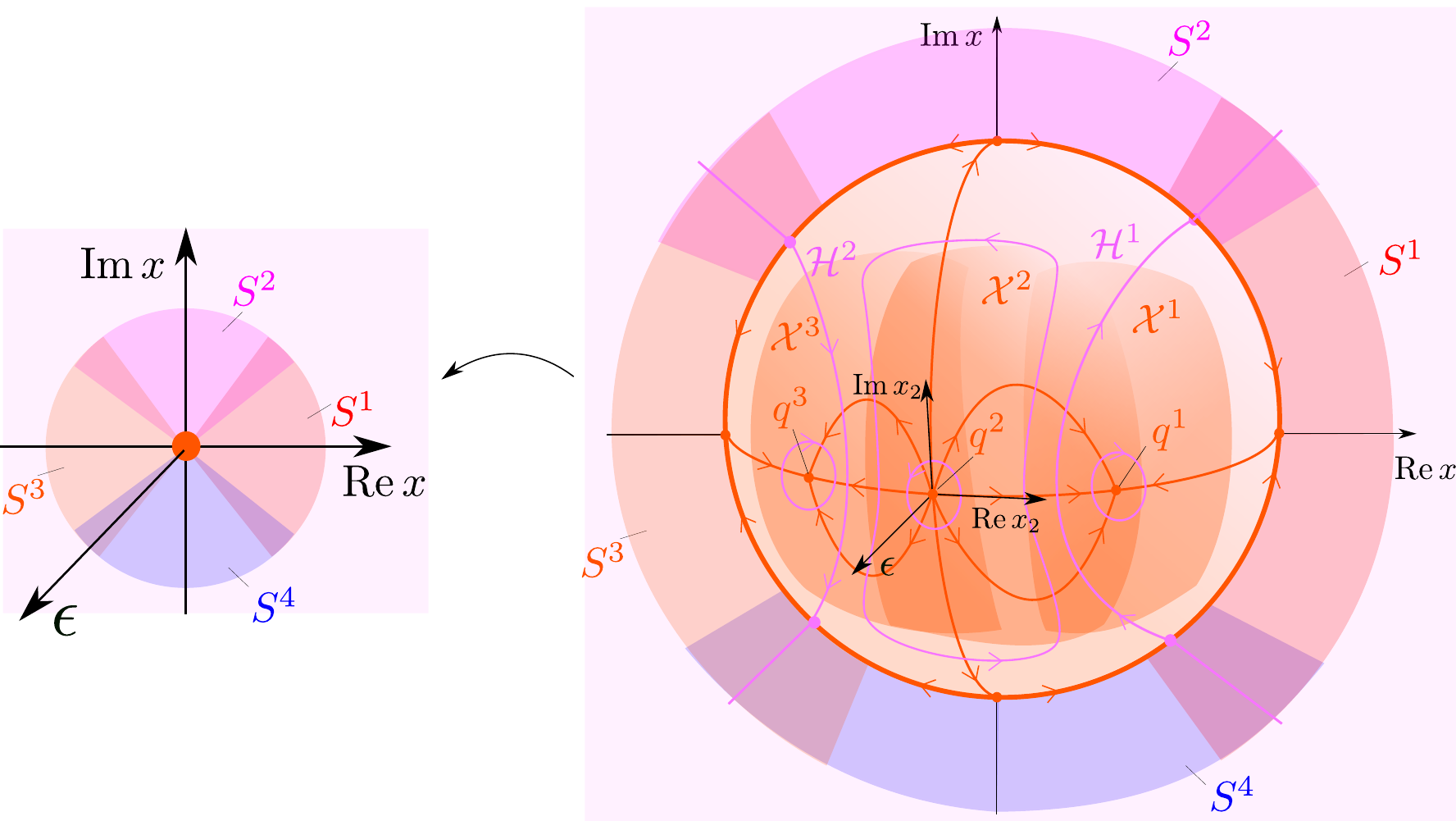}}
% \subfigure[]{\includegraphics[width=.45\textwidth]{psi0.pdf}}
\end{center}
\caption{Illustration of our approach (for $\kappa=3$). We blow up $(x,y,z,\epsilon)=(0,0,0,0)$ to a complex 3-sphere $\mathbb S^3$, see \eqref{bu}, which we here indicate in the $(\operatorname{Re}(x),\operatorname{Im}(x),\epsilon)$-projection. On top of the sphere, we use the coordinates $(x_2,y_2,z_2)$ to describe the existence of the saddle-foci $(x_2,y_2,z_2)=\mathbf{q}^{j}(\epsilon)\to (q^j,0,0)$ as $\epsilon\to 0$, and their one-dimensional invariant manifolds as graphs  $(y_2,z_2)=m_2^{j}(x_2,\epsilon)$ over fixed large compact domains $x_2\in \mathcal X^{j}\subset \mathbb C$ for all $0<\epsilon\ll 1$. Here $j\in \{1,\ldots,\kappa\}$. The orange curves are orbits of $x_2'=Q(x_2)$, $x_2\in \mathbb C$, (on the Poincar\'e sphere,) and $\mathcal X^j$ is in the basin of attraction for $q^j$ for the forward flow for $j$ odd (backward flow for $j$ even, respectively). On the other hand, the purple curves are orbits of $x_2'=iQ(x_2)$, $x_2\in \mathbb C$. A central objective of our approach is to extend the unstable and stable manifolds of $q^{j}$ by working in the coordinates $(r_1,y_1,z_1,\epsilon_1)$, see \eqref{hatx1}, so that they can be compared with the unperturbed invariant manifolds that are defined as graphs over $x\in S^{l}$, $l\in \{1,\ldots,2(\kappa-1)\}$, within $\epsilon=0$. Our results show that if the unperturbed manifolds are non-analytic, then there is an $\mathcal O(1)$-splitting of the unstable and stable manifolds $\mathbf{W}^{u/s}(\mathrm{q}(\epsilon))$ for $x\in S^j\cap S^{j+1}$, $\vert x\vert>0$ small enough, as $\epsilon\to 0$ (for at least one $j\in \{1,\ldots,\kappa-2\}$). Finally, we integrate the differences $(\Delta y_2^j,\Delta z_2^j):=(m_2^{j+1}-m_2^{j})(x,\epsilon)$ along the special trajectories $\mathcal H^j$ (in purple) of $x_2' = iQ(x_2)$ to the point $p^j:=\mathcal H^j \cap \{\operatorname{Im}(x_2)=0\}\in (q^{j+1},q^j)\subset \mathbb R$. In this step, we view the equation for the difference as a slow-fast system with a normally hyperbolic critical manifold, which we then treat by GSPT (through a Fenichel-type normal form). }
\figlab{blowup}
% Remark: If $U$ smooth for $|u|<\nu$, the compact manifolds lie inside $|u|<\nu$ for large $n$
\end{figure}

\subsection{Notation}
We use superscripts ($j,l$) when we refer to the roots of $Q$ and more broadly to objects (including $\mathcal H^j$) relating to \eqref{fred} in real and imaginary time. It should be clear from the context when superscripts ($\alpha,\beta,\gamma,\kappa$) denote powers. Subscripts are on the other hand used to refer to the charts ($\breve x=1$ and $\breve \epsilon=1$ with subscripts $1$ and $2$, respectively, as usual in blowup \cite{krupa_extending_2001}) and for summation indices ($\alpha,\beta,\gamma$). $B_\delta^{\alpha}\subset \mathbb C^\alpha$ (writing $B_\delta=B_\delta^1$ for $\alpha=1$) denotes the open ball of radius $\delta>0$ centered at the origin in $\mathbb C^\alpha$, $\alpha\in \mathbb N$. Finally, $i$ denotes the imaginary unit throughout.

% In this paper, we demonstrate the power of the approach in \cite{bkt} by applying it to the following normal form for \eqref{x2y2z2new0}. Notice that

% We will add the following technical assymption
%
%
%
% In fact, our novel results cover a more broad class of problems (see \eqref{tildexuv}). Indeed, we can write \eqref{3rd} as the first order system
% \begin{align*}
%  ss
% \end{align*}
%  We plan to study complex roots in future work.. We discuss these aspects further in \secref{conclusion}.

\section{Basic properties and statement of the main result}\label{sec:basics}
We  consider \eqref{x2y2z2new0} in the following form
\begin{equation}\label{eq:x2y2z2slow}
\begin{aligned}
 x_2'  &={Q}(x_2)+\epsilon^{2(\kappa-1)} F_2(x_2,y_2,z_2,\epsilon),\\%\epsilon^{\kappa-1} \sum_{n=0}^{\lfloor \frac{\kappa}{2}\rfloor} \frac{1}{2^{2n} n!^2}{Q}^{(2n)}( x_2) ((\epsilon^{\kappa-1} y)^2+(\epsilon^{\kappa-1} z)^2)^n+\epsilon^{2(\kappa-1)}\widetilde R_{2\kappa-1,2},\\
 \epsilon^{\kappa-1}y_2'  &=  z_2 -  \frac12 y_2 \epsilon^{\kappa-1} {Q}'(x_2) +\epsilon^{2(\kappa-1)} G_2(x_2,y_2,z_2,\epsilon),\\
 \epsilon^{\kappa-1}z_2'  &= - y_2 -  \frac12 z_2 \epsilon^{\kappa-1} {Q}'(x_2) +\epsilon^{2(\kappa-1)} H_2(x_2,y_2,z_2,\epsilon),
%   \dot{ y}_2  &=  z_2 -{ y_2} \epsilon^{\kappa-1}\sum_{n=0}^{\lfloor \frac{\kappa-1}{2}\rfloor } \frac{1}{2^{2n+1}(n+1)! n!}{Q}^{(2n+1)}((\epsilon^{\kappa-1} y)^2+(\epsilon^{\kappa-1} z)^2)^n+\epsilon^{\kappa-1}\widetilde S_{2\kappa-1,2},\\
%   \dot{ z}_2  &= - y_2 -{ z}_2\epsilon^{\kappa-1}\sum_{n=0}^{\lfloor \frac{\kappa-1}{2}\rfloor } \frac{1}{2^{2n+1}(n+1)! n!}{Q}^{(2n+1)}((\epsilon^{\kappa-1} y)^2+(\epsilon^{\kappa-1} z)^2)^n+\epsilon^{\kappa-1} \widetilde T_{2\kappa-1,2},
%   \dot \epsilon &=0.
\end{aligned}
\end{equation}
with $()'=\frac{d}{dt}$ and where
\begin{align*}
 W_2(x_2,y_2,z_2,\epsilon) :=\epsilon^{-3\kappa+2}W(\epsilon x_2,\epsilon^\kappa y_2,\epsilon^\kappa z_2,\epsilon),\quad W=F,G,H.
\end{align*}
Notice that $F_2,G_2$ and $H_2$ are well-defined and analytic for $\epsilon=0$ by \eqref{Xcond}.
\eqref{x2y2z2fast} is a slow-fast system, with the time $t$ being the slow time. By introducing the fast time $\tau =\epsilon^{1-\kappa} t$, we obtain the fast formulation
\begin{equation}\label{eq:x2y2z2fast}
\begin{aligned}
 \dot{ x}_2  &=\epsilon^{\kappa-1}\left({Q}(x_2)+\epsilon^{2(\kappa-1)} F_2(x_2,y_2,z_2,\epsilon)\right),\\%\epsilon^{\kappa-1} \sum_{n=0}^{\lfloor \frac{\kappa}{2}\rfloor} \frac{1}{2^{2n} n!^2}{Q}^{(2n)}( x_2) ((\epsilon^{\kappa-1} y)^2+(\epsilon^{\kappa-1} z)^2)^n+\epsilon^{2(\kappa-1)}\widetilde R_{2\kappa-1,2},\\
 \dot y_2  &=  z_2 -  \frac12 y_2 \epsilon^{\kappa-1} {Q}'(x_2) +\epsilon^{2(\kappa-1)} G_2(x_2,y_2,z_2,\epsilon),\\
 \dot z_2  &= - y_2 -  \frac12 z_2 \epsilon^{\kappa-1} {Q}'(x_2) +\epsilon^{2(\kappa-1)} H_2(x_2,y_2,z_2,\epsilon),
%   \dot{ y}_2  &=  z_2 -{ y_2} \epsilon^{\kappa-1}\sum_{n=0}^{\lfloor \frac{\kappa-1}{2}\rfloor } \frac{1}{2^{2n+1}(n+1)! n!}{Q}^{(2n+1)}((\epsilon^{\kappa-1} y)^2+(\epsilon^{\kappa-1} z)^2)^n+\epsilon^{\kappa-1}\widetilde S_{2\kappa-1,2},\\
%   \dot{ z}_2  &= - y_2 -{ z}_2\epsilon^{\kappa-1}\sum_{n=0}^{\lfloor \frac{\kappa-1}{2}\rfloor } \frac{1}{2^{2n+1}(n+1)! n!}{Q}^{(2n+1)}((\epsilon^{\kappa-1} y)^2+(\epsilon^{\kappa-1} z)^2)^n+\epsilon^{\kappa-1} \widetilde T_{2\kappa-1,2},
%   \dot \epsilon &=0.
\end{aligned}
\end{equation}
with $\dot{()}=\frac{d}{d\tau}$.
% with respect to a fast time defined by $\dot{()}=\epsilon^{\kappa-1}()'$, and
% where $\kappa\in \mathbb N$, $\kappa\ge 2$, is the degree of the polynomial ${Q}$, which we write in the following form:
% \begin{align}
%  {Q}(f)  = -f^{\kappa}+ \sum_{j=0}^{\kappa-1} a_j f^j.\label{eq:Qkdef}
%  \end{align}
%  The functions $R,S,T$ are all assumed to be analytic functions satisfying
% \begin{align}
%  X(r\breve x,r^\kappa \breve y,r^\kappa\breve z,r\breve \epsilon)=\mathcal O(r^{3\kappa-2})\quad X=R,S,T.\eqlab{Xcond}
% \end{align}
Setting $\epsilon=0$ in \eqref{x2y2z2fast} gives the following layer problem:
\begin{equation}\label{eq:layer}
\begin{aligned}
 \dot x_2 &=0,\\
 \dot y_2 &=z_2,\\
 \dot z_2 &= -y_2,
\end{aligned}
\end{equation}
which has a normally elliptic critical manifold $S$ defined by $(y_2,z_2)=(0,0)$; the linearization having eigenvalues $0,\pm i$. On the other hand, by
setting $\epsilon=0$ in \eqref{x2y2z2slow} we obtain the reduced problem:
\begin{equation}\label{eq:reduced}
\begin{aligned}
 x_2'&={Q}(x_2),
%  0&=z_2,\\
%  0 &=-y_2.
\end{aligned}
\end{equation}
on $S$.
% on $(y_2,z_2)=(0,0)$.
The dynamics of \eqref{layer} and \eqref{reduced} (in green and blue, respectively) are illustrated in \figref{x2y2z2} for $\kappa=4$ (under \assumptionref{ass1}).

\begin{lemma}\label{lem:Wuspj}
 Let $q\in \mathbb R$ denote a simple real root of ${Q}$, recall \eqref{Qkdef}. Then there is an analytic function $\mathbf{q}:[0,\epsilon_0)\rightarrow \mathbb R^3$:
 \begin{align*}
 \mathbf{q}(\epsilon)=(q,0,0)+\mathcal O(\epsilon^{2(\kappa-1)}),
 \end{align*}
 so that $\mathbf{q}(\epsilon)$ is a hyperbolic saddle-focus of \eqref{x2y2z2fast} for all $\epsilon\in (0,\epsilon_0)$, $0<\epsilon_0\ll 1$, with the following eigenvalues of the linearization
 \begin{align*}
 \begin{cases}
  \lambda_1 = \epsilon^{\kappa-1} {Q}'(q)+\mathcal O(\epsilon^{2(\kappa-1)}),\\
  \lambda_{2} = i - \frac12 \epsilon^{\kappa-1} {Q}'(q)+\mathcal O(\epsilon^{2(\kappa-1)}),\\
  \lambda_3 =\overline \lambda_2.
  \end{cases}
 \end{align*}
In particular, if ${Q}'(q)<0$ then $\mathbf{W}^s(\mathbf{q}(\epsilon))$ is one-dimensional whereas $\mathbf{W}^u(\mathbf{q}(\epsilon))$ is two-dimensional for all $\epsilon\in (0,\epsilon_0)$. If ${Q}'(q)>0$ then it is the other way around: $\mathbf{W}^s(\mathbf{q}(\epsilon))$ is two-dimensional whereas $\mathbf{W}^u(\mathbf{q}(\epsilon))$ is one-dimensional for all $\epsilon\in [0,\epsilon_0)$.
\end{lemma}
\begin{proof}
 The existence of $\mathbf{q}(\epsilon)$ is an easy consequence of the implicit function theorem. To see that the equilibrium is a hyperbolic saddle-focus, we notice that $\operatorname{Re}(\lambda_{2,3}) \operatorname{Re}(\lambda_1)<0$  for all $0<\epsilon\ll 1$ under the assumption that $p$ is a simple real root ${Q}'(p)\ne 0$.
\end{proof}
Under \assumptionref{ass1}, we obtain $\kappa$-many hyperbolic saddle-foci $\mathbf{q}^j(\epsilon)$, $j\in \{1,\ldots,\kappa\}$, with $\mathbf{q}^j(0)=(q^j,0,0)$, each having a one-dimensional invariant manifold
\begin{align*}
\mathbf{W}^{\sigma^j}(\mathbf{q}^j(\epsilon)) \quad \mbox{with}\quad  \begin{cases}
  \sigma^j = s & \mbox{ if $j$ is odd},\\
  \sigma^j = u& \mbox{ if $j$ is even}, \end{cases}
\end{align*}
for all $0<\epsilon \ll 1$, see \figref{x2y2z2} (in red). We are interested in asymptotic formulas for the $(\kappa-1)$-many separations of the stable and unstable invariant manifolds. The separation will be measured within planes $\{x_2=p^j\}$ for some $p^j\in (q^{j+1},q^j)$, $j\in \{1,\ldots,\kappa-1\}$, to be specified later (see Lemma \ref{lem:imagtime}). However, in order to determine the real separation it is crucial to extend \eqref{x2y2z2fast} to complex variables $(x_2,y_2,z_2)\in \mathbb C^3$, see e.g. \cite{baldom2013a,bkt}.

\begin{figure*}[t!]
		\centering
		{\includegraphics[width=0.7\textwidth]{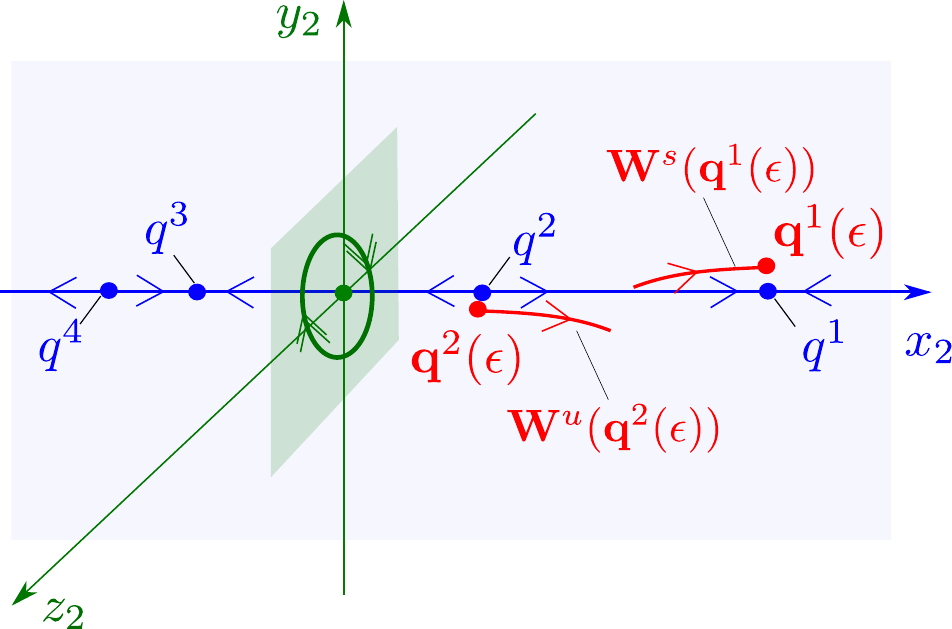}}
				\caption{The slow-fast dynamics of \eqref{x2y2z2fast} for $\kappa=4$. The layer problem \eqref{layer} is in green whereas the reduced problem \eqref{reduced} is in blue. The $x_2$-axis is the critical manifold, which is normally elliptic. In red we show stable and unstable manifolds of the perturbed points $\mathbf{q}^1(\epsilon)$ and $\mathbf{q}^2(\epsilon)$ for $0<\epsilon\ll 1$.  We are concerned with the asymptotic splitting of $\mathbf{W}^{\sigma^j}(\mathbf{q}^j(\epsilon))$ and $\mathbf{W}^{\sigma^{j+1}}(\mathbf{q}^{j+1}(\epsilon))$ within some fixed section $\{x_2=p^j\}$ with $p^j\in (q^{j+1},q^j)$. }
		\label{fig:x2y2z2}
	\end{figure*}
	
% % \subsection{A normal form}
% 
% 
% 
% 
% 
\subsection{The reduced problem in real and imaginary time}
In this section, we consider
\begin{align}\label{eq:realtime}
 \frac{dx_2}{ds} &={Q}(x_2),
\end{align}
and
\begin{align}\label{eq:imagtime}
 \frac{dx_2}{ds} &=i {Q}(x_2).
\end{align}
% with $\dot{()}=\frac{d}{ds}$ with $s\in \mathbb R$.
These two systems can be understood as the reduced problem \eqref{reduced}:
\begin{align*}
 \frac{dx_2}{dt} ={Q}(x_2),
\end{align*}
in real time $t=s\in \mathbb R$ and imaginary time $t=is\in i \mathbb R$, respectively.

We will study \eqref{realtime} and \eqref{imagtime} using Poincar\'e compactification, see e.g. \cite{dumortier2006a}. The compactification is based upon
\begin{align*}
 \begin{cases}\operatorname{Re}(x_2) = {\breve u}{\breve w}^{-1},\\
 \operatorname{Im}(x_2) = {\breve v}{\breve w}^{-1},
 \end{cases}
\end{align*}
with $(\breve u,\breve v,\breve w)\in \mathbb H^2$. Here $\mathbb H^2$ denotes the hemisphere:
\begin{align*}
 \mathbb H^2:=\{(\breve u,\breve v,\breve w)\,:\,\breve u^2+\breve v^2+\breve w^2=1,\,\breve w\ge 0\}.
\end{align*}
The equator circle $\mathbb S^1 = \mathbb H^2\cap \{\breve w=0\}$ represents $x_2=\infty$.
There is a one-to-one correspondence between $x_2\in \mathbb C$ and $(\breve u,\breve v,\breve w)\in \operatorname{int}\mathbb H^2:=\mathbb H^2\cap \{\breve w>0\}$.
Any object $\breve \gamma \in \operatorname{int}\mathbb H^2$ will be denoted by $\gamma$ when it is understood as a subset of the complex $x_2$-plane (and vice versa). For example, the roots $q^j$ of ${Q}$ are denoted by $\breve q^j\in \operatorname{int}\mathbb H^2$ when viewed on the hemisphere.
% Therefore a trajectory $\gamma \subset \mathbb C$ of \eqref{realtime} defines a subset $\breve \gamma \subset \operatorname{int}\mathbb H$ under the Poincar\'e compactification.
On the other hand, in a neighborhood of $\mathbb S^1$, we use the coordinates $(w_1,\theta)$ defined by
\begin{align}\label{eq:x2w1}
x_2=w_1^{-1} \e^{i\theta}, \quad w_1\ge 0, \quad \theta \in \mathbb T^1,
\end{align}
with $\mathbb T^1=\mathbb R/(2\pi\mathbb Z)$. For the flow to be well-defined in the $(w_1,\theta)$-coordinates, we use a desingularization corresponding to division of the right hand side by $w_1^{1-\kappa}=\vert x_2\vert^{\kappa-1}$. This gives a complete flow on $\mathbb H^2$. In particular, \eqref{realtime} and \eqref{imagtime} are both complete with respect to $s_{P}$ ($P$ for Poincar\'e) defined by
\begin{align*}
ds_P = (1+\vert x_2\vert^{\kappa-1})ds.
\end{align*}
%In the $(w_1,\theta)$-coordinates, any object $\breve \gamma$ will be denote by $\gamma_1$.  %For a neighborhood of $\mathbb S$, we use the $(w_1,\theta)$-coordinates.
{We suppose throughout that \assumptionref{ass1} holds}. %The roots .

\begin{remark}
 The recent reference \cite{fiedler2025a} studies general scalar polynomial systems (without \assumptionref{ass1}) in complex time. Part of our approach to  study \eqref{realtime} and \eqref{imagtime} is inspired by this reference.
\end{remark}

\begin{lemma}\label{lem:realtime}
 Consider \eqref{realtime}. Then the following holds.
 \begin{enumerate}
 \item \label{it:real1} Conjugation $x\mapsto \overline x$ defines a symmetry of \eqref{realtime}.
 \item \label{it:real2} $q^j$ with $j$ odd (even) is a hyperbolic stable (unstable, respectively) node of \eqref{realtime} for all $j\in \{1,\ldots,\kappa\}$.
 \item \label{it:real4}
 There are no cycles of \eqref{realtime}.
  \item \label{it:real3} There are $2(\kappa-1)$-many hyperbolic saddles $\breve e^l\in \mathbb S^1$, $l\in \{1,\ldots,2(\kappa-1)\}$, of the Poincar\'e compactification of \eqref{realtime} on $\mathbb H$ given by
 \begin{align*}
   \breve e^l:\quad w_1= 0,\quad \theta=\frac{\pi}{\kappa-1}(l-1),
 \end{align*}
 in the $(w_1,\theta)$-coordinates.
 \item \label{it:real5} The stable (unstable) manifold of $\breve e^{l}$ with $l$ odd (even, respectively) is a subset of $\mathbb S^1$.
 \item \label{it:real6} There are $2(\kappa-1)$-many heteroclinic connections $\breve{\mathcal E}^l\subset \operatorname{int}\mathbb H^2$, $l\in \{1,\ldots,2(\kappa-1)\}$, given by the regular trajectories
 \begin{align}\label{eq:Wgl}
  \mathbf{W}^{\chi^l}(\breve e^l) \setminus  \{\breve e^l \}\quad \mbox{where}\quad \begin{cases} \chi^l=u&\mbox{ if $l$ is odd},\\
   \chi^l=s &\mbox{ if $l$ is even}.
  \end{cases}
 \end{align}
that connect $\breve e^l$ with
 \begin{align*}
  \begin{cases}
    \breve q^{l}& \mbox{if $l\le \kappa$}\\
    \breve q^{2\kappa-l} & \mbox{if $\kappa<l\le 2(\kappa-1)$},
  \end{cases}
   \end{align*}
   in the sense of $\alpha$- and $\omega$-limit sets; here $\breve q^j\in \operatorname{int}\mathbb H^2$, $j\in \{1,\ldots,\kappa\}$, denote the roots of $Q$ on the Poincar\'e hemisphere $\mathbb H^2$.
 \item \label{it:real8} In a neighborhood of $\breve e^l$, $\mathbf{W}_{loc}^{\chi^l}(\breve e^l)$, with $\chi^l$ as in \eqref{Wgl}, takes the graph form:
 \begin{align}
  \theta = E^l(w_1),\quad w_1\in [0,\xi],\label{eq:localWgl}
 \end{align}
 in the $(w_1,\theta)$-coordinates, see \eqref{x2w1},
with $\xi>0$ small enough and where $ E^l$ is real-analytic with $E^l(0)=\frac{\pi}{\kappa-1}(l-1)$. Moreover, in such a neighborhood there is a $w_1$-fibered $C^\infty$-diffeomorphism $(w_1,\theta)\mapsto (w_1,\varphi)$ so that
\begin{equation}\label{eq:w1normalform}
\begin{aligned}
 \frac{dw_1}{ds_1} &= (-1)^{l-1} w_1^{2-\kappa},\\
 \frac{d \varphi }{ds_1} &= w_1^{1-\kappa} ( (-1)^l (\kappa-1) +G^l(w_1^{\kappa-1} \varphi))\varphi,
\end{aligned}
\end{equation}
where $G^l(0)=0$ and
\begin{align}
 \frac{ds_1}{ds} = 1+\mathcal O(w_1).\label{eq:s1time}
\end{align}
 \item \label{it:real7} $\mathcal E^1,\mathcal E^\kappa\subset \mathbb R$ and $\overline{\mathcal E^{2\kappa-l}}={\mathcal E^{l}}$ for all $l\in \{2,\ldots \kappa-1\}$.
  \end{enumerate}
%  Here
% Moreover, there are $\kappa$ heteroclinic connections $\mathcal E^j\subset \operatorname{int}(\mathbb D)$, $j\in \{1,\ldots,\kappa\}$, between $(0,E_j)$ and $(0,E_{2(\kappa-1)-j})$.
\end{lemma}
\begin{proof}
 Assertion \ref{it:real1} is obvious since ${Q}$ is a real polynomial. Assertion \ref{it:real2} follows from the fact that
 \begin{align*}
  {Q}'(q^j) {Q}'(q^{j+1})<0,
 \end{align*}
 for all $j\in \{1,\ldots,\kappa-1\}$,
with ${Q}'(q^1)<0$, which is a consequence of the roots being simple roots. Assertion \ref{it:real4}, follows from the fact that all equilibria in $\operatorname{int}\mathbb H^2$ are contained within the invariant set defined by $x_2\in \mathbb R$ (cf. \assumptionref{ass1}). For assertion \ref{it:real3}, we use the $(w_1,\theta)$-coordinates:
\begin{equation}\label{eq:w1theta}
\begin{aligned}
 \dot w_1 &=-w_1  \left(-\cos((\kappa-1)\theta)+\sum_{\alpha=0}^{\kappa-1} a_\alpha w_1^{\kappa-\alpha}\cos ((\alpha-1) \theta)\right),\\
 \dot \theta &= -\sin((\kappa-1)\theta)+\sum_{\alpha=0}^{\kappa-1} a_\alpha w_1^{\kappa-\alpha}\sin ((\alpha-1) \theta).
\end{aligned}
\end{equation}
Here we have, as advertised above, used the desingularization corresponding to division of the right hand side by $w_1^{1-\kappa}$. Now, $w_1=0$ corresponds to $\mathbb S^1$ and here we find
\begin{align*}
 \dot \theta = \sin((\kappa-1)\theta),
\end{align*}
with hyperbolic equilibria $\theta = \frac{\pi}{\kappa-1}(l-1)$, $l\in \{1,\ldots,2(\kappa-1)\}$. This defines $\breve e^l$. Linearization gives the eigenvalues
\begin{align*}
(-1)^{l-1},\quad (\kappa-1) (-1)^{l},
\end{align*}
with associated eigenvectors $(1,0)^{\operatorname{T}}$ and $(0,1)^{\operatorname{T}}$, respectively. This also proves assertion \ref{it:real5} and the expansion \eqref{localWgl} of $\mathbf{W}_{loc}^{\chi^l}(\breve e^l)$ in assertion \ref{it:real8}.

Assertion \ref{it:real6} follows from a simple phase portrait analysis based upon Poincar\'e-Bendixson. Indeed, $l=1$ and $l=\kappa$ are direct consequences of $x\in \mathbb R$ being an invariant set. The cases $2<l<\kappa-1$ and $\kappa+1<l<2(\kappa-1)$ are related by complex conjugation within $\operatorname{Im}x\gtrless 0$, respectively. The fact that there are no cycles and no equilibria for $x\notin \mathbb R$ (cf. \assumptionref{ass1}) gives the result.

The normal form \eqref{w1normalform} follows from applying \cite[Theorem 2.15]{dumortier2006a} to the system \eqref{w1theta} after dividing the right hand side by
\begin{align*}
 \frac{ds_1}{ds}=(-1)^l \left(-\cos((\kappa-1)\theta)+\sum_{\alpha=0}^{\kappa-1} a_\alpha w_1^{\kappa-\alpha}\cos ((\alpha-1) \theta)\right).
\end{align*}
Finally, assertion \ref{it:real7} follows from the symmetry in assertion \ref{it:real1}.
\end{proof}

We illustrate the phase portraits of \eqref{realtime} in Fig. \ref{fig:realtime} on the Poincar\'e hemisphere (illustrated as a disc) for $\kappa=2$ and $\kappa=3$ under \assumptionref{ass1}.

	\begin{figure*}[t!]
		\centering
		\subfigure[]{\includegraphics[width=0.47\textwidth]{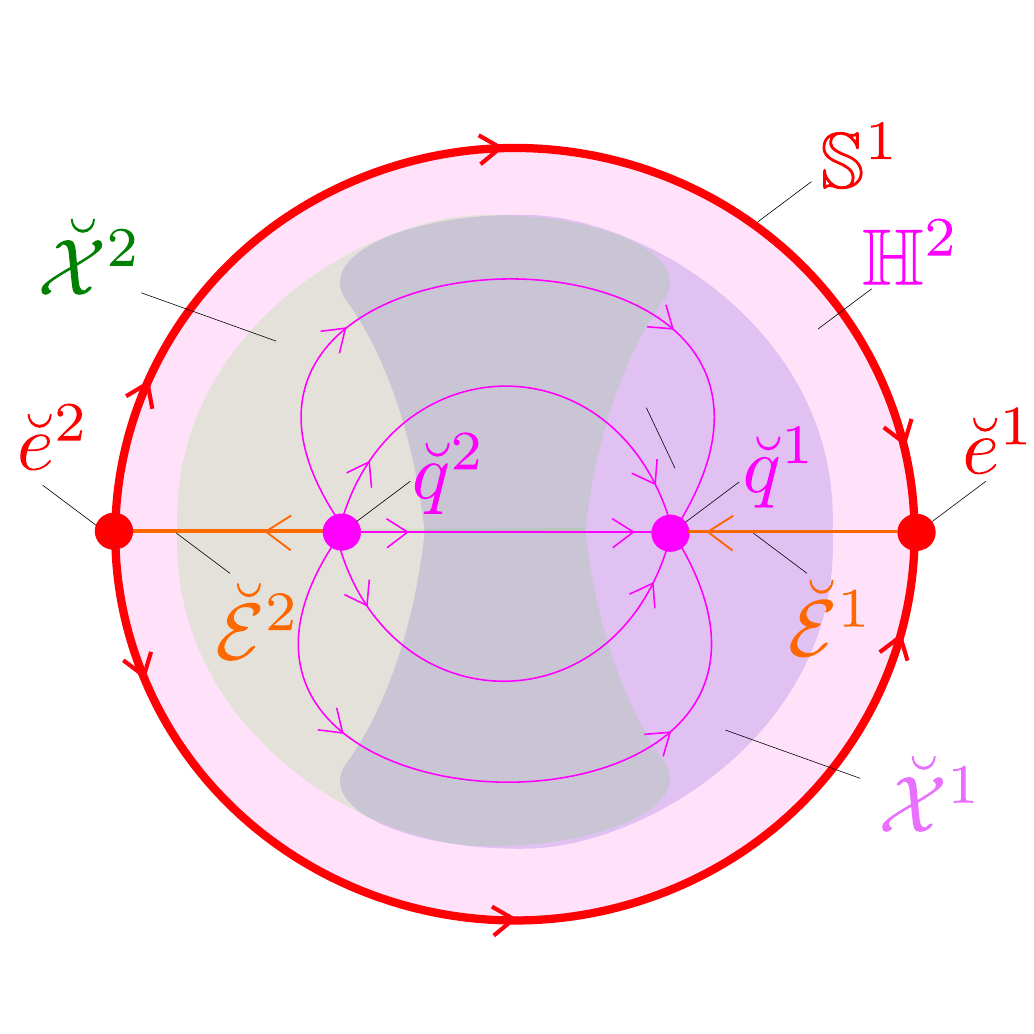}}
		\subfigure[]{\includegraphics[width=0.47\textwidth]{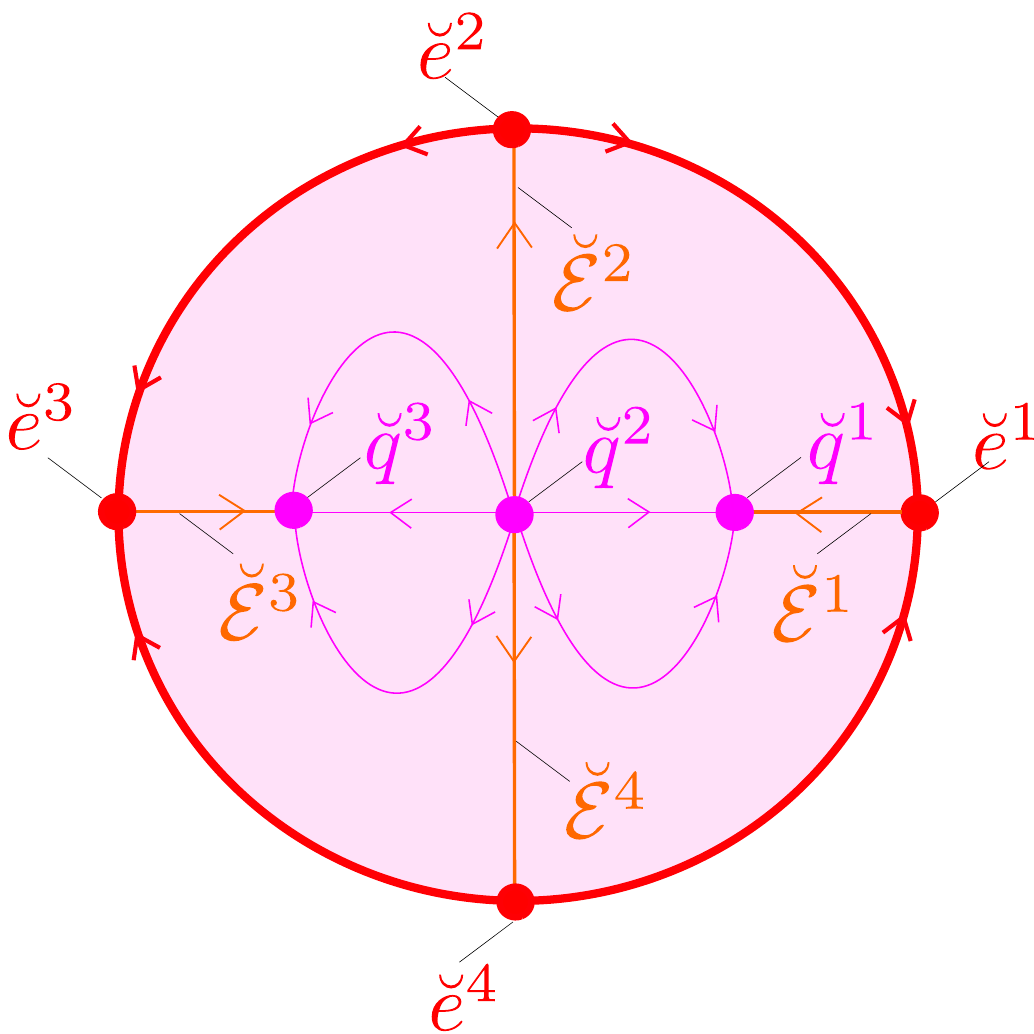}}
		\caption{Phase portraits of $x_2'=Q(x_2)$ on the Poincar\'e hemisphere $\mathbb H^2$ for $\kappa=2$ (a) and $\kappa=3$ (b). The points $q^3<q^2<q^1$ denote the roots of $Q$.  The sets $\breve {\mathcal X}^j$ (green and purple) in (a) are compact neighborhoods of $\breve q^j$ within the basins of attraction for the forward/backward flow of $x_2'=Q(x_2)$. These sets are relevant for Proposition \ref{prop:Wuspj2}. }
		\label{fig:realtime}
	\end{figure*}
\begin{lemma}\label{lem:imagtime}
 Consider \eqref{imagtime}. Then the following holds true.
 \begin{enumerate}
  \item \label{it:imag1} Conjugation and reversing time $(x,s)\mapsto (\overline x,-s)$ defines a time-reversal symmetry of \eqref{imagtime}.
 \item \label{it:imag2} All $q^j$, $j\in \{1,\ldots,\kappa\}$, are centers of \eqref{imagtime} and there are no limit cycles.
 \item \label{it:imag3} There are $2(\kappa-1)$-many hyperbolic saddles $\breve h^l\in \mathbb S^1$, $l\in \{1,\ldots,2(\kappa-1)\}$,  of the Poincar\'e compactification of \eqref{imagtime} on $\mathbb H$ defined by
 \begin{align*}
   \breve h^l:\quad w_1= 0,\quad \theta=\frac{\pi}{\kappa-1}(l-1)+\frac{\pi}{2(\kappa-1)},
 \end{align*}
 in the $(w_1,\theta)$-coordinates.
 \item \label{it:imag4}
 The stable (unstable) manifold $\mathbf{W}^s$ ($\mathbf{W}^u$) of $\breve h^{l}$ with $l$ even (odd, respectively) is a subset of $\mathbb S^1$.
 \item \label{it:imag5} There are $(\kappa-1)$-many heteroclinic connections $\breve{\mathcal H}^l\subset \operatorname{int}\mathbb H^2$, given by the regular trajectories
 \begin{align}\label{eq:Wgj}
 \mathbf{W}^{\chi^l}(\breve h^l)\setminus \{\breve h^l\} \quad \mbox{where}\quad
  \begin{cases}
   \chi^l = s & \mbox{ if $l$ is odd},\\
\chi^l = u & \mbox{ if $l$ is even},
  \end{cases}
 \end{align} that connect $\breve h^l\in \mathbb S^1$ with
 \begin{align}\label{eq:conjhl}
 \overline{\breve h^l}=\breve h^{2\kappa-1-l}\in \mathbb S^1,\quad \forall\,l\in \{1,\ldots,\kappa-1\},
 \end{align}
 in the sense of $\alpha$- and $\omega$-limit sets.
 \item \label{it:imag8}  In a neighborhood of $\breve h^l$, the invariant manifold $\mathbf{W}^{\chi^l}(\breve h^l)$, with $\chi^l$ as in \eqref{Wgj}, takes the graph form:
 \begin{align*}
  \theta =  H^l(w_1),\quad w_1\in [0,\xi],
 \end{align*}
 in the $(w_1,\theta)$-coordinates, see \eqref{x2w1},
with $\xi>0$ small enough and $H^l$ real-analytic with $H^l(0) = \frac{\pi}{\kappa-1}(l-1)+\frac{\pi}{2(\kappa-1)}$. Moreover, in such a neighborhood there is a $w_1$-fibered $C^\infty$ diffeomorphism $(w_1,\theta)\mapsto (w_1,\varphi)$ so that
\begin{equation}\label{eq:w1normalform2}
\begin{aligned}
 \frac{dw_1}{ds_1} &= (-1)^{l} w_1^{2-\kappa},\\
 \frac{d \varphi }{ds_1} &= w_1^{1-\kappa} ( (-1)^{l-1} (\kappa-1) +G^l(w_1^{\kappa-1} \varphi))\varphi,
\end{aligned}
\end{equation}
where $G^l(0)=0$ and
\begin{align*}
 \frac{ds_1}{ds} = 1+\mathcal O(w_1).
\end{align*}
 \item \label{it:imag6} $\overline{\mathcal H^j}=\mathcal H^j$ for all $j\in \{1,\ldots,\kappa-1\}$ and  the points $p^l :=\mathcal H^l\cap \{\operatorname{Im}(x_2)=0\}$, $l\in \{1,\ldots,\kappa-1\}$, satisfy $p^l \in (q^{l+1},q^l)$.
 %  On the other hand, the unstable (stable) manifold of $H^j$ with $j$ odd (even, respectively) is asymptotic as $t\rightarrow +\infty $ ($-\infty$, respectively) to
%  \begin{align*}
%   \begin{cases}
%     x_2 = e_{j}& \mbox{if $j\le \kappa$}\\
%     x_2 =e_{2\kappa-j} & \mbox{if $\kappa<j\le 2(\kappa-1)$}.
%   \end{cases}
%  \end{align*}
\end{enumerate}
\end{lemma}
\begin{proof}
 Assertions \ref{it:imag1} and \ref{it:imag2} are simple calculations. For assertion \ref{it:imag3}, we use the $(w_1,\theta)$-coordinates:
\begin{align*}
 \dot w_1 &=w_1 \left(-\sin((\kappa-1)\theta)+\sum_{\alpha=0}^{\kappa-1} a_\alpha w_1^{\kappa-\alpha}\sin ((\alpha-1) \theta)\right),\\
 \dot \theta &= -\cos((\kappa-1)\theta)+\sum_{\alpha=0}^{\kappa-1} a_\alpha w_1^{\kappa-\alpha}\cos ((\alpha-1) \theta).
\end{align*}
Here we have, as advertised above, used the desingularization corresponding to division of the right hand side by $w_1^{1-\kappa}$.
Now, $w_1=0$ corresponds to $\mathbb S^1$ and here we find
\begin{align*}
 \dot \theta = -\cos((\kappa-1)\theta),
\end{align*}
with hyperbolic equilibria $\theta = \frac{\pi}{\kappa-1}(l-1)+\frac{\pi}{2(\kappa-1)}$, $l\in \{1,\ldots,2(\kappa-1)\}$. This defines $\breve h^l$. Linearization gives the eigenvalues
\begin{align*}
(-1)^{l},\quad (\kappa-1) (-1)^{l-1},
\end{align*}
with associated eigenvectors $(1,0)^{\operatorname{T}}$ and $(0,1)^{\operatorname{T}}$, respectively. This also proves the first part of assertion \ref{it:imag8}; the proof of the latter part is identical to the proof of \eqref{w1normalform} above. Next, assertions \ref{it:imag5} and \ref{it:imag6} follow from the time-reversal symmetry, see assertion \ref{it:imag1}.
\end{proof}

We illustrate the phase portraits of \eqref{imagtime} in Fig. \ref{fig:imagtime} on the Poincar\'e hemisphere (illustrated as a disc) for $\kappa=2$ and $\kappa=3$ under \assumptionref{ass1}.

As corollaries of \eqref{w1normalform}  and \eqref{w1normalform2}, we have the following:
\begin{cor}\label{cor:blowuptime}
\eqref{realtime} and \eqref{imagtime} has finite time blowup along the invariant manifolds \eqref{Wgl} and \eqref{Wgj}, respectively.
\end{cor}
\begin{proof}
 We focus on \eqref{realtime} and \eqref{w1normalform}. Notice in particular that \eqref{realtime} has finite time blow-up along \eqref{Wgl} with respect to $s$ if and only if \eqref{w1normalform} has a finite time singularity $w_1(s_1)\rightarrow 0$ as $s_1\rightarrow \omega$ within $\varphi=0$, cf.  \eqref{x2w1} and \eqref{s1time}. We therefore consider
 \begin{align*}
  \frac{dw_1}{ds_1} = (-1)^{l-1} w_1^{2-\kappa}.
 \end{align*}
  We focus on $l$ even (for $l$ odd, we consider the backward time). Since $\kappa\ge 2$, we have that
  \begin{align*}
   \frac{dw_1}{ds_1}<-\mu^{2-\kappa}<0,
  \end{align*}
for all $w_1\in [0,\mu]$. Hence for any initial condition $w_1(0)\in [0,\mu]$, there is an $\omega>0$ so that $w_1(s_1)\rightarrow 0^+$ for $s_1\rightarrow \omega^-$.

\end{proof}

	\begin{figure*}[t!]
		\centering
		\subfigure[]{\includegraphics[width=0.47\textwidth]{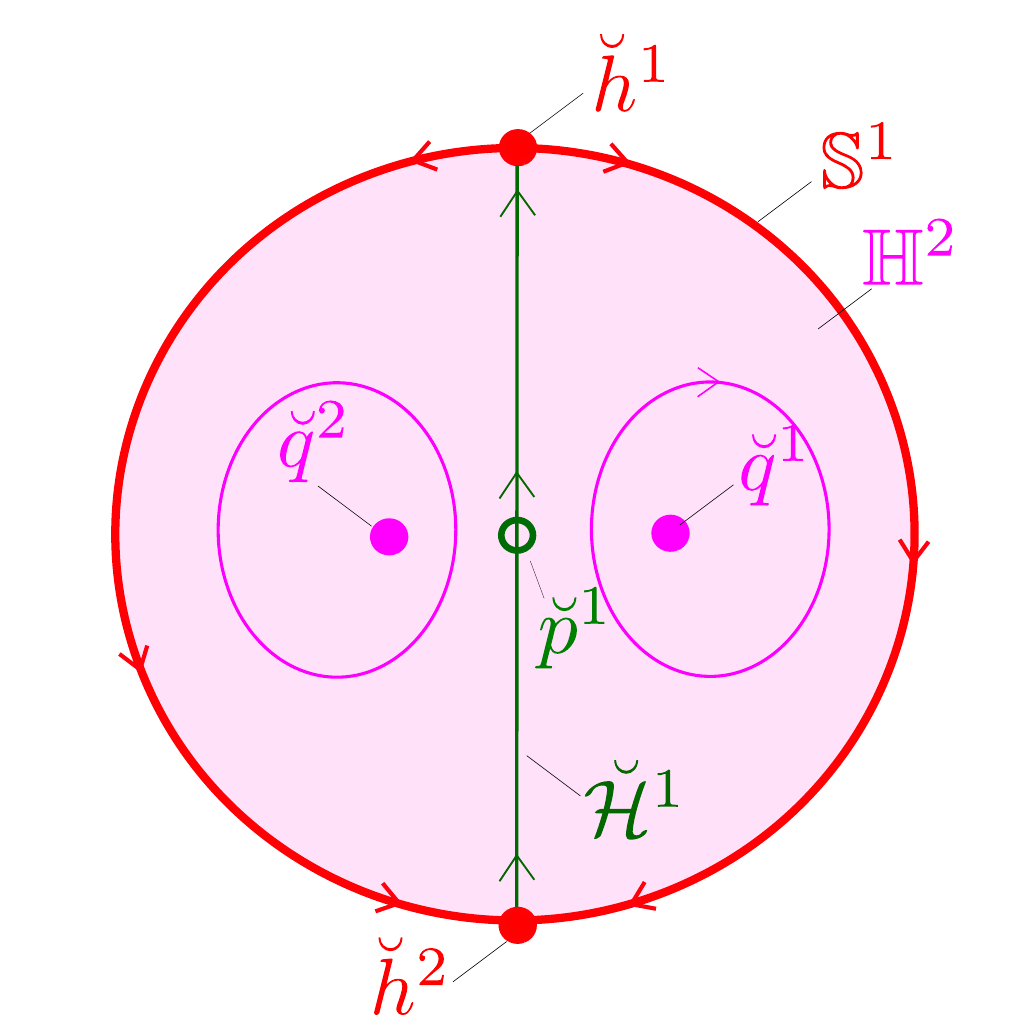}}
		\subfigure[]{\includegraphics[width=0.47\textwidth]{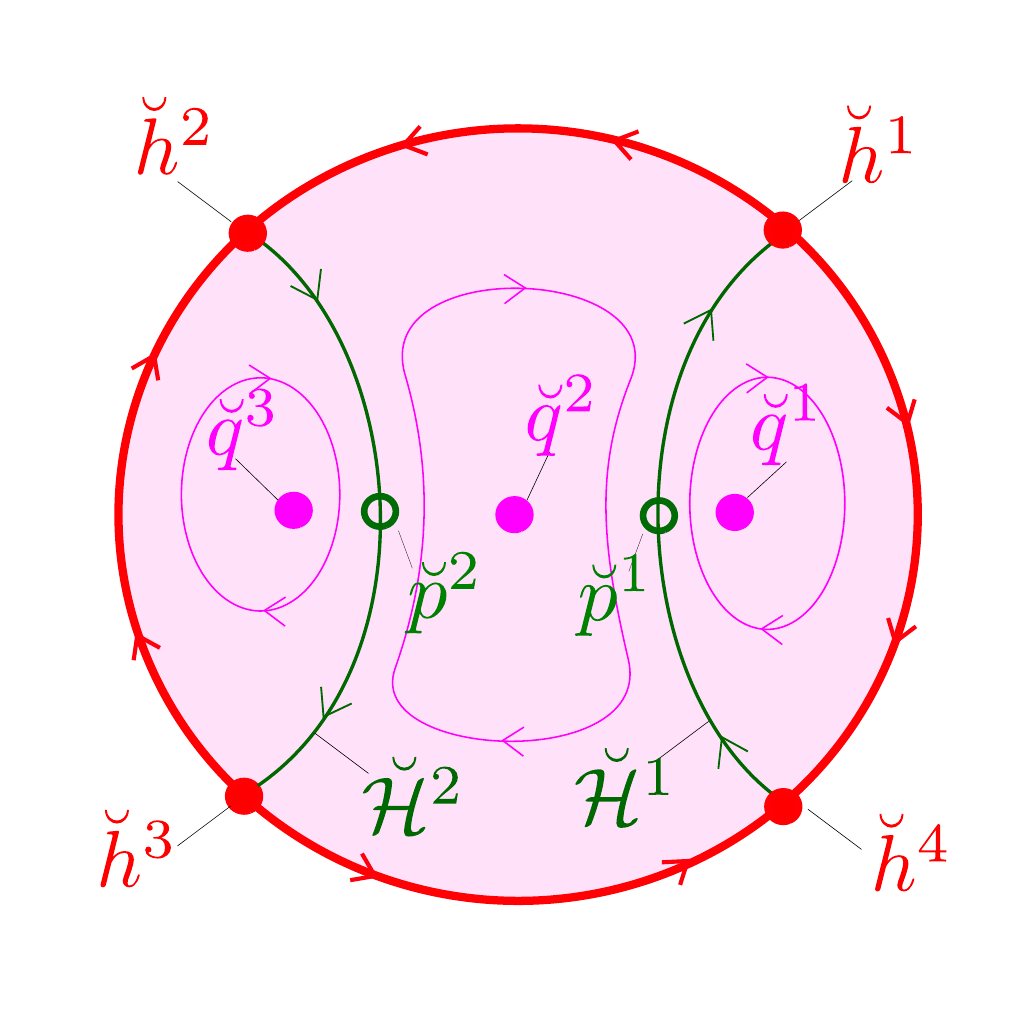}}
		\caption{Phase portraits of $x_2'=iQ(x_2)$ on the Poincar\'e hemisphere $\mathbb H^2$ for $\kappa=2$ (a) and $\kappa=3$ (b). The points $p^2<p^1$ define the sections $\{x_2=p^j\}$ where the separation of the invariant manifolds are measured.}
		\label{fig:imagtime}
	\end{figure*}

	Let $p^j:=\mathcal H^j \cap \{\operatorname{Im}(x_2)=0\}$, recall assertion \ref{it:imag6} of Lemma \ref{lem:imagtime}, and denote by $T^j>0$ the finite blowup time  along $\mathcal H^j$ for $x_2'(s)=iQ(x_2(s))$, $x_2(0)=p^j$, $s\ge 0$. Here $j\in \{1,\ldots,\kappa-1\}$. $T^j$ is well-defined by Corollary \ref{cor:blowuptime}: $T^j\in (0,\infty)$.
	\begin{lemma}
	 The following holds true
	 \begin{align}
	  \sum_{l=1}^{\kappa} \frac{1}{{Q}'(q^l)}=0,\label{eq:sumQkk}
	 \end{align}
and for every $j\in \{1,\ldots,\kappa-1\}$:
\begin{align}
 \sum_{l=1}^{j} \frac{1}{{Q}'(q^l)} =\frac{T^j (-1)^j}{\pi}.\label{eq:sumQkj}
\end{align}

	\end{lemma}
	\begin{proof}
        The formula \eqref{sumQkk} is a simple consequence of the Euler-Jacobi formula. Here we give an independent proof, which will also form the basis of \eqref{sumQkj}: Let $\gamma_\mu \subset \mathbb C$ denote the positively oriented circle of radius $\mu>0$ centered at the origin and assume that $\mu>0$ is large enough so that $\gamma_\mu$ contains all simple roots $q^\kappa<q^{\kappa-1}<\cdots <q^1$ (by \assumptionref{ass1}) in its interior.
		 Then by the residue theorem
		 \begin{align*}
		  \int_{\gamma_\mu} \frac{1}{{Q}(x_2)}dx_2 = 2\pi i \sum_{l=1}^\kappa \frac{1}{{Q}'(q^l)},
		 \end{align*}
for all such $\gamma_\mu$. We then complete the proof of \eqref{sumQkk} by letting $\mu\rightarrow \infty$, using
\begin{align*}
 \left| \int_{\gamma_\mu} \frac{1}{{Q}(x_2)}dx_2\right| =\mathcal O( \mu^{1-\kappa}).
\end{align*}

Next, for \eqref{sumQkj}, we consider the curve $ \gamma_{\mu}^j$ (with positive orientation) constructed in the following way: $\mathcal H^{j}$ divides $\mathbb C$ into two regions: $\mathcal H^{j,-}$ on the $\operatorname{Re}(x_2)$-negative side and $\mathcal H^{j,+}$ on $\operatorname{Re}(x_2)$-positive side, see \figref{imagtime}. We then define $\gamma_{\mu}^{j,-}$ as the subset of $\mathcal H^{j-1}$ that is contained in the interior of the closed curve $\gamma_\mu\subset \mathbb C$ (defined above), and let $\gamma_{\mu}^{j,+}=\gamma_\mu\cap \operatorname{Cl}({\mathcal H^{j,+}})$, where $\operatorname{Cl}$ denotes the closure. We then set $\gamma_{\mu}^j=\gamma_{\mu}^{j,-}\cup \gamma_{\mu}^{j,+}$, which is a closed curve for $\mu>0$ large enough by Lemma \ref{lem:imagtime}. Moreover, for all $\mu>0$ large, $\gamma_{\mu}^j$ contains all the simple roots $q^j<\ldots<q^1$ so that
 \begin{align*}
		  \int_{\gamma_{\mu}^j} \frac{1}{{Q}(x_2)}dx_2 = 2\pi i \sum_{l=1}^j \frac{1}{{Q}'(q^l)},
		 \end{align*}
		 by the residue theorem.
Letting $\mu\rightarrow \infty$, we conclude that
\begin{align*}
		  2\pi i \sum_{l=1}^j \frac{1}{{Q}'(q^l)} = \int_{\mathcal H^{j}} \frac{1}{{Q}(x_2)}dx_2 = 2 i T^j(-1)^j.
		 \end{align*}
        In the last equality, we have used that $dx_2 = i {Q}(x_2)ds$ by \eqref{imagtime} and the fact that $\mathcal H^j$ is symmetric with respect to conjugation, see assertion \ref{it:imag6} of Lemma \ref{lem:imagtime}.
		 The factor $(-1)^j$ comes from the fact that $\mathcal H^{j}$ is oriented positively if and only if  $j$ is even, see \figref{imagtime}. %Setting $T^j = \frac{1}{2\pi}S_{j-1}$ completes the proof.
	\end{proof}
\begin{remark}\remlab{rem_michaelson}
 For $\kappa=2$ and $Q(x_2)=-x_2^2+1$, we have $\mathcal H^1=i\mathbb R$ and therefore $p^1=0$.  Indeed, by setting $x_2=iv_2$ in \eqref{imagtime}, we have
 \begin{align*}
  v_2'=v_2^2+1,
 \end{align*}
 with the solution $v_2(t)=\tan (t)$ for $v_2(0)=0$. Consequently,
%  The solution of $x_2'=iQ(x_2)$ with $x_2(0)=0$ is therefore $x_2(t)=i\tan(t)$, satisfying
 \begin{align*}
  v_2(t)\to \infty \mbox{  for  $t\to T^1:=\frac{\pi}{2}$},
 \end{align*}
 in agreement with \eqref{sumQkj}.
%  in agreement with the definition above.

% and we therefore conclude that $T^1=\frac{\pi}{2}$ in this case.
\end{remark}

\begin{remark}\label{rem:el}
 Below we identify $\mathbb S^1$ with $\mathbb T^1=\mathbb R/(2\pi \mathbb Z)$ in the usual way. Then $\breve e^l$ and $\breve h^l$ become
 \begin{align*}
  \breve e^l:=\frac{\pi}{\kappa-1}(l-1),\quad \breve h^l:=\frac{\pi}{\kappa-1}(l-1)+\frac{\pi}{2(\kappa-1)},\quad l\in \{1,\ldots,2(\kappa-1)\}.
 \end{align*}

\end{remark}

\subsection{The unperturbed manifolds}
We first recall from the introduction, that the change of coordinates, defined by
\begin{align}\label{eq:xyzx2y2z21}
(x,y,z)\mapsto \begin{cases}
 x_2 =\epsilon^{-1} x,\\
 y_2 = \epsilon^{-\kappa} y,\\
 z_2 =\epsilon^{-\kappa} z,
\end{cases}
 \end{align}
 brings \eqref{x2y2z2fast} into \eqref{XYZfirst}, repeated here for convenience:
 \begin{equation}\label{eq:XYZ}
\begin{aligned}
 \dot x&={P}(x,\epsilon)+F(x,y,z,\epsilon),\\
 \dot y&=z-\frac12 P'_x(x,\epsilon) y+G(x,y,z,\epsilon),\\
 \dot z &=-y -\frac12 P'_x(x,\epsilon) z+ H(x,y,z,\epsilon),
\end{aligned}
\end{equation}
where $P$ is the polynomial defined in \eqref{PkQkfirst}:
\begin{align}
 {P} (x,\epsilon) := \epsilon^\kappa {Q}(\epsilon^{-1} x) = -x^\kappa + \sum_{\alpha=0}^{\kappa-1} \epsilon^{\kappa-\alpha} a_\alpha x^\alpha.\label{eq:PkQk}
\end{align}
Notice that  $P$ is homogeneous:
\begin{align}
 P(r\breve x,r\breve \epsilon) = r^\kappa P(\breve x,\breve \epsilon)\quad \forall\,r\ge 0,\,(\breve x,\breve \epsilon)\in \mathbb S^1.\label{eq:homoP}
\end{align}
Now, setting $\epsilon=0$ in \eqref{XYZ} gives
\begin{equation}\label{eq:XYZ0}
\begin{aligned}
 \dot x &= -x^\kappa+F(x,y,z,0),\\
  \dot y&=z+\frac{\kappa}{2} x^{\kappa-1} y+G(x,y,z,0),\\
 \dot z &=-y +\frac{\kappa}{2} x^{\kappa-1}z+H(x,y,z,0),
\end{aligned}
\end{equation}
Here $(x,y,z)=(0,0,0)$ is a zero-Hopf singularity with the eigenvalues of the linearization being $0,\pm i$ (since $\kappa\in \mathbb N$, $\kappa\ge 2$)
% The degeneracy condition of \cite{baldom2013a} (that the coefficient of $x^2$ in the $x$-equation is nonzero at the bifurcation) is however violated for $k>2$ and therefore \eqref{XYZ} cannot be brought into normal form \eqref{zeroHopf0} in this case.
\begin{remark} We mention that in the case of \eqref{3rd}, the system \eqref{XYZ0} can be brought into the following form
% \begin{align*}
%  \dot x
% \end{align*}
\begin{equation}\label{eq:xyzP00}
\begin{aligned}
 \dot x &=-(x-z)^\kappa,\\
 \dot y &=z,\\
 \dot z &=-y-(x-z)^\kappa,
\end{aligned}
\end{equation}
by an analytic change of coordinates,
see \eqref{xyzP0app}. We can also reduce this system to the third order equation
\begin{align*}                                                                                                                                                                                                                                                                                                                                                                                 \widehat f'''+\widehat f' = -\widehat f^\kappa.                                                                                                                                                                                                                                                                                                                                                                           \end{align*}
upon setting $\widehat f=x-z$. Alternatively, one can define $f=\epsilon^{-1} \widehat f$ in \eqref{3rd} and let $\epsilon \to 0$. %We emphasize that $\widehat f=x-z$.
                                                                                                                                                                                                                                                                                                                                                                               \end{remark}
We now turn to invariant manifold solutions of \eqref{XYZ0} that are graphs over $x$ in sectors of the complex plane. We therefore define the local and open sector:
\begin{align}\label{eq:Sector}
 S(\breve e,\delta,\eta) = \left\{x\in \mathbb C\,:\,0< \vert x\vert<\delta,\,\vert\operatorname{arg}(x)-\breve e\vert< \frac12 \eta\right\},
\end{align}
centered along the \textit{direction} $\breve e\in [0,2\pi)$ with \textit{radius} $\delta>0$ and \textit{opening} $\eta\in [0,2\pi)$.
\begin{lemma}\label{lem:Sj}
Consider $\chi>0$ and $\delta>0$ small enough and define
\begin{align}
 \breve e^l: = \frac{\pi}{\kappa-1} (l-1),\quad l\in\{1,\ldots,2(\kappa-1)\},\quad \eta := \frac{\pi}{\kappa-1}+\chi,\label{eq:ehatlangles}
\end{align}
see Remark \ref{rem:el} (or recall Lemma \ref{lem:realtime} (assertion \ref{it:real3})),
and 
\begin{align}
 S^l:=S(\breve e^l,\delta,\eta)\subset \mathbb C,\quad l\in \{1,\ldots,2(\kappa-1)\},\label{eq:Sldefn}
\end{align}
see Fig. \ref{fig:SS}.
Then the system \eqref{XYZ0} has $2(\kappa-1)$-many analytic invariant manifolds defined locally by
 the graph form 
 \begin{align}\label{eq:unperturbed}
 \begin{pmatrix}
   y\\
   z
 \end{pmatrix} &= x^\kappa \psi^l(x),\quad x\in S^l,
 \end{align}
 for any $l\in \{1,\ldots,2(\kappa-1)\}$.
 The family $\{\psi^l\}_{l=1}^{2(\kappa-1)}$, with $\psi^l:S^l\to \mathbb C^2$ for each $l\in \{1,\ldots,2(\kappa-1)\}$, is equivariant with respect to conjugation in the following sense:
\begin{align}\label{eq:conj}
\overline{\psi^l(x)} = \psi^{2\kappa-l}(\overline x),\quad x\in S^l, % \overline{\phi_i(s)}=\phi_i(\overline s)\quad \forall\, s\in S(\pi(i-1),\chi^{\kappa-1},(\kappa-1)\eta),
\end{align}
with $2\kappa-l$ understood $\operatorname{mod}(2(\kappa-1))$. Moreover, $\psi^l$ is continuous on the closure of $S^l$ with $\psi^l(0)=(0,0)$,
%  Here
% \begin{align*}
%
% \end{align*}
%  with $2\kappa-l$ understood $\operatorname{mod}(2(\kappa-1))$, for any $l\in \{1,\ldots,2(\kappa-1)\}$, and each
% Finally, $\psi^l:S^l\rightarrow \mathbb C^2$ is
being the $(\kappa-1)$-sum of a Gevrey-$\frac{1}{\kappa-1}$ series:
\begin{align}\label{eq:fseries0}
 \psi^l(x)\sim_\frac{1}{\kappa-1} \sum_{\alpha=1}^\infty \psi_\alpha x^\alpha,\quad \vert \psi_{\alpha}\vert\le  c_1 c_2^{\alpha-1} \Gamma\left(\frac{\alpha}{\kappa-1}\right) \quad \forall\,\alpha\ge 2(\kappa-1),
\end{align}
for some $c_1,c_2>0$,
 in the direction $\breve e^l$ for any $l\in \{1,\ldots,2(\kappa-1)\}$. In particular, the series is independent of $l$.
% The series $\sum_{n=\kappa}^\infty \psi_{n} s^{n}$ is $1$-summable with $s\in i\mathbb R$ defining the singular directions.
% with $\sum_{n=\kappa}^\infty \phi_{n} x^n$ a formal $\frac{1}{\kappa}$-Gevrey series.
% % \end{align'}

\end{lemma}
\begin{proof}
We define $(r_1,y_1,z_1)$ by 
\begin{align}\label{eq:blowup0}
 \begin{cases}
 x =r_1,\\
  y =r_1^\kappa y_1,\\
  z = r_1^\kappa z_1.
 \end{cases}
\end{align}
This brings \eqref{XYZ0} into the following form
\begin{equation}\label{eq:y1z1r1}
\begin{aligned}
%  (\kappa-1)r_1^{2(\kappa-1)}  \frac{dy_1}{dr_1^{\kappa-1}}
%  
 r_1^\kappa \left[-(1-r_1^{\kappa-1}z_1)^\kappa+r_1^\kappa F_1(r_1,y_1,z_1,0)\right]\frac{dy_1}{d r_1} & =  z_1+\frac{\kappa}{2}r_1^{\kappa-1} y_1+r_1 G_1(r_1,y_1,z_1,0),\\
 r_1^\kappa\left[-(1-r_1^{\kappa-1}z_1)^\kappa+r_1^\kappa F_1(r_1,y_1,z_1,0)\right]\frac{dz_1}{dr_1} &=  -y_1+\frac{\kappa}{2}r_1^{\kappa-1} z_1+r_1 H_1(r_1,y_1,z_1,0),
\end{aligned}
\end{equation}
upon eliminating time through division by $\dot r_1=\dot x$.
Here we have defined
\begin{align}\label{eq:R1S1T1}
 W_1(r_1,y_1,z_1,\epsilon_1):=r_1^{-3\kappa+2} W(r_1,r_1^\kappa y_1,r_1^\kappa z_1,r_1\epsilon_1),\quad W=F,G,H,
\end{align}
which are well-defined and analytic for $r_1=0$ by \eqref{Xcond}.
The system \eqref{y1z1r1} defines a generalized saddle-node with Poincar\'e rank $\kappa-1$, see e.g. \cite{bonckaert2008a,ksum}. We therefore obtain the solutions $(y_{1},z_{1})=\psi^l(r_1)\in \mathbb C^2$, $r_1\in S^l$, by applying \cite[Theorem 2.1]{ksum} with $k= \kappa-1$, see \cite[Eq. (1)]{ksum}. For comparison with \cite{ksum}, we emphasize that $\lambda_1=\overline \lambda_2 = \pm i$ and $\omega_{\kappa-1}(\delta,\hat \eta,\chi)=S(\breve e,\delta,\eta)$. Recall that $\breve h^l:=\frac{\pi}{\kappa-1}(l-1)+\frac{\pi}{2(\kappa-1)}$, see Remark \ref{rem:el} (recall also Lemma \ref{lem:imagtime} (assertion \ref{it:imag3})). Then a simple calculation shows that the condition \cite[Eq. (12)]{ksum} is
% therefore satisfied for all $(k-1)\theta  \ne \frac{\pi}{2} +\pi n$, $n\in \mathbb Z$, for some $\xi>0$ small enough.
satisfied for some $\xi>0$ small enough if
\begin{align*}
 \breve e\ne \breve h^l,
\end{align*}
for all $l\in\{1,\ldots,2(\kappa-1)\}$.
The theorem therefore applies for all $\breve e=\breve e^l$, $l\in \{1,\ldots,2(\kappa-1)\}$. Due to the invariance of $z=0$ for the normal form \cite[Eq. (14)]{ksum}, the desired invariant manifold is then given by \cite[Eq. (13)]{ksum} with $z=0$, see also \cite[Proposition 4.1]{ksum}. %The property  \eqref{conj} is an easy consequence of \eqref{
%We are only interested in $\mathbf j=0$ so that the normal form \cite[Eq. (14)]{ksum} has $z=0$ as invariant set. Consequently Notice that
Finally, the property \eqref{conj} follows from the uniqueness of $\psi^l:S^l\rightarrow \mathbb C^2$ (see \cite{ksum}) and the fact that the system \eqref{XYZ0} is real-analytic, being invariant with respect to conjugation.
\end{proof}

%  \begin{remark}

In the following, we will refer to the $2(\kappa-1)$-many invariant manifolds given by \eqref{unperturbed} as \textbf{the unperturbed manifolds of \eqref{XYZ}}.

%  is not trnonzero.  invariant manifolds are nonanalytic. 
% \begin{remark}
\begin{remark}\label{rem:ksum}
In the case of \eqref{xyzP00}, one can reduce \eqref{y1z1r1} to Poincar\'e rank $1$ upon setting $s=r_1^{\kappa-1}$:
\begin{equation}\label{eq:y1z1s}
\begin{aligned}
%  (k-1)r_1^{2(k-1)}  \frac{dy_1}{dr_1^{k-1}}
%
(\kappa-1)s^2 \frac{dy_1}{ds} & =-(1-s z_1)^{-\kappa} z_1 - \kappa s  y_1,\\
 (\kappa-1)s^2\frac{dz_1}{ds} &=  (1-s z_1)^{-\kappa}y_1+1 - \kappa s  z_1.
\end{aligned}
\end{equation}
Consequently, for \eqref{3rd} we can write the resulting functions $\psi^l$ in the following form
\begin{align}\label{eq:PsiPhi}
 \psi^{l}(x) = \begin{cases}
             \phi^1 (x^{\kappa-1})\mbox{ for $l$ odd},\\
             \phi^2 (x^{\kappa-1})\mbox{ for $l$ even},
            \end{cases}
\end{align}
%
%
%
% \begin{align*}
%  =x^\kappa\phi_1(x^{\kappa-1}),\quad x\in S_{2l-1},\quad \forall\,l\in \left\{1,3,\ldots,2\lfloor \tfrac{\kappa+1}{2}\rfloor -1\right\},\\
%  \begin{pmatrix}
%    y\\
%    z
%  \end{pmatrix} &= \psi^j(x):=x^\kappa\phi_2(x^{\kappa-1}),\quad x\in S_{2l},\quad \forall\,l\in \left\{2,4,\ldots,2\lfloor \tfrac{\kappa}{2}\rfloor\right\},
% \end{align*}
where $$\phi^i:S(\pi(i-1),\delta^{\kappa-1},(\kappa-1)\eta)\rightarrow \mathbb C^2, \quad i\in\{1,2\}.$$ In particular, $\phi^1$ and $\phi^2$ in \eqref{PsiPhi} are the $1$-sums of a Gevrey-$1$ formal series
\begin{align}\label{eq:fseries}
 \phi^i(s) \sim_{1} \sum_{\alpha=0}^\infty \phi_{\alpha} s^{\alpha}\quad \mbox{ for } \, s\in S(\pi(i-1),\delta^{\kappa-1},(\kappa-1)\eta),
\end{align}
for $i\in \{1,2\}$, along the directions $\breve e=0$ and $\breve e=\pi$, respectively.
% Moreover, $\phi_n\in \mathbb R^2$ for all $n\in \mathbb N_0$ and
% % with
% \begin{align*}
%  \phi_n = \begin{cases}
%            (*,0) & \mbox{$n$ even},\\
%            (0,*) & \mbox{$n$ odd},
%           \end{cases}\quad \phi_{0} = (-1,0),\quad \phi_1 = (0,\kappa).
% \end{align*}
% This is
%  The existence of $2(\kappa-1)$ locally invariant manifold can also be obtained from \cite{\kappasum} directly by working on \eqref{y1z1r1} (with Poincar\'e rank $k-1$). In this way, we obtain solutions $(y_{1},z_{1})=\psi^l(x)\in \mathbb C^2$, $x\in S^l$, that are $k$-sums of a Gevrey-$\frac1k$ formal series in $x$. This will be important in generalizations of our results, where the system for $\epsilon=0$ cannot be reduced to Poincar\'e rank $1$ as in \eqref{y1z1s}. In the present paper, we will also use \cite{ksum} for our generalized system below in the development of certain normal forms suitable for extended the invariant manifolds, see Section \ref{sec:exit}.
%  \end{remark}
% In the cases, where \eqref{y1z1r1} can be reduced to Poinar\'e rank $1$ (upon setting $s=r_1^{k-1}$) we have
In this case, we have that
\begin{align}\label{eq:diffPsi}
x^\kappa(\psi^{j+1}(x)-\psi^j(x)) = x^\kappa (\phi^2(x^{\kappa-1})-\phi^1(x^{\kappa-1})),\quad x\in S(\breve h^j,\delta,\chi).
\end{align}
for all odd $j$. Here the right hand side of \eqref{diffPsi} is independent of $j$. In \eqref{diffPsi}, we have used that
\begin{align*}
(S^{j+1}\cap S^j)=S(\breve h^j,\delta,\chi),
\end{align*}
cf. \eqref{Sldefn} and \eqref{Sector}, see also Remark \ref{rem:el}.
\end{remark}
% is determined by $\phi_1$ and $\phi_2$.

 In our main result, we will assume the following:
 \begin{assumption}
	\assumptionlab{ass2}
	%The vector field $Z : \mathbb R^2 \times \mathbb R \times I \to \mathbb R^2$ satisfies the following constraints:
% 	${Q}$ is real and there are $k$ simple real roots $q^k<q^{k-1}<\cdots<q^1$ of ${Q}$.
The asymptotic series \eqref{fseries0} is divergent for any $x\ne 0$
\end{assumption}
\begin{lemma}\label{lem:Psil}
 Suppose that \assumptionref{ass2} holds true. Then there exists at least one $j\in \{1,\ldots,2(\kappa-1)\}$ such that $\psi^{j+1}\ne \psi^j$ on
 \begin{align*}
 (S^{j+1}\cap S^j)=S(\breve h^j,\delta,\chi),\quad \breve h^j = \breve e^j + \frac{\pi}{2(\kappa-1)},
 \end{align*}
 with $2\kappa-1\equiv 1$, recall \eqref{ehatlangles}.  Suppose next that \eqref{diffPsi} holds and that the series \eqref{fseries} is nonanalytic. Then $\psi^{j+1}\ne \psi^j$ for all $j$.
\end{lemma}
\begin{proof}
 Suppose first that $\psi^{j+1}(x)\ne \psi^j(x)$ for some
 $x\in S(\breve h^j,\delta,\chi)$ and some $j\in \{1,\ldots,2(\kappa-1)\}$. Then since \eqref{y1z1r1} is regular for all $x=r_1\ne 0$ small enough, this implies that $\psi^{j+1}(x)\ne \psi^j(x)$ for all $
  x\in S(\breve h^j,\delta,\chi)$ by uniqueness of solutions.

  Suppose next that
 $\psi^{l+1}= \psi^l$ on $S(\breve h^l,\delta,\chi)$ for all $l\in \{1,\ldots,2(\kappa-1)\}$. Then $\{\psi^l\}_{l\in \{1,\ldots,2(\kappa-1)\}}$ defines a single analytic function $\psi$ on $B_{\delta}$ by Riemann's theorem of removable singularities. Hence \eqref{fseries0} is convergent and we arrive at a contradiction.

 We believe that the statement regarding \eqref{diffPsi} is clear enough, see also \cite[Lemma 5.3]{bkt}.
%  It is clear that $\psi_{l+1}(x)\ne \psi^l(x)$ for some
%  $x\in S(\breve h^l,\delta,\chi)$,
%  for at least one $l\in \{1,\ldots,2(k-1)\}$. However, since \eqref{y1z1r1} is regular for $x=r_1\ne 0$, this implies that $\psi_{l+1}(x)\ne \psi^l(x)$ for all $
%  x\in S(\breve h^l,\delta,\chi).$
\end{proof}

%  This means that Clearly
% \end{remark}
% \begin{lemma}
%  There exists a constant $\Theta$ so that 
%  \begin{align*}
%   \phi_2(s)-\phi_1(s) ??
%  \end{align*}
% 
% \end{lemma}
% \begin{proof}
%  We define $(\Delta y,\Delta z) = \phi_2(s)-\phi_1(s)$. Then by the mean value theorem we find
%  \begin{align*}
%  (k-1)s^2 \frac{d\Delta y}{ds} &=-ks \Delta y - (1+2k^2 s^2+\mathcal O(s^3))\Delta z,\\
%  (k-1)s^2 \frac{d\Delta z}{ds} &=-2ks(1 +\mathcal O(s^2)) \Delta z + (1+k^2 s^2+\mathcal O(s^3))\Delta y,
%  \end{align*}
% for $s\in S(\pi/2,\chi,\chi^{k-1})\cup G(-\pi/2,\chi,\chi^{k-1})$. We put $s=it$:
% 
% \end{proof}

\begin{figure*}[t!]
		\centering
		\subfigure[]{\includegraphics[width=0.47\textwidth]{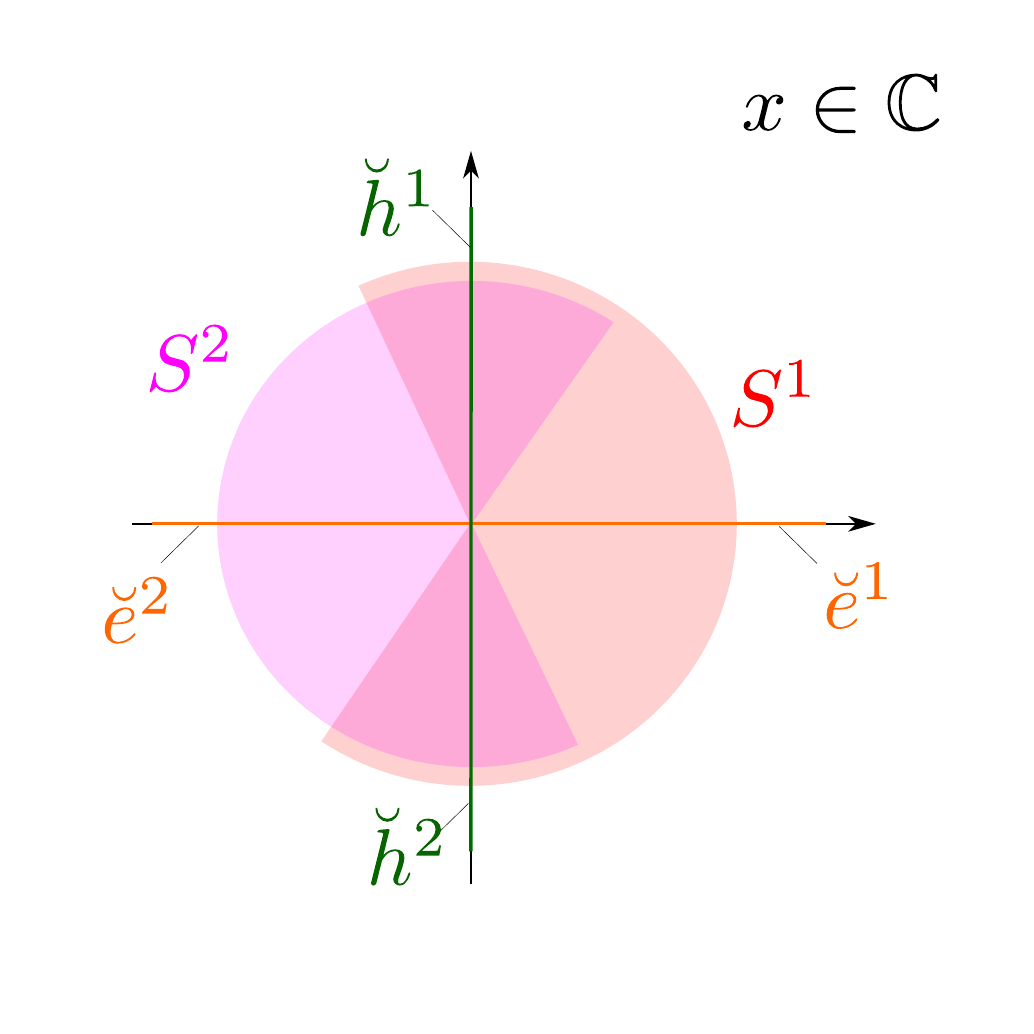}}
		\subfigure[]{\includegraphics[width=0.47\textwidth]{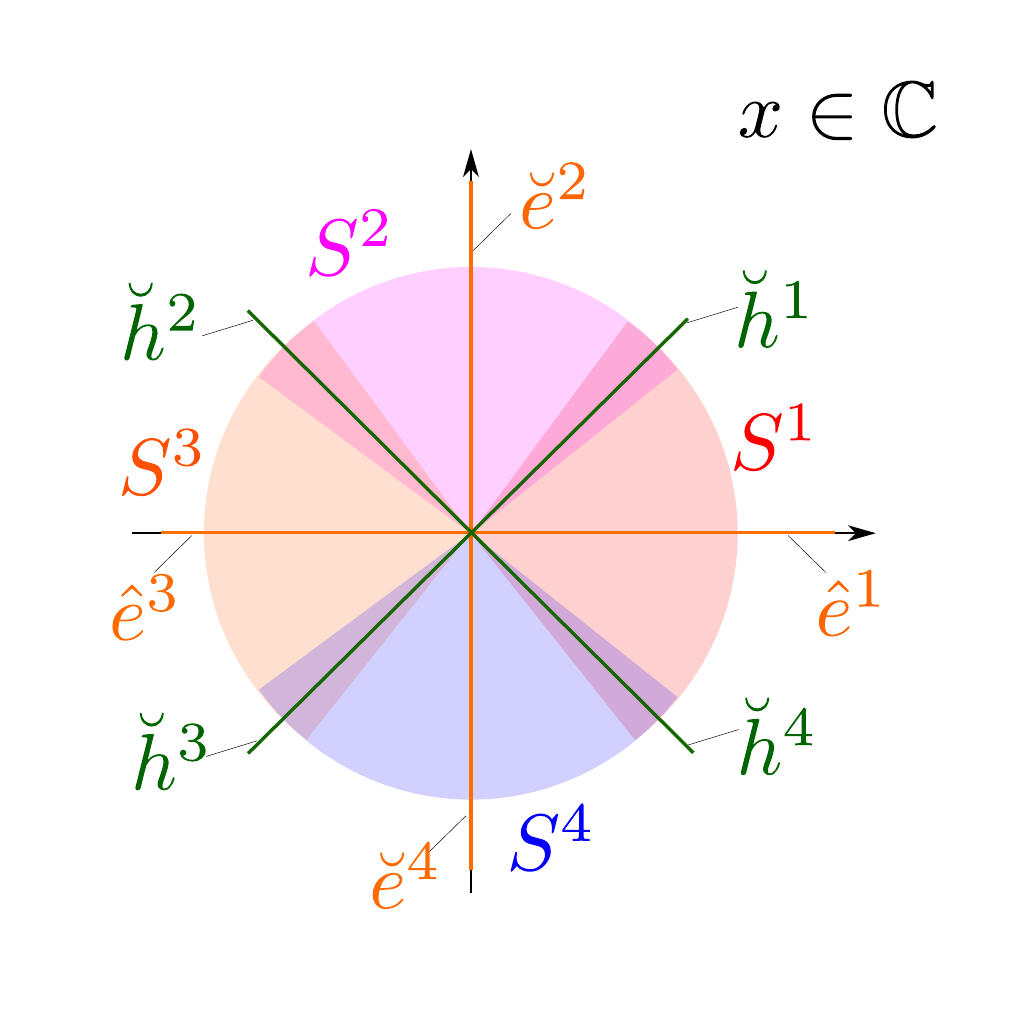}}
		\caption{The sectors $S^l$, $l\in \{1,\ldots,2(\kappa-1)\}$, in Lemma \ref{lem:Sj} for $\kappa=2$ (a) and $\kappa=3$ (b). The sectors $S^l$ are centered along $\breve e^l$. On the hand, $\breve h^l$ corresponds to a $\frac{\pi}{2(\kappa-1)}$ rotation of $\breve e^l$ counter-clockwise. It will be important that $S^{l}\cap S^{l+1}$ (with $2\kappa-1\equiv 1$) is an open sector centered along $\breve h^l$ with opening $2\chi>0$.  }
		\label{fig:SS}
	\end{figure*}

	\subsection{Main result}\label{sec:main}
% \note{I think we should present the normal form here; we could delay the proof to the appendix}
We are now ready to state our main result on the asymptotic expansion of the splitting of $\mathbf{W}^{\sigma^j}(\mathbf q^{j}(\epsilon))$ and $\mathbf{W}^{\sigma^{j+1}}(\mathbf q^{j+1}(\epsilon))$ for $j\in \{1,\ldots,\kappa-1\}$. Here $\sigma^j\in \{u,s\}$ is given by \eqref{Wsu}.
\begin{theorem}\label{thm:main}
Consider \eqref{x2y2z2new0}, with $F,G,H$ locally defined and real-analytic, and
suppose that \assumptionref{ass1} and \assumptionref{ass2} hold true. Let
 $j\in \{1,\ldots,\kappa-1\}$ be such that $\psi^{j+1}\ne \psi^j$ on
 \begin{align*}
 (S^{j+1}\cap S^j)=S(\breve h^j,\delta,\chi),
 \end{align*}
 recall Lemma \ref{lem:Psil}.

 Next, define $p^j\in (q^{j+1},q^j)\subset \mathbb R$ as the intersection of ${\mathcal H}^j\subset \mathbb C$ with $\operatorname{Im}(x_2)=0$, recall Lemma \ref{lem:imagtime}.

 Finally, let $
(y_2^l,z_2^l)(p^j,\epsilon)$, $l\in \{j,j+1\}$, denote the $(y_2,z_2)$-coordinates of the first intersection of the local invariant manifolds $\mathbf{W}_{loc}^{\sigma^{l}}(\mathbf q^{l}(\epsilon))$ with $\{x_2=p^j\}$, and define
\begin{align*}
 \begin{cases} \Delta y_2^j(p^j,\epsilon): = y_2^{j+1}(p^j,\epsilon)-y_2^{j}(p^j,\epsilon),\\
  \Delta z_2^j(p^j,\epsilon): = z_2^{j+1}(p^j,\epsilon)-z_2^{j}(p^j,\epsilon),
 \end{cases}
\end{align*}
as the separation of $\mathbf{W}_{loc}^{\sigma^{j+1}}(\mathbf q^{j+1}(\epsilon))\cap \{x_2=p^j\}$ and $\mathbf{W}_{loc}^{\sigma^{j}}(\mathbf q^{j}(\epsilon))\cap \{x_2=p^j\}$.

Then there is an $\epsilon_0>0$ small enough such that the following holds true for all $\epsilon\in (0,\epsilon_0)$: $\Delta y_2^j(p^j,\epsilon)$ and $\Delta z_2^j(p^j,\epsilon)$ are well-defined, and have the following expansions:
% the ½+
% ly small splitting is given by
 \begin{equation}\label{eq:finaldist}
\begin{aligned}
\Delta y_2^j(p^j,\epsilon) +i \Delta z_2^j(p^j,\epsilon) &= \epsilon^{-\frac{3\kappa}{2}}\e^{-\epsilon^{1-\kappa} T^j +\Xi^j(\epsilon)} %
% \epsilon^{-\frac{3\kappa}{2}}\exp\left(-\epsilon^{1-\kappa} \Xi^j(\epsilon)\right),
\end{aligned}
\end{equation}
with $T^j>0$ given by \eqref{sumQkj}:
\begin{align*}
 T^j = \left| \sum_{l=1}^j \frac{\pi}{Q'(q^l)}\right|>0,
\end{align*}
and where $\Xi^j:[0,\epsilon_0)\rightarrow \mathbb C$ is a  $C^\infty$-smooth function.

% and where the $\mathcal O(\epsilon)$-term is $C^\infty$-smooth with respect to $0\le \epsilon\ll 1$.

% $\mathbf{W}_{loc}^{\sigma_{i}}(\mathbf q^{i}(\epsilon))$, $i\in \{j,j+1\}$, denote the local invariant manifolds of the hyperbolic singularities $\mathbf q^{j+1}(\epsilon)$ and $\textbf{q}_j$ of \eqref{x2y2z2new0} for all $0<\epsilon\ll 1$.  Finally, define
% Then the following holds true:
%
% Finally,
%
%
% define
% \begin{align*}
%  \Delta :=),
% \end{align*}
% where $ HereThen the following holds true:

\end{theorem}
% \begin{remark}
% Through elementary calculations based on \eqref{finaldist}, we find that the distance $d^j(\epsilon):=\vert (\Delta y^j(p^j,\epsilon),\Delta z^j(p^j,\epsilon)\vert$, takes the following form
%  \begin{align*}
%  d^j(\epsilon) =\epsilon^{-\frac{3\kappa}{2}}\e^{-\epsilon^{1-\kappa} T^j}
%  \left(1+\mathcal O(\epsilon^{2(\kappa-1)})\right),\quad \operatorname{Im}(\Xi_0^j(0)) = 0.
%  \end{align*}
%   Here $\mathcal O(\epsilon^{2(\kappa-1)})$ is a $C^\infty$-smooth function of $\epsilon$.
% \end{remark}

		\begin{remark}\label{rem:phin}
		Theorem \ref{thm:main} agrees with Theorem \ref{thm:baldom2013a} for $\kappa=2$ with $Q(f)=-f^2+1$ and $b=1$ since $p^1=0$ and $T^1=\frac{\pi}2$ in this case, see also \remref{rem_michaelson}. The result also agrees with \eqref{michaelson} since
% 		Theorem \ref{thm:main} agrees with Theorem \ref{thm:baldom2013a} for $\kappa=2$ with $Q(f)=-f^2+1$ and $b=1$. The result also agrees with \eqref{michaelson} since $p_1=0$ and $T^1=\frac{\pi}2$ in this case, see \remref{rem_michaelson},  and
		\begin{align*}
		 z_{2}^1-z_{2}^2 = \epsilon \frac{\partial^2 f^+}{\partial t^2}-\epsilon \frac{\partial^2 f^-}{\partial t^2},
		\end{align*}
		for $f=0$,
cf. \appref{A}, see \eqref{x2y2z2f} and \eqref{nft} with $\kappa=2$ and $Q(f)=-f^2+1$.

	 To understand how Theorem \ref{thm:main} applies to \eqref{3rd}, we first recall \eqref{PsiPhi} and \eqref{fseries}. Then in Fig. \ref{fig:test} we illustrate $\log \vert \phi_\alpha\vert^{\frac{1}{\alpha}}$ as a function of $\log \alpha$ for $\alpha=1,\ldots,80$ and $\kappa=2,\ldots,8$. Here $\phi_\alpha$, $\alpha\in \mathbb N$, are the coefficients in \eqref{fseries}, which we determine recursively in Maple. The result in Fig. \ref{fig:test} provides strong evidence for the factorial growth of $\vert \phi_\alpha\vert$ as $\alpha\rightarrow \infty$ (and in turn evidence for the divergence of \eqref{fseries}) for these values of $\kappa$. This implies that $\phi^2\ne \phi^1$ on the overlapping domain and then in turn (by \eqref{diffPsi}) that Theorem \ref{thm:main} applies to all connections $j\in \{1,\ldots,\kappa-1\}$.
				\begin{figure*}[t!]
		\centering
		{\includegraphics[width=0.67\textwidth]{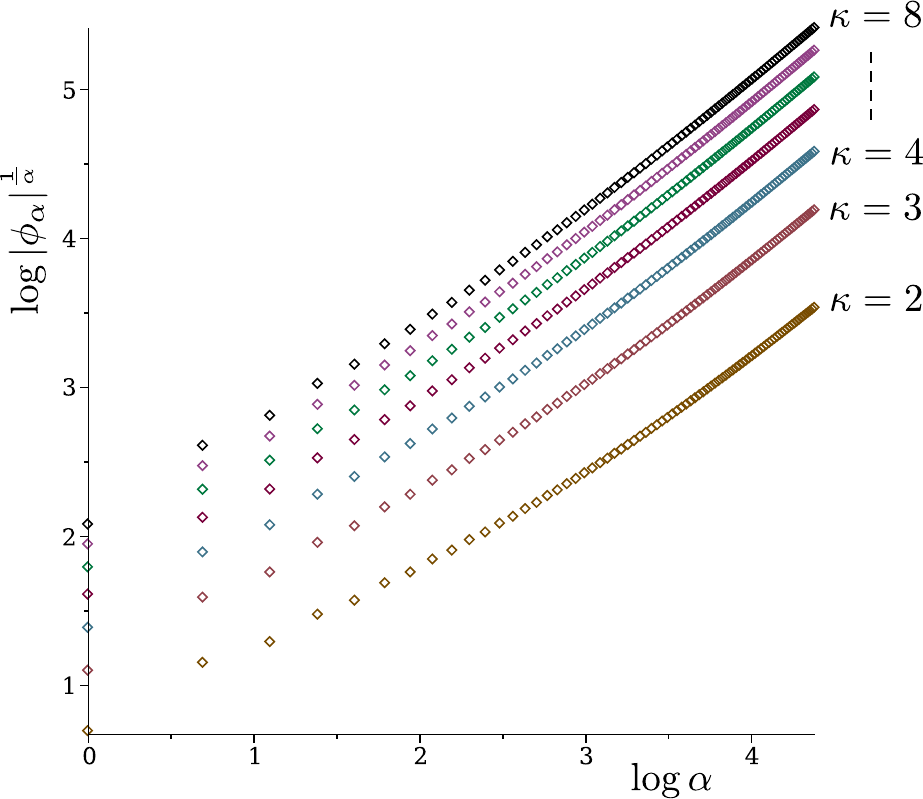}}
				\caption{$\log \vert \phi_\alpha\vert^{\frac{1}{\alpha}}$ as a function of $\log \alpha$ for $\kappa=2,\ldots,8$. Here $\phi_\alpha$, $\alpha\in \mathbb N$, is the asymptotic series \eqref{fseries} for \eqref{y1z1s}.  }
		\label{fig:test}
	\end{figure*}

	\end{remark}

\begin{remark}
The unperturbed manifolds are given by Borel-Laplace. The linear part of the equation for $\mathcal B(\psi_0)$ have poles at $w\,:\,w^{\kappa-1}\pm i=0$, see \cite[Eq. (37)]{ksum} and this is the general case for the Borel transform of $\psi_0$ too. In this case, $\psi^{j+1}\ne \psi^j$ for all $j$.

As emphasized in \cite[Remark 1.4]{bkt}, our approach does not require us to assume that  $\psi^{j+1}\ne \psi^j$. In fact,
\begin{equation}
\begin{aligned}
\Delta y_2^j(p^j,\epsilon) +i \Delta z_2^j(p^j,\epsilon) &= \epsilon^{-\frac{3\kappa}{2}}\e^{-\epsilon^{1-\kappa} T^j}\widetilde \Xi^j(\epsilon), %
% \epsilon^{-\frac{3\kappa}{2}}\exp\left(-\epsilon^{1-\kappa} \Xi^j(\epsilon)\right),
\end{aligned}
\end{equation}
would be a more general statement for the splitting, with $\widetilde \Xi^j(0)\ne 0$ if $\psi^{j+1}\ne \psi^j$. This is a consequence of Lemma \ref{lem:final}. Similarly,
\begin{align*}
    \begin{cases}\frac{\partial^{\alpha}\widetilde \Xi^j}{\partial \epsilon^\alpha}(0)=0,\quad \forall\,\alpha\in \{0,1\ldots,\beta-1\},\\
     \frac{\partial^{\beta}\widetilde \Xi^j}{\partial \epsilon^\beta}(0)\ne 0,
    \end{cases}
    \end{align*}
 where $\beta\in \mathbb N$ is the smallest number for which the higher order ``corrections'' in Lemma \ref{lem:Pnl} of the unperturbed manifolds satisfy $\psi_{1,\beta}^{j+}\ne \psi_{1,\beta}^{j}$. We leave out further details.
\end{remark}

% 	It is possible to determine $\Xi^j(\epsilon)$ explicitly up to terms including $\mathcal O(\epsilon^{\kappa})$, but we leave the details of this to the interested reader.
	\subsection{Strategy of the proof}

Our strategy follows \cite{bkt}. The idea is to extend the stable and unstable manifolds of $\mathbf q^j(\epsilon)$, that are naturally described in the $(x_2,y_2,z_2)$-coordinates, to $x=\mathcal O(1)\in \mathbb C$, $\vert x\vert>0$, and compare them there with the unperturbed manifolds of Lemma \ref{lem:Sj} (along the appropriate sectors). To perform this extension we use blow-up and view the scaling \eqref{xyzx2y2z21} as a chart ($\breve \epsilon=1$) associated with the blow-up transformation
\begin{align}\label{eq:bu}
 r\ge 0,\,(\breve x,\breve y,\breve z,\breve \epsilon)\in \mathbb S^3\mapsto \begin{cases}
                                                       x = r \breve x,\\
                                                       y =r^\kappa \breve y,\\
                                                       z =r^\kappa \breve z,\\
                                                       \epsilon =r\breve \epsilon.
                                                      \end{cases}
                                                      \end{align}
As usual in blow-up analysis, the matching between $x=\mathcal O(1)$ and $x=\mathcal O(\epsilon)$ is therefore performed in the $\breve x=1$-chart with chart-specific coordinates $(r_1,y_1,z_1,\epsilon_1)$ defined by
\begin{align}\label{eq:hatx1}
 \begin{cases}
  x =r_1,\\
  y =r_1^\kappa y_1,\\
  z =r_1^\kappa z_1,\\
  \epsilon =r_1 \epsilon_1.
 \end{cases}
\end{align}
We used the $(r_1,y_1,z_1)$-coordinates in the proof of Lemma \ref{lem:Sj} (within $\epsilon_1=0$).
% \begin{remark}\label{rem:ksum2}
%  We emphasize, that {under \assumptionref{ass2} (in particular the property \eqref{hatX}), Lemma \ref{lem:Sj} also holds true for the system \eqref{tildexuv}.}
% % \begin{remark}
% Notice also that for \eqref{XYZ}, \eqref{hatX} can actually be replaced with the following stronger result: The functions
%  \begin{align}
% \widetilde r_1^{-3k+2} \widetilde X(\widetilde r_1,\widetilde r_1^k \widetilde y_1,\widetilde r_1^k \widetilde z_1,\widetilde \epsilon_1),\quad X=R,S,T,\nonumber %\label{eq:hatX}
% \end{align}
% are each analytic functions with respect to $(\widetilde r_1^{k-1},\widetilde y_1, \widetilde z_1,\widetilde \epsilon_1)$ (in a neighborhood of $(0,0,0,0)$). This means that in the $\breve x=1$-chart the system can be written as an analytic system with respect to $(r_1^{k-1},y_1,z_1,\epsilon_1)$.  However, we will not impose this simplifying property in \assumptionref{ass2} on our generalized system \eqref{tildexuv}. In fact, as already mentioned in Remark \ref{rem:ksum}, the condition \eqref{hatX} in \assumptionref{ass2} (for $\epsilon_1=0$) is also not essential (we can use \cite{ksum} instead for the unperturbed manifolds). We include the condition \eqref{hatX} nevertheless since then the main result, Theorem \ref{thm:main}, only requires one single nondegeneracy condition. Similarly, condition \eqref{xprop} is also not essential but primarily included for simplicity.
% % \end{remark}
% \end{remark}
Notice that the $\breve \epsilon=1$-chart with chart-specific coordinates $(x_2,y_2,z_2,r_2)$, defined by
\begin{align*}
 \begin{cases}
  x =r_2x_2,\\
  y =r_2^\kappa y_2,\\
  z =r_2^\kappa z_2,\\
  \epsilon =r_2,
 \end{cases}
\end{align*}
corresponds to \eqref{xyzx2y2z21} upon eliminating $r_2$. There is a diffeomorphic change of coordinates between the $\breve x=1$-chart and the $\breve \epsilon=1$-chart for $x_2\ne 0$, defined by
\begin{align}\label{eq:cc}
 \begin{cases}
  r_1 =r_2 x_2,\\
  y_1 =x_2^{-\kappa} y_2,\\
  z_1 =x_2^{-\kappa} z_2,\\
  \epsilon_1 =x_2^{-1}.
 \end{cases}
\end{align}
In its most common form, blow-up is used to gain hyperbolicity of nilpotent singularities. In such situations, desingularization, through division of the pull-back vector-field by some power of $r$, plays a crucial role. In the present work, $(x,y,z)=(0,0,0)$ is not nilpotent for $\epsilon=0$ and we will not use desingularization. Our eigenvalues will be  $0,\pm i$ before and after blowup. (This suggest that our approach can also be used for maps.)

%
% \section{Scalar polynomial vector fields in real and imaginary time}
% In this section, we consider
% \begin{align}\label{eq:realtime}
%  \dot x_2 &={Q}(x_2),
% \end{align}
% and
% \begin{align}\label{eq:imagtime}
%  \dot x_2 &=i {Q}(x_2),
% \end{align}
% with $\dot{()}=\frac{d}{ds}$ with $s\in \mathbb R$. These two systems can be understood as the reduced problem \eqref{reduced}, repeated here for convinience
% \begin{align*}
%  \frac{dx_2}{dt} ={Q}(x_2),
% \end{align*}
% in real time $t=s\in \mathbb R$ and imaginary time $t=is\in i \mathbb R$, respectively.
%
\begin{remark}
% \note{where does this go?}
In our proof of Theorem \ref{thm:main}, we will make use of trajectories of \eqref{realtime} and \eqref{imagtime}. In particular, regular trajectories of \eqref{realtime} and \eqref{imagtime} will constitute so-called elliptic (Stokes) respectively hyperbolic (anti-Stokes) integration curves, see \cite{Berry1991,dingle1973a,hayes2016a}. To see this, let $x_2(s)$ be a solution of \eqref{imagtime}. Then if we ignore the higher terms in \eqref{x2y2z2slow}, we find that
\begin{equation}\label{eq:hyp}
\begin{aligned}
 \frac{dy_2}{ds}  &=  i\left( \epsilon^{1-\kappa}z_2 -  \frac12 y_2 {Q}'(x_2(s))\right),\\
 \frac{dz_2}{ds}  &= i\left( -\epsilon^{1-\kappa}y_2 -  \frac12 z_2  {Q}'(x_2(s))\right),
\end{aligned}
\end{equation}
which can be integrated:
% with the general solution
\begin{align}\label{eq:hypsol}
 \begin{pmatrix}
  y_2\\
  z_2
 \end{pmatrix}(s) = \left(\frac{{Q}(x_{2}(0)) }{{Q}(x_{2}(s))}\right)^{\frac12} \exp(i\epsilon^{1-\kappa} \Omega s)\begin{pmatrix}
  y_2\\
  z_2
 \end{pmatrix}(0),
\end{align}
where
\begin{align*}
 \Omega = \begin{pmatrix}
           0 & 1\\
           -1 & 0
          \end{pmatrix},\quad \exp(i\Omega s) =  \begin{pmatrix}
  \frac12 (\e^s+\e^{-s}) & \frac{i}{2}(\e^s-\e^{-s})\\
   -\frac{i}{2}(\e^s-\e^{-s}) & \frac12 (\e^s+\e^{-s})
 \end{pmatrix}.
\end{align*}
Consequently, \eqref{hyp} has
 exponentially decaying and exponentially growing solutions as $\epsilon\rightarrow 0$ for $s\ne 0$. %\fbox{need to correct this}
 To obtain the ${Q}$-dependent factor of the matrix exponential in \eqref{hypsol}, we have used that
\begin{align}\label{eq:logQk}
 \int \frac{{Q}'(x_2)}{{Q}(x_2)}dx_2 = \log {Q}(x_2).
\end{align}

Similarly, if $x_2(s)$ is a solution of \eqref{realtime}, then
\begin{equation}\label{eq:ell}
\begin{aligned}
 \frac{dy_2}{ds}  &=   \epsilon^{1-\kappa}z_2 -  \frac12 y_2 {Q}'(x_2(s)),\\
 \frac{dz_2}{ds}  &=  -\epsilon^{1-\kappa}y_2 -  \frac12 z_2  {Q}'(x_2(s)),
\end{aligned}
\end{equation}
have solutions with fast oscillations as $\epsilon\rightarrow 0$.
\end{remark}

    \section{Stable and unstable manifolds in the $\breve \epsilon=1$-chart}\label{sec:exit}
In this section, we consider \eqref{x2y2z2slow}, repeated here for convenience
\begin{equation}\label{eq:x2y2z2new2}
\begin{aligned}
 { x}_2'  &={Q}(x_2)+\epsilon^{2(\kappa-1)} F_2( x_2, y_2, z_2,\epsilon),\\%\epsilon^{\kappa-1} \sum_{n=0}^{\lfloor \frac{\kappa}{2}\rfloor} \frac{1}{2^{2n} n!^2}{Q}^{(2n)}( x_2) ((\epsilon^{\kappa-1} y)^2+(\epsilon^{\kappa-1} z)^2)^n+\epsilon^{2(\kappa-1)}\widetilde R_{2\kappa-1,2},\\
 \epsilon^{\kappa-1}  { y}_2'  &=  z_2 -  \frac12 y_2 \epsilon^{\kappa-1} {Q}'(x_2) +\epsilon^{2(\kappa-1)} G_2( x_2, y_2, z_2,\epsilon),\\
 \epsilon^{\kappa-1}  z_2'  &= - y_2 -  \frac12 z_2 \epsilon^{\kappa-1} {Q}'(x_2) +\epsilon^{2(\kappa-1)} H_2( x_2, y_2, z_2,\epsilon).
%   \dot{ y}_2  &=  z_2 -{ y_2} \epsilon^{\kappa-1}\sum_{n=0}^{\lfloor \frac{\kappa-1}{2}\rfloor } \frac{1}{2^{2n+1}(n+1)! n!}{Q}^{(2n+1)}((\epsilon^{\kappa-1} y)^2+(\epsilon^{\kappa-1} z)^2)^n+\epsilon^{\kappa-1}\widetilde S_{2\kappa-1,2},\\
%   \dot{ z}_2  &= - y_2 -{ z}_2\epsilon^{\kappa-1}\sum_{n=0}^{\lfloor \frac{\kappa-1}{2}\rfloor } \frac{1}{2^{2n+1}(n+1)! n!}{Q}^{(2n+1)}((\epsilon^{\kappa-1} y)^2+(\epsilon^{\kappa-1} z)^2)^n+\epsilon^{\kappa-1} \widetilde T_{2\kappa-1,2},
%   \dot \epsilon &=0.
\end{aligned}
\end{equation}
By \assumptionref{ass1}, $q^j$, $j\in \{1,\ldots,\kappa\}$, define $\kappa$-many real hyperbolic singularities $\mathbf{q}^j(\epsilon)$ of \eqref{x2y2z2new2} for all $0<\epsilon\ll 1$, see
Lemma \ref{lem:Wuspj}. Recall also that
the singularity $\mathbf{q}^j(\epsilon)$ of \eqref{x2y2z2new2} has a (complex) one-dimensional invariant manifold
 \begin{align}\label{eq:Wsu}
 \mathbf{W}^{\sigma^j}(\mathbf{q}^j(\epsilon))\quad \mbox{where}\quad
 \begin{cases} 
  \sigma^j=s & \mbox{ if $j$ is odd},\\
 \sigma^j = u & \mbox{ if $j$ is even},
 \end{cases}
 \end{align}
 for all $0<\epsilon\ll 1$. In the following, we will provide a description of these manifolds as graphs $(y_2,z_2)=m_2^j(x_2,\epsilon)$ over $x_2$. These manifolds are solutions of the invariance equation:
\begin{align}\label{eq:invman0}
 \left(\Omega -\frac12 \epsilon^{\kappa-1} Q'(x_2)\operatorname{Id}\right)\begin{pmatrix}
                                       y_2\\
                                       z_2
                                      \end{pmatrix}+\epsilon^{2(\kappa-1)} \begin{pmatrix}
                                      G_2\\
                                      H_2\end{pmatrix}=\epsilon^{\kappa-1} \left(Q(x_2)+\epsilon^{2(\kappa-1)} F_2 \right) \frac{d}{dx_2}\begin{pmatrix}
                                       y_2\\
                                       z_2
                                      \end{pmatrix},
\end{align}
where $W_2=W_2(x_2,y_2,z_2,\epsilon)$, $W=F,G,H$, and
\begin{align*}
\operatorname{Id}=\operatorname{diag}(1,1)\in \mathbb C^{2\times 2}.
\end{align*}
Now, we recall the following:
\begin{lemma}\label{lem:formal}
There exists a unique formal series solution of \eqref{invman0} of the form
\begin{align}\label{eq:m2n}
  (y_2,z_2) = \sum_{\alpha=2(\kappa-1)}^\infty  m_{2,\alpha}( x_2)\epsilon^\alpha,
\end{align}
with $ m_{2,\alpha}:\mathbb C\rightarrow \mathbb C^2$ real-analytic for all $\alpha\in \mathbb N$, $\alpha\ge 2(\kappa-1)$.
\end{lemma}
\begin{proof}
 We follow \cite[Lemma 3.2]{bkt}: Since $W=W(x,y,z,\epsilon),W=F,G,H,$ are assumed to be real-analytic on $(x,y,z,\epsilon)\in \mathcal B_\xi^4\subset \mathbb C^4$, for some $\xi>0$ small enough, we have that for any $\varrho>0$, there is an $\epsilon_0>0$ such that  $F_2$, $G_2$ and $H_2$ are real-analytic on $(x_2,y_2,z_2,\epsilon) \in B_\varrho^3 \times (-\epsilon_0,\epsilon_0)$, recall \eqref{W2defn}. But then \cite[Proposition 2.1]{de2020a} gives the existence and uniqueness of the formal series \eqref{m2n} with $ m_{2,\alpha}:B_\varrho \rightarrow \mathbb C^2$ real-analytic for all $\alpha\in \mathbb N$, $\alpha\ge 2(\kappa-1)$, and any $\varrho>0$. The fact that the series starts with $\alpha=2(\kappa-1)$ follows directly from \eqref{invman0}. %The result therefore follows.
\end{proof}
 We now fix any compact neighborhood $\mathcal X^j\subset \mathbb C $ of $q^j\in \mathbb C$ such that $\mathcal X^j$ for $j$ even (odd, respectively) is a subset of the basin of attraction of $q^j$ for the forward (backward, respectively) flow of the reduced problem $x_2'=Q(x_2)$ (associated with \eqref{x2y2z2new2} on $(y_2,z_2)=(0,0)$).

%  and that the flow of
% \begin{align}
%  \frac{dx_2}{ds_P} = \frac{{Q}(x_2)}{1+\vert x_2\vert^{\kappa-1}},\quad x_2\in \mathbb C,\label{eq:realtime_p}
% \end{align}
% is complete. We let $\phi_{s_P}:\mathbb C\rightarrow \mathbb C$, $s_P\in \mathbb R$, denote the flow. %(Due to the division by the nonanalytic function $1+\vert x_2\vert^{\kappa-1}$, this flow is not analytic. However, this will not be important here.)
% The following result provides a description of the local versions of \eqref{Wsu} in compact domains of the complex $x_2$-plane.
\begin{proposition}\label{prop:Wuspj2}
%  Fix any $\tau>0$, any $j\in \{1,\ldots,\kappa\}$ and consider a small complex neighborhood $U_j$ of $q^j$ so that $q^l\notin U_j$ for all $l\ne j$.
%  Let $\mathcal X^j$ denote the
%  \begin{align*}
%   \begin{cases}
%    \mbox{backward flow of $U_j$ if $j$ is odd},\\
%    \mbox{forward flow of $U_j$ if $j$ is even},
%   \end{cases}
%  \end{align*}
%  see \figref{realtime} (a) for an illustration,
% under the finite time flow
% \begin{align*}
% \phi_{s_P},\quad s_P\in \begin{cases} [-\tau,0] & \mbox{ if $j$ is odd},\\
%                                         [0,\tau] & \mbox{ if $j$ is even},
%                    \end{cases}
%                    \end{align*}
%  of \eqref{realtime_p}. Then
Consider any $j\in \{1,\ldots,\kappa\}$. Then $\mathbf{W}_{loc}^{\sigma^j}(\mathbf{q}^j(\epsilon))$, with $\sigma^j$ given by \eqref{Wsu}, is a graph over $x_2\in \mathcal X^j$:
 \begin{align}\label{eq:m2jgraph}
 \begin{pmatrix}
  y_2\\
  z_2 
 \end{pmatrix}=m^j_2(x_2,\epsilon),\quad x_2\in \mathcal X^j,
 \end{align}
 for all $\epsilon\in [0,\epsilon_0)$, $\epsilon_0>0$ sufficiently small.
Here $m^j_2(\cdot,\epsilon)\,:\,\mathcal X^j\rightarrow \mathbb C^2$ is analytic for all $\epsilon\in [0,\epsilon_0)$, being $C^\infty$ with respect $\epsilon$. In particular, $m^j_2$ has an asymptotic expansion:
\begin{align}\label{eq:En2}
 m^j_2(x_2,\epsilon) \sim \sum_{\alpha=2(\kappa-1)}^\infty m_{2,\alpha}(x_2) \epsilon^\alpha\quad \mbox{for}\quad \epsilon\rightarrow 0,
\end{align}
for $x_2\in \mathcal X^j$, where each $m_{2,\alpha}:\mathbb C\rightarrow \mathbb C^2$, $\alpha\in\mathbb N_0$, is real-analytic and independent of $j$. %and $m_{0,2} = (0,0)$.
\end{proposition}
\begin{proof}
% The existence of the asymptotic series \eqref{En2} is standard, see e.g. \cite{de2020a}.
The proof follows \cite[Proposition 3.3]{bkt}.
We fix $j$ odd so that $Q'(q^j)<0$ and so that $x_2=q^j$ is asymptotically stable for the reduced problem:
\begin{align*}
 x_2' = Q(x_2),\quad x_2\in \mathbb C.
\end{align*}
We will therefore describe $\mathbf{W}^s_{loc}(\mathbf q^j(\epsilon))$ as the graph \eqref{m2jgraph}.
The case with $j$ even (and $\mathbf{W}^u_{loc}(\mathbf q^j(\epsilon))$) can be studied in the same way upon changing the direction of time. We then write
\begin{align*}
 \begin{pmatrix}
  x_2\\
  y_2\\
  z_2
 \end{pmatrix} =: \mathbf{q}^j(\epsilon) + \begin{pmatrix}
  \widetilde x_2\\
  \widetilde y_2\\
  \widetilde z_2
 \end{pmatrix},
\end{align*}
and define
\begin{align*}
\zeta_2:= \widetilde x^{-1} \begin{pmatrix}
                          \widetilde y_2\\
                          \widetilde z_2
                         \end{pmatrix}.
                         \end{align*}
A simple calculation shows that this brings \eqref{x2y2z2new2} into the following form
\begin{equation}\label{eq:x2zeta2}
\begin{aligned}
 \widetilde x_2' &= \widetilde Q^j(\widetilde x_2,\epsilon) \widetilde x_2,\\
 \epsilon^{\kappa-1} \zeta_2' &=\left(\Omega + \epsilon^{\kappa-1} W_{21}^j(\widetilde x_2,\epsilon)\right) \widetilde \zeta_2 + \epsilon^{2(\kappa-1)} W_{20}^j(\widetilde x_2, \zeta_2,\epsilon),
\end{aligned}
\end{equation}
where
\begin{align*}
\Omega = \begin{pmatrix}
          0 & 1\\
          -1 & 0
         \end{pmatrix},
 \end{align*}
after dividing the right hand side by a quantity $1+\mathcal O(\epsilon^{2(\kappa-1)})$ which is nonzero and real-analytic for all $q^j+\widetilde x_2\in \mathcal X^j$, $\vert  \zeta_2\vert \le c$, with $c>0$ fixed, and all $0\le\epsilon\ll 1$. Here we have used that $Q(x_2)\ne 0$ for all $x_2\in \mathcal X^j\setminus \{q^j\}$ and that $Q'(q^j)\ne 0$. All quantities are real-analytic. In particular,
\begin{align*}
 \widetilde Q^j(\widetilde x_2,0) = \widetilde x_2^{-1} Q(q^j+\widetilde x_2)\rightarrow Q'(q^j)\quad \mbox{for}\quad \widetilde x_2\to 0,
\end{align*}
and
\begin{align*}
 W_{21}^j(\widetilde x_2,0) =
\left(-\frac12 Q'(q^j+\widetilde x_2) -\widetilde x_2^{-1} Q(q^j+\widetilde x_2)\right) \operatorname{Id},\quad \operatorname{Id}=\operatorname{diag}(1,1)\in \mathbb C^{2\times 2};
\end{align*}
specifically
\begin{align}\label{eq:cond}
 W_{21}^j(0,0) = -\frac32 Q'(q^j)\operatorname{Id}.
\end{align}
The eigenvalues of $W_{21}^j(0,0)$ are therefore strictly positive by assumption.

% We now drop the tildes on $\zeta_2$, $B_j$ and $W_j$.

As in Lemma \ref{lem:formal}, we have a unique formal series solution of \eqref{x2zeta2} of the form
\begin{align}\label{eq:zeta2n}
  \zeta_2 = \sum_{\alpha=2(\kappa-1)}^\infty \zeta_{2,\alpha}(\widetilde x_2)\epsilon^\alpha.
\end{align}
We therefore introduce ${\widetilde \zeta}_2$ by
\begin{align*}
 \zeta_2 = \sum_{\alpha=2(\kappa-1)}^{2N} \zeta_{2,\alpha}(\widetilde x_2)\epsilon^\alpha+\epsilon^N \widetilde{ \zeta}_2,
\end{align*}
with  $N\ge \kappa-1$. This brings the equations into the form
\begin{equation}\label{eq:x2zeta2f}
\begin{aligned}
 \widetilde x_2' &= \widetilde Q^j(\widetilde x_2,\epsilon) \widetilde x_2,\\
 \epsilon^{\kappa-1} \widetilde{\zeta}_2' &=\left(\Omega + \epsilon^{\kappa-1} \widetilde W_{21}^{j,N} (\widetilde x_2,\epsilon)\right) \widetilde \zeta_2 + \epsilon^{N} {\widetilde W}_{20}^{j,N}(\widetilde x_2,\widetilde{ \zeta}_2,\epsilon),
\end{aligned}
\end{equation}
with $\widetilde W_{21}^{j,N} (0,0)=-\frac32 Q'(q^j)\operatorname{Id}$ for any $N\ge \kappa-1$. This follows from a simple calculation.

We will prove the statement by working on \eqref{x2zeta2f}. Define $\widetilde{\mathcal X}^j$ as follows: $\widetilde x_2\in \widetilde{\mathcal X}^j$ if and only if $q^j+\widetilde x_2\in \mathcal X^j$. %We therefore take $\widetilde x_{20}$ such that $\widetilde x_{20}\in \widetilde{\mathcal X}^j$.
Then, since $\mathcal X^j$ is compact and $j$ is odd, the solution $\widetilde x_2(t,\widetilde x_{20},\epsilon)$ of the initial value problem
$$\widetilde x_2'=\widetilde Q^j(\widetilde x_2,\epsilon)\widetilde x_2,\quad \widetilde x_2(0)=\widetilde x_{20}\in \widetilde{\mathcal X}^j,$$ satisfies $\widetilde x_2(t,\widetilde x_{20},\epsilon)\rightarrow 0$ for $t\rightarrow \infty$, uniformly for $\tilde x_{20}\in \widetilde{\mathcal X}^j$, $0\le \epsilon\ll 1$. We then solve for $\widetilde \zeta_{2}(t,\widetilde x_{20},\epsilon)$ such that $(\widetilde x_{20},\widetilde \zeta_{2}(0,\widetilde x_{20},\epsilon))$ is a point on the stable manifold through the usual fixed-point formulation:
\begin{align*}
\widetilde \zeta_2(t,\widetilde x_{20},\epsilon) = -\int_t^\infty \widetilde{\mathcal M}_2^{j,N}(t,\widetilde x_{20},\epsilon))\widetilde{\mathcal M}_2^{j,N}(s,\widetilde x_{20},\epsilon)^{-1} \epsilon^N \widetilde W_{20}^{j,N}(\widetilde x_2(s,\widetilde x_{20},\epsilon), \widetilde \zeta_{2}(s,\widetilde x_{20},\epsilon),\epsilon) ds,
%  \epsilon^{k-1} \widetilde{\zeta}_2' &=\left(\Omega + \epsilon^{k-1} {\widetilde B}_j(\widetilde x_2(t,\widetilde x_{20}),\epsilon)\right) \widetilde \zeta_2 + \epsilon^{N} {\widetilde W}_j(\widetilde x_2(t,\widetilde x_{20}),\widetilde{ \zeta}_2,\epsilon),\quad \widetilde z_2(0)=\widetilde \zeta_{20},
\end{align*}
for $t\ge 0$,
see e.g. \cite[Theorem 5.3]{meiss2007a}. Here $\widetilde{\mathcal M}_2^{j,N}(t,\widetilde x_{20},\epsilon)$ is the fundamental matrix associated with the linear problem
\begin{align*}
 \epsilon^{\kappa-1} \widetilde{\zeta}_2' &=\left(\Omega + \epsilon^{\kappa-1} {\widetilde W}_{21}^{j,N}(\widetilde x_2(t,\widetilde x_{20},\epsilon),\epsilon)\right) \widetilde \zeta_2,
\end{align*}
satisfying $\widetilde{\mathcal M}_2^{j,N}(0,\widetilde x_{20},\epsilon)=\operatorname{Id}$.
In fact, using \eqref{cond}, we easily obtain the following bound
\begin{align*}
 \Vert \widetilde{\mathcal M}_2^{j,N}(t,\widetilde x_{20},\epsilon)\widetilde{\mathcal M}_2^{j,N}(s,\widetilde x_{20},\epsilon)^{-1}\Vert \le c_1^{-1} \e^{c_2 (t-s)},
\end{align*}
for $c_1>0,c_2>0$ both small enough, for any $s\ge t$, $q^j+\widetilde x_2\in \mathcal X^j$, and all $0<\epsilon\ll 1$, see also \cite[Lemma 3.2]{bkt} (based upon \cite[Proposition 1, p. 34]{coppel1978a}).
The existence of a bounded solution $t\mapsto \widetilde \zeta_{2}(t,\widetilde x_{20},\epsilon), t\ge 0$, then proceeds completely analogously to \cite[Theorem 5.3]{meiss2007a} for $0<\epsilon\ll 1$. In particular, the desired graph is given by $\widetilde \zeta_2=\widetilde \zeta_{2}(0,\widetilde x_{2},\epsilon)$ over $q^j+\widetilde x_2\in \mathcal X^j$ for all $0\le \epsilon\ll 1$.
The remaining statements, including the existence of the asymptotic series \eqref{En2}, follows from \eqref{zeta2n}, the arbitrariness of $N$ and the uniqueness and smoothness of $\mathbf{W}^s_{loc}(\mathbf q^j(\epsilon))$ for $\epsilon\in (0,\epsilon_0)$, see \cite[Proposition 3.3]{bkt}.

\end{proof}
An important consequence of Proposition \ref{prop:Wuspj2} is that we have control of $\mathbf{W}_{loc}^{\sigma^j}(\mathbf{q}^j(\epsilon))$ along the elliptic paths $\mathcal E^l$ that connect to $q^j$ (within compact sets of the complex $x_2$-plane), recall Lemma \ref{lem:realtime} and see \figref{realtime}. These paths are invariant manifolds of hyperbolic singularities $\breve e^l\in \mathbb S^1$ at infinity $x_2=\infty$, see Lemma \ref{lem:realtime} (assertion \ref{it:real6}). But $\breve e^l$ also define directions (see \eqref{ehatlangles}) of $x$ upon which the domain of the unperturbed manifolds $(y,z)=x^\kappa\psi^l(x)$, $x\in S^l$,  are centered, recall Lemma \ref{lem:Sj}. However, $x$ and $x_2$ are related by $x=\epsilon x_2$ and to connect the two regimes $x=\mathcal O(\epsilon)$ and $x=\mathcal O(1)$ we use the directional chart $\breve x=1$, see \eqref{hatx1}.

    \section{Stable and unstable manifolds in the $\breve x=1$-chart}\label{sec:exitt}
We consider the extended system defined by \eqref{XYZ} and $\dot \epsilon=0$ and apply the change of coordinates $(r_1,y_1,z_1,\epsilon_1)\mapsto (x,y,z,\epsilon)$ given by \eqref{hatx1}. This gives
\begin{equation}\label{eq:K1eqns}
\begin{aligned}
 \dot r_1 &= r_1^\kappa P(1,\epsilon_1) ,\\
 \dot y_1 &=z_1 - \frac12 r_1^{\kappa-1}{{P}'_x(1,\epsilon_1)} y_1-\kappa r_1^{\kappa-1} P(1,\epsilon_1)y_1 + r_1^{2(\kappa-1)} G_1(r_1,y_1,z_1,\epsilon_1),\\
 \dot z_1 &=-y_1- \frac12 r_1^{\kappa-1}{{P}'_x(1,\epsilon_1)} z_1-\kappa r_1^{\kappa-1}P(1,\epsilon_1) z_1 + r_1^{2(\kappa-1)} H_1(r_1,y_1,z_1,\epsilon_1),\\
 \dot \epsilon_1 &=- r_1^{\kappa-1}\epsilon_1P(1,\epsilon_1) ,
%  \dot r_1 &= -r_1^\kappa ,\\
%  \dot y_1 &=-\frac{z_1}{P(1,\epsilon_1)} + \frac12 r_1^{\kappa-1}\frac{{P}'_x(1,\epsilon_1)}{P(1,\epsilon_1)} y_1+k r_1^{k-1}y_1 + r_1^{2(k-1)} G_1(r_1,y_1,z_1,\epsilon_1),\\
%  \dot z_1 &=\frac{y_1}{P(1,\epsilon_1)} + \frac12 r_1^{k-1}\frac{{P}'_x(1,\epsilon_1)}{P(1,\epsilon_1)} z_1+k r_1^{k-1}z_1 + r_1^{2(k-1)} H_1(r_1,y_1,z_1,\epsilon_1),\\
%  \dot \epsilon_1 &=r_1^{k-1}\epsilon_1,
%   \dot z_1 &=-y_1 - \frac12 r_1^{\kappa -1}{P}'_x(1,\epsilon_1)z_1-\kappa  r_1^{\kappa -1} P(1,\epsilon_1) z_1 + r_1^{2(\kappa -1)} \widetilde H_1(r_1,y_1,z_1,\epsilon_1),
\end{aligned}
\end{equation}
% \note{replace $Z_1()$ by $Z_1(r_1,r_1^{\kappa -1}y_1,r_1^{\kappa -1}z_1,\epsilon_1)$, $Z=S,T$}
where $G_1$ and $H_1$ are locally defined real-analytic functions. Here we have used \eqref{homoP}:
\begin{equation}\label{eq:Pk1}
\begin{aligned}
 P(1,\epsilon_1)& = r_1^{-\kappa } {P}(r_1,r_1\epsilon_1) = -1+\sum_{\alpha=0}^{\kappa -1} \epsilon_1^{\kappa -\alpha} a_\alpha,\\
 {P}'_x(1,\epsilon_1)&=r_1^{-\kappa +1} {P}_x'(r_1,r_1\epsilon_1) = -\kappa +\sum_{\alpha=1}^{\kappa -1} \epsilon_1^{\kappa -\alpha} a_\alpha \alpha.
\end{aligned}
\end{equation}
Since we are only interested in invariant manifold solutions, as graphs over $r_1$ and $\epsilon_1$ with $(r_1,\epsilon_1)\in B_\delta^2\subset \mathbb C^2$, we have in the derivation of \eqref{K1eqns} also divided the right hand side by $$1+r_1^{2(\kappa -1)} P_1(1,\epsilon_1)^{-1} F_1(r_1,y_1,z_1,\epsilon_1),$$ which is nonzero and real-analytic for all $\vert r_1\vert$ and $\vert \epsilon_1\vert$ small enough. We have also redefined $G_1$ and $H_1$, recall \eqref{R1S1T1}. %We use $B_\delta^2\subset \mathbb C^2$ to denote the ball of radius $\delta>0$ centered at the origin.
 
We now seek to normalize \eqref{K1eqns}. For this purpose, we first look for two separate formal solutions; formal here referring to formal series. Let $\zeta_1=(y_1,z_1)^{\operatorname{T}}$, then invariant manifold solutions $\zeta_1=\zeta_1(r_1,\epsilon_1)$ of \eqref{K1eqns} satisfy a PDE of the form:
\begin{align}\label{eq:pdew}
 -r_1^\kappa \frac{\partial \zeta_1}{\partial r_1}  +r_1^{\kappa-1}\epsilon_1 \frac{\partial \zeta_1}{\partial \epsilon_1}= \Omega \zeta_1 + r_1^{\kappa -1}R_1(r_1,\zeta_1,\epsilon_1),
\end{align}
where $\Omega=\begin{pmatrix} 0&1\\
          -1 &0
         \end{pmatrix}$, $R_1(r_1,0,\epsilon_1)=\mathcal O(r_1^{\kappa -1})$. Here $R_1=R_1(r_1,\zeta_1,\epsilon_1)\in \mathbb C^2$ is locally well-defined and real-analytic, and can obviously be expressed in terms of $P$, $G_1$ and $H_1$, but the details of this will not be important.

\begin{lemma}\label{lem:En1}
There exists a unique formal series solution of \eqref{pdew} of the form
\begin{align}\label{eq:formal1}
 \zeta_1 = \sum_{\alpha=2(\kappa -1)}^\infty m_{1,\alpha}(\epsilon_1)r_1^\alpha,
\end{align}
where each $m_{1,\alpha}:B_\delta\rightarrow \mathbb C^2$ is real-analytic for $\delta>0$ small enough.
\end{lemma}
\begin{proof}
 We write  \eqref{pdew} in the form
 \begin{align*}
  \zeta_1 = r_1^{\kappa -1} \Omega^{-1}\left(-r_1 \frac{\partial \zeta_1}{\partial r_1} +\epsilon_1 \frac{\partial \zeta}{\partial \epsilon_1} - R_1(r_1,\zeta_1,\epsilon_1)\right).
 \end{align*}
The result can then be proven as the proof of \cite[Proposition 2.1]{de2020a} (by the virtue of the $r_1^{\kappa -1}$-factor on the right hand side). The sum in \eqref{formal1} starts with $\alpha=2(\kappa -1)$ since $R_1(r_1,0,\epsilon_1)=\mathcal O(r_1^{\kappa -1})$.

\end{proof}

\begin{lemma}\label{lem:Pnl}
Fix $\chi>0$ and $\delta>0$ small enough and recall that $\breve e^l=\frac{\pi}{\kappa -1}(l-1)$. Then for any $l\in \{1,\ldots,2(\kappa-1)\}$ there exists a formal series solution of \eqref{pdew} of the form
\begin{align}\label{eq:wepsexp}
 \zeta_1 = \sum_{\alpha=0}^\infty \psi_{1,\alpha}^l(r_1)\epsilon_1^\alpha,
\end{align}
where each $\psi_{1,\alpha}^l:S^l\rightarrow \mathbb C^2$ is analytic on the local sector $$S^{l}=S(\breve e^l,\delta,\eta),\quad \eta=\frac{\pi}{\kappa -1}+\chi,$$ recall \eqref{Sector}.

Moreover, for every fixed $\alpha\in \mathbb N_0$, $\psi_{1,\alpha}^l$, $l\in\{1,\ldots,2(\kappa -1)\}$, is the $(\kappa-1)$-sum in the direction $\breve e^l$, respectively, of a Gevrey-$\frac{1}{\kappa-1} $ series
\begin{align}
 \psi_{1,\alpha}^l(r_1) \sim_{\frac{1}{\kappa-1}} \sum_{\beta=2(\kappa -1)}^{\infty} \psi_{1,\alpha,\beta}r_1^\beta.\label{eq:psinl}
\end{align}
In particular, the series is independent of $l$.

% whereas $E_n:B(\delta_n)\rightarrow \mathbb  R$ will be real-analytic in a full neighborhood $B(\delta_n)$ of $\epsilon_1=0$. 

\end{lemma}
\begin{proof}
% We first write $ \zeta = \sum_{n=0}^\infty \psi_{\alpha}(r)\epsilon^\alpha$ as a formal series with $r_1$-dependent coefficients.
% \end{align}
By proceeding as in \cite[Lemma 4.3]{bkt}, we obtain the following equations for $\psi_\alpha$, $\alpha\in \mathbb N_0$:
\begin{align}
 -r_1^{\kappa } \psi_{1,0}' = \Omega \psi_{1,0}+r_1^{\kappa -1} R_1(r_1,\psi_{1,0},0),\label{eq:phi0eqn}
\end{align}
and
\begin{equation}\label{eq:phin1eqn}
\begin{aligned}
 -r_1^{\kappa } \psi_{1,\alpha}' &= \Omega\psi_{1,\alpha} +r_1^{\kappa -1}\left( \frac{\partial R_1}{\partial \zeta_1}(r_1,\psi_{1,0},0) -\alpha \operatorname{Id}\right)\psi_{1,\alpha} \\
 &+Q_{1,\alpha-1}(r_1,\psi_{1,0},\ldots,\psi_{1,\alpha-1})\quad \forall\,\alpha\in \mathbb N,
\end{aligned}
\end{equation}
where $Q_{1,\alpha-1}=Q_{1,\alpha-1}(r_1,\psi_{1,0},\ldots,\psi_{1,\alpha-1})$ is real-analytic and satisfies $Q_{1,\alpha-1}(0,\ldots,0)=0$ for all $\alpha\in \mathbb N$. For simplicity, we have here dropped the $l$-superscript on $\psi_{1,\alpha}$, which refers to the direction $\breve e^l$.
% where
% \begin{align*}
%  Q_{n-1}(r,\psi_0,\ldots,\psi_{n-1})=r^{k-1} \sum_{\alpha=0}^\infty \Bigg( F_{\alpha,0}  M_{\alpha,n-1}+\sum_{\beta=0}^{n-1} F_{\alpha,n-\beta} (\zeta^\alpha)_{\beta}\Bigg),
% \end{align*}
% is analytic.
Then from \cite[Corollary 4.2]{ksum}, we obtain $\psi_{1,\alpha} \in S(\breve e^l,\delta_\alpha,\eta)$, with $\delta_\alpha>0$, for each $\alpha\in \mathbb N_0$ and all $l\in \{1,\ldots,2(\kappa -1)\}$. Each $\psi_{1,\alpha}$ is the $(\kappa-1)$-sum of a Gevrey-$\frac{1}{\kappa-1 }$ series \eqref{psinl} (independent of $l$) in the direction $\breve e^l$. To obtain that $\delta_\alpha>0$ is uniformly bounded from below, we use that \eqref{phin1eqn} is a linear first order equation for $\psi_{\alpha}$, $\alpha\in \mathbb N$, which is regular for $r_1\ne 0$. The solution of \eqref{phi0eqn} for $\alpha=0$ is covered by Lemma \ref{lem:Sj}. Therefore by the existence and uniqueness of solutions of linear equations and induction on $\alpha$, we conclude that $\psi_{1,\alpha}=\psi_{1,\alpha}^l$ can be extended to
$S^l=S(\breve e^l,\delta,\eta)$, $\delta=\delta_0$, for all $\alpha\in \mathbb N_0$, as claimed.
% \note{Since each equation for $n\in \mathbb N$ is linear with  respect to $w_n$ we could also refer to Braaksma etc. and maybe even Wasow has it.} 
\end{proof}
\begin{lemma}\label{lem:conj}
 The following holds for any $l\in \{1,\ldots,\kappa \}$, any $r_1\in S^l$ and any $\alpha\in \mathbb N_0$:
 \begin{align*}
  \overline{\psi_{1,\alpha}^l(r_1)} = \psi_{1,\alpha}^{2\kappa -l}(\overline r_1),
 \end{align*}
 with $2\kappa -l$ understood $\operatorname{mod}(2(\kappa -1))$.
\end{lemma}
\begin{proof}
 This is a consequence of the symmetry with respect to conjugation.
\end{proof}

\begin{lemma}\label{lem:Phi0l}
Let $m_{2,\alpha}=m_{2,\alpha}(x_2)$, $x_2\in \mathbb C$, be the sequence of analytic functions in Proposition \ref{prop:Wuspj2}. Then
\begin{align*}
m_{1,\alpha}(\epsilon_1)= \epsilon_1^{\kappa +\alpha}m_{2,\alpha}(\epsilon_1^{-1})\quad \forall\,\epsilon_1>0,\,\alpha\in \mathbb N,
\end{align*}
with $m_{1,\alpha}$ given in Lemma \ref{lem:En1}. Moreover, let $\psi^l:S^l\rightarrow \mathbb C^2$, $l\in\{1,\ldots,2(l-1)\}$, be as in Lemma \ref{lem:Sj}. Then
\begin{align*}
\psi_{1,0}^l=\psi^l\quad \forall\, l\in \{1,\ldots,2(\kappa -1)\},
\end{align*}
with $\psi_{1,0}^l$ given in Lemma \ref{lem:Pnl}.
\end{lemma}
\begin{proof}
 The result follows from \eqref{hatx1} and \eqref{cc} and the uniqueness of the series.
\end{proof}

In the following, we let $\mathcal G(S)\{\zeta_1,\epsilon_1\}$ denote the set of analytic functions $W_1:S\times B_{\xi}^3 \rightarrow \mathbb C^\gamma$, where $S=S(\breve e,\delta,\eta)$ is an open sector,
% and where $B_\xi^3\subset \mathbb C^3$ is the ball of radius $\xi>0$ centered at the origin,
such that
\begin{align*}
 W_1(r_1,\zeta_1,\epsilon_1) = \sum_{\alpha,\beta} W_{1,\alpha,\beta}(r_1)\zeta_1^\alpha\epsilon_1^\beta ,
\end{align*}
with the series being absolutely convergent on a neighborhood of $(\zeta_1,\epsilon_1)=(0,0)$ uniformly with respect $r_1\in S$, and where each $W_{1,\alpha,\beta}$ is the $(\kappa-1) $-sum on $S$ (being continuous on the boundary with $W_{1,\alpha,\beta}(0)=0$ for all ${\alpha,\beta}\in \mathbb N_0$) of a Gevrey-$\frac{1}{\kappa-1}$ series as $r_1\rightarrow 0$. As indicated, we suppress the dependency on $\gamma$, which should be clear from the context. $\mathcal G(S)\{\epsilon_1\}$ (as a set of functions that are independent of $\zeta_1$) is defined similarly. We fix $\chi>0$ and $\delta>0$ small enough. Recall that $\epsilon=r_1\epsilon_1$ and that $\eta=\frac{\pi}{\kappa-1}+\chi$.
\begin{proposition}\label{prop:Hn}
 Consider any $l\in\{1,\ldots,2(\kappa-1)\}$ and let $N\in \mathbb N$, $N\gg 1$. Then there exists a $W_{1}^{l,N}\in \mathcal G(S^l)\{\epsilon_1\}$ such that the blow-up transformation
 \begin{align*}
 (r_1,\widetilde \zeta_1,\epsilon_1)\mapsto (r_1, \zeta_1,\epsilon_1), \quad r_1 \in S^l:=S\left(\breve e^l,\delta,\eta\right),\quad (\widetilde \zeta_1,\epsilon_1) \in B_\xi^2,\quad 0<\xi\ll 1,
 \end{align*}
  defined by 
 \begin{align*}
\zeta_1=W_{1}^{l,N}(r_1,\epsilon_1) + (r_1\epsilon_1)^N \widetilde \zeta_1,
 \end{align*}
brings \eqref{K1eqns} into the following form
\begin{equation}\label{eq:K1eqnsnf}
\begin{aligned}
 \dot r_1 &= r_1^\kappa P(1,\epsilon_1) ,\\
 \dot{\widetilde \zeta}_1 &=\left(\Omega +r_1^{\kappa-1} \widetilde W_{11}^{l,N}(r_1,\epsilon_1)\right)\widetilde \zeta_1+(r_1\epsilon_1)^N \widetilde W_{10}^{l,N}(r_1,\widetilde \zeta_1,\epsilon_1),\\
 \dot \epsilon_1 &=-r_1^{\kappa -1}\epsilon_1P(1,\epsilon_1).
%   \dot z_1 &=-y_1 - \frac12 r_1^{\kappa -1}{P}'_x(1,\epsilon_1)z_1-\kappa  r_1^{\kappa -1} P(1,\epsilon_1) z_1 + r_1^{2(\kappa -1)} \widetilde H_1(r_1,y_1,z_1,\epsilon_1),
\end{aligned}
\end{equation}
Here $\widetilde W_{11}^{l,N}\in \mathcal G(S^l)\{\epsilon_1\}$ and $\widetilde W_{10}^{l,N}\in \mathcal G(S^l)\{\widetilde \zeta_1,\epsilon_1\}$. Moreover,
\begin{align*}
 \widetilde W_{11}^{l,N}(r_1,\epsilon_1) = \left( -\frac12 {{P}'_x(1,\epsilon_1)} -\kappa P(1,\epsilon_1)\right)\operatorname{Id}+\mathcal O(r_1^{\kappa -1}),\quad (r_1,\epsilon_1)\in S^l\times B_\xi,
 \end{align*}
%  \widetilde W_{11}^{l,N}(r_1,\epsilon_1) = \begin{pmatrix}
%                                  \frac12 r_1^{\kappa -1} \frac{{P}'_x(1,\epsilon_1)}{P(1,\epsilon_1)} + \kappa  r_1^{\kappa -1} & - \frac{1}{P(1,\epsilon_1)}\\
%                                  \frac{1}{P(1,\epsilon_1)} &
%                                  \frac12 r_1^{\kappa -1} \frac{{P}'_x(1,\epsilon_1)}{P(1,\epsilon_1)} + \kappa  r_1^{\kappa -1}
%                                 \end{pmatrix}+\mathcal O(r_1^{2(\kappa -1)}),
% \end{align*}
% \fbox{check remainder?}
where $\mathcal O(r_1^{\kappa -1})$ is uniform with respect to $\epsilon_1\in B_\xi$.
\end{proposition}
\begin{proof}
 We follow \cite[Proposition 4.4]{bkt} (which in turn is based upon \cite[Lemma 4.5]{uldall2024a}) and take
 \begin{align*}
  W_{1}^{l,N}(r_1,\epsilon_1):=\sum_{n=0}^{2N-1} \psi_{1,\alpha}^l(r_1)\epsilon_1^n + \sum_{n=2(\kappa -1)}^{2N-1}( m_{1,\alpha}-J^{2N-1}(m_{1,\alpha}))(\epsilon_1) r_1^n,
 \end{align*}
where $J^\gamma(W_1)$, $\gamma\in \mathbb N$, denotes the partial sum of an analytic function $W_1=\sum_{\alpha=0}^\infty W_{1,\alpha} \epsilon_1^\alpha$:
 \begin{align*}
  J^\gamma(W_1)(\epsilon_1): = \sum_{\alpha=0}^\gamma W_{1,\alpha} \epsilon_1^\alpha.
 \end{align*}
 Then by construction $\zeta_1=W_{1}^{l,N}(r_1,\epsilon_1)$ defines an invariant manifold up to order $\mathcal O((r_1\epsilon_1)^{2N})$, i.e. it solves \eqref{pdew} up to remainder terms of order $\mathcal O((r_1\epsilon_1)^{2N})$.
 The result then follows from a simple calculation. 
\end{proof}

% \begin{figure*}[t!]
% 		\centering
% 		\subfigure[]{\includegraphics[width=0.47\textwidth]{./S1S2.pdf}}
% 		\subfigure[]{\includegraphics[width=0.47\textwidth]{./S1S2S3.pdf}}
% 		\caption{The sectors $S_j$ in Lemma \ref{lem:Sj} for $\kappa =2$ (a) and $\kappa =3$ (b). The directions $\breve e^j$ define the direction of $S_j$ whereas $\breve h^j$ corresponds to a $\frac{\pi}{\kappa -1}$ rotation of $\breve e^j$ clockwise. These directions will be important later. }
% 		\label{fig:realtime}
% 	\end{figure*}

% Next, we consider 
\subsection{Extension of the invariant manifolds}
We now use \eqref{K1eqnsnf}, in the following form
\begin{equation}\label{eq:K1eqnsnf2}
\begin{aligned}
 \dot r_1 &= r_1^\kappa P(1,\epsilon r_1^{-1}) ,\\
 \dot{\widetilde \zeta}_1 &=\left(\Omega+r_1^{\kappa-1}\widetilde W_{11}^{l,N}(r_1,\epsilon r_1^{-1})\right)\widetilde\zeta_1+\epsilon^N \widetilde W_{10}^{l,N}(r_1,\widetilde \zeta_1,\epsilon r_1^{-1}),
%  \dot \epsilon_1 &=-r_1^{\kappa -1}\epsilon_1P(1,\epsilon_1).
%   \dot z_1 &=-y_1 - \frac12 r_1^{\kappa -1}{P}'_x(1,\epsilon_1)z_1-\kappa  r_1^{\kappa -1} P(1,\epsilon_1) z_1 + r_1^{2(\kappa -1)} \widetilde H_1(r_1,y_1,z_1,\epsilon_1),
\end{aligned}
\end{equation}
to extend the invariant manifolds from Proposition \ref{prop:Wuspj2} from $x=\mathcal O(\epsilon)$ to $x=\mathcal O(1)$; in comparison with \eqref{K1eqnsnf} we have eliminated $\epsilon_1=\epsilon r_1^{-1}$ through the conservation of $\epsilon=r_1\epsilon_1$, see \eqref{hatx1}.
% \begin{lemma}
%  
% \end{lemma}
For this purpose, we will use the elliptic paths, i.e. solutions $x_2(s)$, $s\in \mathbb R$, of \eqref{realtime}. Notice that if $x_2(s)$, $s\in \mathbb R$, is a solution of \eqref{realtime}, then 
\begin{align*}
 r_1(t) = \epsilon x_2(\epsilon^{\kappa -1} t),
\end{align*}
solves the first equation in \eqref{K1eqnsnf2}.
Define
\begin{align*}
\mathcal A^l(\epsilon):=S^l\setminus B_{\epsilon \xi^{-1}}\subset \mathbb C,\quad l\in \{1,\ldots,\kappa \}.
\end{align*}
We suppose that $r_1(t)\in \mathcal A^l(\epsilon)$ for all $t\in I(\epsilon)\subset \mathbb R$ and focus on $l\in \{1,\ldots,\kappa \}$; the remaining cases $l\in \{\kappa +1,\ldots,2(\kappa -1)\}$ can be obtained by applying conjugation, see Lemma \ref{lem:conj}. We fix $\xi>0$ and $\delta>0$ small enough. %Since \eqref{realtime} has finite time blow-up near $\widehat e_l$, we consider $t\in [0,\epsilon^{-k+1}H(r_1(0),\epsilon)]$, $H(r_1(0),\epsilon)>0$ small enough. \note{need to update this?}

% \note{We need something on $t\in I(\epsilon)$?}
%&Consider \eqref{K1eqnsnf2} with $\widetilde W_{10}^{l,N}=0$, and s
\begin{lemma}
% Let $l\in \{1,\ldots,\kappa \}$ and suppose that $r_1(t)\in \mathcal A^l(\epsilon)$ for all $t\in I(\epsilon)\subset \mathbb R$ and all $0<\epsilon \ll 1$.
%  \note{need to update this}
 Let $\widetilde{\mathcal M}_1^{l,N}(t,r_{1}(0),\epsilon)$ denote the fundamental matrix  associated with the linear problem
 \begin{align*}\dot{\widetilde\zeta}_1 &=\left(\Omega+r_1(t)^{\kappa-1}\widetilde{W}_{11}^{l,N}(r_1(t),\epsilon r_1(t)^{-1})\right)\widetilde \zeta_1,
 \end{align*}
 with $\widetilde{\mathcal M}_1^{l,N}(0,r_{1}(0),\epsilon)=\operatorname{Id}$.
Then
 \begin{align}
 \widetilde{\mathcal M}_1^{l,N}(t,r_{1}(0),\epsilon) = \left(\frac{{P}(r_{1}(0),\epsilon)}{{P}(r_1(t),\epsilon)}\right)^{\frac12} \left(\frac{r_{1}(0)}{r_1(t)}\right)^\kappa  \exp\left(\Omega t \right)\mathcal O(1),\nonumber
\end{align}
for all $t\in I(\epsilon)$.
In particular, there exist a constant $C>0$ so that
\begin{align}
\Vert \widetilde{\mathcal M}_1^{l,N}(t,r_{1}(0),\epsilon) \widetilde{\mathcal M}_1^{l,N}(s,r_{1}(0),\epsilon)^{-1}\Vert \le C\epsilon^{-\frac{3\kappa}{2}}\quad \forall\,0\le s\le t,\,t\in I(\epsilon),\label{eq:stmest1}
\end{align}
for all $0<\epsilon\ll 1$.
% \fbox{update}

\end{lemma}
\begin{proof}
  First, we put 
 \begin{align*}
    \widetilde{\mathcal M}_1^{l,N}(t,r_1(0),\epsilon) = \left(\frac{{P}(r_{1}(0),\epsilon)}{{P}(r_1(t),\epsilon)}\right)^{\frac12} \left(\frac{r_{1}(0)}{r_1(t)}\right)^\kappa  \exp\left(\Omega t \right)\widehat{\mathcal M}_1^{l,N}(t,r_1(0),\epsilon).
 \end{align*}
 Now,
$\vert r_{1}(t)\vert \in [\epsilon \xi^{-1},\delta]$ implies that $\vert \epsilon_1(t)\vert\in [\epsilon \delta^{-1},\xi]$. We therefore find that $\widehat{\mathcal M}_1^{l,N}(t,r_1(0),\epsilon)$ is the fundamental matrix associated with
\begin{align*}
 \frac{d\widehat \zeta_1}{dt} = \mathcal O(r_1^{2(\kappa -1)})\widehat \zeta_1,
\end{align*}
satisfying $\widehat{\mathcal M}_1^{l,N}(0,r_1(0),\epsilon)=\operatorname{Id}$. Here we have used \eqref{logQk} (with ${Q}\mapsto {P}$, $x_2\mapsto x$) and that $\Vert \exp(\Omega t)\Vert \equiv 1$.
But then
\begin{align*}
 \frac{d\widehat \zeta_1}{dr_1} = \mathcal O(r_1^{\kappa -2})\widehat \zeta_1,
\end{align*}
by the chain rule, see \eqref{K1eqnsnf2}. It follows that $\Vert \widehat{\mathcal M}_1^{l,N}(t,r_{1}(0),\epsilon) \widehat{\mathcal M}_1^{l,N}(s,r_{1}(0),\epsilon)^{-1}\Vert$, $0\le s\le t,\, t\in I(\epsilon)$, is uniformly bounded, which in turn then gives the desired estimate \eqref{stmest1}. %The result then follows from a simple calculation upon using that .% upon using that the exponential matrix has norm $1$.
% \begin{align*}
%  \vert \zeta(r_1)\vert\le C \left(\frac{w_{10}}{w_1}\right)^{\frac{3k}{2}}\vert \zeta(r_{10})\vert,
% \end{align*}
% since $P_{k,1}(0)=-1$

\end{proof}

We now consider a section $\Sigma^{l,in}\subset \mathbb C$ defined as a neighborhood of $\mathcal E^l \cap \{\vert x_2\vert = \xi^{-1}\}$ within $\{\vert x_2\vert = \xi^{-1}\}$ with $\xi>0$ small enough and $l\in \{1,\ldots,\kappa \}$. We then in turn take $\mathcal X^l$ so large that it overlaps with $\Sigma^{l,in}$ (which is possible cf. Lemma \ref{lem:realtime}). Then on $\Sigma^{l,in}$, $\mathbf{W}_{loc}^{\sigma^l}(\mathbf{q}^l(\epsilon))$ is a graph over $x_2$, recall \eqref{Wgl}. In the $(r_1,\widetilde \zeta_1)$-coordinates, the graph takes the following form
\begin{align}\eqlab{zeta1init}
\widetilde \zeta_1 = \epsilon^N \widetilde m_1^{l,N}(r_1,\epsilon),
\end{align}
with $x_2=r_1 \epsilon^{-1} \in \Sigma^{l,in}$. Here we have used \eqref{cc} and the uniqueness of the formal series. We fix the direction of time so that $\breve{\mathcal E}^l$ is a stable manifold and then use the graph as initial conditions for the flow of \eqref{K1eqnsnf2}. In the following, we fix $\xi>0$, $\chi>0$ and $\delta>0$ small enough.
% \begin{proposition}
%  j
% \end{proposition}

\begin{proposition}\label{prop:WlO1}
The invariant manifold $\mathbf{W}_{loc}^{\sigma^l}(\mathbf{q}^l(\epsilon))$ is
graph over $r_1\in \mathcal A^l(\epsilon)$:
\begin{align*}
 (y_1,z_1) = m_1^l(r_1,\epsilon),\quad r_1 \in  \mathcal A^l(\epsilon),
\end{align*}
for all $\epsilon\in (0,\epsilon_0)$ with $0<\epsilon_0\ll 1$.
Here $m_1^l(\cdot,\epsilon)$ is real-analytic with respect to $r_1\in \mathcal A^l(\epsilon)$ and  $C^\infty$ with respect to $\epsilon\in [0,\epsilon_0)$. In particular,
\begin{align}\label{eq:yzO1}
 m_1^l(r_1,0) = \psi_{1,0}^l(r_1),\quad r_1\in S^l.
\end{align}
%
% In particular, for  $r_1\in S^l\setminus B_{\mu}$, with $\mu>0$ small enough, so that $r_1=\mathcal O(1)$ with respect to $\epsilon\rightarrow 0$, then $\mathbf{W}_{loc}^{\sigma^l}(\mathbf{q}^l(\epsilon))$ takes the form
% \begin{align}
%  \begin{pmatrix}
%   y\\
%   z
%  \end{pmatrix} = \psi^l(r_1)+\mathcal O(\epsilon).
%  \end{align}
% Here $\mathcal O(\epsilon)$ is analytic with respect to $r_1\in S^l\setminus B_{\mu}$, being $C^\infty$ with respect to $0\le \epsilon<\ll 1$.
% uniformly $\mathcal O(\epsilon)$ close to the unperturbed manifold.
\end{proposition}
\begin{proof}
We fix $N\gg 1$ and use variations of constants based upon \eqref{K1eqnsnf2}:
 \begin{align*}
 \widetilde \zeta_1(t)  =& \widetilde{\mathcal M}_1^{l,N}(t,r_{1}(0),\epsilon) \epsilon^N \widetilde m_1^{l,N}(r_1(0),\epsilon) \\
  &+ \int_0^t \widetilde{\mathcal M}_1^{l,N}(t,r_{1}(0),\epsilon)\widetilde{\mathcal M}_1^{l,N}(s,r_{1}(0),\epsilon)^{-1} \epsilon^N \widetilde W_{10}^{l,N}(r_1(s),\widetilde \zeta_1(s),\epsilon r_1(s)) ds,
 \end{align*}
 using \eqref{zeta1init} with $r_1=r_1(0)$ as initial condition.
 We can then estimate the right hand side using \eqref{stmest1}: Suppose that $t\in I(\epsilon)$, $t\ge 0$, is so that $\vert \widetilde \zeta_1(t)\vert \le C_0$, $C_0>0$ fixed, $\vert r_1(t)\vert \le [\epsilon \xi^{-1},\delta]$, $r_1(t)\in \mathcal A^l(\epsilon)$ for all $0<\epsilon\ll 1$. Then there are constants $C_1$ and $C_2$ so that
\begin{align*}
 \vert \zeta(t)\vert\le C_1 \epsilon^{N-\frac{3\kappa}{2}}+tC_2 \epsilon^{N-\frac{3\kappa}{2}}.
\end{align*}
Now, we recall that \eqref{realtime} has finite time blowup with respect to $s$ along $\mathcal E^l$, see Corollary \ref{cor:blowuptime}.
% \note{Maybe we need to update the notation for time. Should it be $\tau$?}
We have $s=\epsilon^{\kappa -1}t$ and hence $0\le t\le C_3\epsilon^{1-\kappa }$ for all $0<\epsilon\ll 1$. In this way, we have
\begin{align*}
 \vert \widetilde \zeta_1(t)\vert\le C_1 \epsilon^{N-\frac{3\kappa}{2}}+C_2 C_3 \epsilon^{N+1-\frac{5\kappa}{2}}\le C_4 \epsilon^{N+1-\frac{5\kappa}{2}},
\end{align*}
which is uniformly bounded and $o(1)$ with respect to $\epsilon\to 0$ for $N\ge  \frac{5\kappa}{2}$.
We then obtain an analytic continuation of $\mathbf{W}^{\sigma^l}(\mathbf{q}^l(\epsilon))$ along paths $r_1(t)\in \mathcal A^l(\epsilon)$ in the complex plane. Since $r_1=r_1(t)$ is regular (i.e. $r_1'(t)\ne 0$), the existence of the desired graph follows.
Finally, \eqref{yzO1} follows from the definition of $\widetilde \zeta_1$, see Proposition \ref{prop:Hn}, and Lemma \ref{lem:Phi0l}. To prove the $C^\infty$-property of $m_1^l(\cdot,\epsilon)$ with respect to $\epsilon\in [0,\epsilon_0)$, we proceed as in the proof of Proposition \ref{prop:Wuspj2} using that the invariant manifolds are unique and independent of the (arbitrary) choice of $N$. %Further details are left out for simplicity.
\end{proof}

\section{Computing the differences}\label{sec:diff}
 A Corollary of Proposition \ref{prop:Wuspj2} and Proposition \ref{prop:WlO1}, is that $\mathbf{W}_{loc}^{\sigma^j}(\mathbf{q}^j(\epsilon))$ and $\mathbf{W}_{loc}^{\sigma^{j+1}}(\mathbf{q}^{j+1}(\epsilon))$, $j\in \{1,\ldots,\kappa -1\}$,  have the graph forms $(y^{j},z^j)(x,\epsilon)$, $(y^{j+1},z^{j+1})(x,\epsilon)$, respectively, that are both defined and uniformly bounded on the set $\widetilde{\mathcal H}^j\subset \mathbb C$ defined by
  \begin{align}\eqlab{tildeHj}
  x \in \widetilde{\mathcal H}^j:=\left(\epsilon \mathcal H^j\right) \cap B_{\delta},
  \end{align}
  with $\delta>0$ fixed small, for  all $0<\epsilon\ll 1$. Moreover, $x\mapsto (y^j,z^j)(x,\epsilon)$, $x\mapsto (y^{j+1},z^{j+1})(x,\epsilon)$ are each (a): equivariant with respect to conjugation, (b): $\mathcal O(\vert (x,\epsilon)\vert^{3\kappa -2})$ with respect to $x\rightarrow 0$ (in $x\in \widetilde{\mathcal H}^j$), $\epsilon\rightarrow 0$, and finally (c):
 \begin{align}
\begin{pmatrix}
 y^l\\
 z^l
\end{pmatrix}(x,\epsilon) =x^\kappa  \psi^{\nu^l(x)}(x)+\mathcal O(\epsilon),\quad x\in \widetilde{\mathcal H}^j\setminus B_{\widetilde \delta},\,l\in \{j,j+1\},\label{eq:yl}
 \end{align}
 with $0<\widetilde \delta<\delta$ fixed and where
 \begin{align*}
  \nu^l(x) := \begin{cases}
          l & \mbox{if $\operatorname{Im}(x)>0$},\\
          2\kappa -l & \mbox{else}.
         \end{cases}
 \end{align*}
 %  Here $2\kappa-l$ understood $\operatorname{mod}(2(\kappa-1))$.
 Recall Lemma \ref{lem:Sj} and Lemma \ref{lem:conj}. Here $\mathcal O(\epsilon)$ is $C^\infty$-smooth with respect to $\epsilon\in [0,\epsilon_0)$, uniformly with respect to $x\in \widetilde{\mathcal H}^j\setminus B_{\widetilde \delta}\subset \mathbb C$.
%  
%  
%  
%  We have two invariant manifolds , both defined and bounded on a neighborhood of $\vert x\vert=\mu$, $\operatorname{Arg}(x) = \frac{\pi}{\kappa -1}+\frac{\pi}{2}$, and all $0<\epsilon\ll 1$. In fact, they are both bounded on the hyperbolic path $\mathcal H^j$.

We now define $\Delta y^j = y^{j+1}-y^j$, $\Delta z^j  = z^{j+1}-z^j$, for each $j\in \{1,\ldots,\kappa -1\}$. Then as $(y^{j},z^j)(x,\epsilon)$, $(y^{j+1},z^{j+1})(x,\epsilon)$, are each invariant manifold solutions of \eqref{XYZ} for $x\in \widetilde{\mathcal H}^j$, $0\le \epsilon\ll 1$, we obtain by the mean value theorem a system
\begin{equation}\label{eq:Deltayz}
\begin{aligned}
 ({\Delta y}^j)' &= i\left(A_{11}(x,\epsilon) \Delta y^j + A_{12}(x,\epsilon )\Delta z^j\right),\\
 ({\Delta z^j})'&= i \left(A_{21}(x,\epsilon) \Delta y^j + A_{22}(x,\epsilon )\Delta z^j\right),\\
 x' &=i{P}(x,\epsilon),\\
  \epsilon'&=0,
\end{aligned}
\end{equation}
that is linear with respect to $(\Delta y^j,\Delta z^j)$.
Here
\begin{align*}
 A_{11}(x,\epsilon) &= {P}(x,\epsilon) \int_0^1 \bigg(\frac{-\frac12 {P}'(x,\epsilon)+\frac{\partial G}{\partial y}(x,y^j_\gamma,z^j_\gamma,\epsilon)}{{P}(x,\epsilon)+F(x,y^j_\gamma,z^j_\gamma,\epsilon) }\\
 &-\frac{\left( z^j_\gamma-\frac12 y^j_\gamma{P}'(x,\epsilon)+G(x,y^j_\gamma,z^j_\gamma,\epsilon)\right) }{({P}(x,\epsilon)+F(x,y^j_\gamma,z^j_\gamma,\epsilon))^2} \frac{\partial F}{\partial y}(x,y^j_\gamma,z^j_\gamma,\epsilon)\bigg)d\gamma,\\
 A_{12}(x,\epsilon) &= {P}(x,\epsilon) \int_0^1 \bigg(\frac{1+\frac{\partial G}{\partial z}(x,y^j_\gamma,z^j_\gamma,\epsilon)}{{P}(x,\epsilon)+F(x,y^j_\gamma,z^j_\gamma,\epsilon) }\\
 &-\frac{\left( z^j_\gamma-\frac12 y^j_\gamma{P}'(x,\epsilon)+G(x,y^j_\gamma,z^j_\gamma,\epsilon)\right) }{({P}(x,\epsilon)+F(x,y^j_\gamma,z^j_\gamma,\epsilon))^2} \frac{\partial F}{\partial z}(x,y^j_\gamma,z^j_\gamma,\epsilon)\bigg) d\gamma,
%  C(x,\epsilon) &= {P}(x,\epsilon) \int_0^1 \bigg(\frac{1+\frac{\partial S}{\partial z}(x,y^j_\gamma,z^j_\gamma,\epsilon)}{{P}(x,\epsilon)+F(x,y^j_\gamma,z^j_\gamma,\epsilon) }\\
%  &-\frac{\left( z^j_\gamma-\frac12 y^j_\gamma{P}'(x,\epsilon)+G(x,y^j_\gamma,z^j_\gamma,\epsilon)\right) }{({P}(x,\epsilon)+F(x,y^j_\gamma,z^j_\gamma,\epsilon))^2} \frac{\partial R}{\partial z}(x,y^j_\gamma,z^j_\gamma,\epsilon)\bigg) dt,
\end{align*}
along with similar expressions for $A_{21}$ and $A_{22}$ (which we leave out for simplicity).
Here $\sigma^{j}_\gamma = \sigma^{j}_\gamma(x,\epsilon)$, $\sigma=y,z$, are given by $$ \sigma^{j}_\gamma:=\sigma^j+\gamma \Delta \sigma^j,\quad \gamma\in [0,1].$$ These functions are (known) real-analytic functions, depending smoothly on $\epsilon$ (in the same sense as $(y^l,z^l)$ in the charts above).  It is important to note that the equations \eqref{Deltayz} are time reversible with respect to conjugation of $(\Delta y^j, \Delta z^j,x)$; we keep $\epsilon>0$ real.

We cannot easily evaluate $A_{\alpha\beta}$ since the expressions involve (removable) singularities for $x=\epsilon=0$. However, if we put
\begin{align}\label{eq:x2r2}
x=\epsilon x_2,
\end{align} 
then we easily obtain 
\begin{align*}
 A_{11} (\epsilon x_2,\epsilon)&= -\frac12 \epsilon^{\kappa -1} {Q}'(x_2)+\mathcal O(\epsilon^{2(\kappa -1)}),\\
 A_{12} (\epsilon x_2,\epsilon)&=1+\mathcal O(\epsilon^{2(\kappa -1)}),\\
 A_{21}(\epsilon x_2,\epsilon) &=-1+\mathcal O(\epsilon^{2(\kappa -1)}),\\
 A_{22}(\epsilon x_2,\epsilon) &=-\frac12 \epsilon^{\kappa -1} {Q}'(x_2)+\mathcal O(\epsilon^{2(\kappa -1)}),
\end{align*}
for $0\le \epsilon\ll 1$,
uniformly with respect to $x_2\in \mathcal H^j \cap B_{\varrho}$, $\varrho>0$ fixed, and
\begin{align}\eqlab{thisfuck}
 x_2' &= i\epsilon^{\kappa -1} {Q}(x_2).
\end{align}
Here we have used that ${Q}\ne 0$ on $\mathcal H^j$.
Therefore for $\epsilon=0$, we have
\begin{equation}\label{eq:Deltayzx2eps0}
\begin{aligned}
 ({\Delta y}^j)' &= i\Delta z^j,\\
 ({\Delta z}^j)'&=-i\Delta y^j.\\
 x_2' &=0,
\end{aligned}
\end{equation}
Here $(\Delta y^j,\Delta z^j)=(0,0)$, $x_2\in \mathcal H^j \cap B_{\varrho}$, $\varrho>0$ fixed, defines a normally hyperbolic critical manifold of saddle-type. Indeed, the eigenvalues of the linearization of \eqref{Deltayzx2eps0} are $\pm 1,0$. Therefore by Fenichel's theory we can straighten out the stable and unstable manifolds of $(\Delta y^j,\Delta z^j)=(0,0)$ (which are line bundles by the linearity) through a transformation of the form
 \begin{align}\label{eq:uvhere}
  \begin{pmatrix}
   \Delta y^j\\
   \Delta z^j
  \end{pmatrix} = \begin{pmatrix}
  Y_2^j(x_2,\epsilon) & \overline{Y_2^j}(\overline x_2,\epsilon)\\
  1 & 1
  \end{pmatrix}\begin{pmatrix}
   \Delta s^j\\
   \Delta u^j
  \end{pmatrix},\quad x_2\in \mathcal H^j \cap B_{\varrho},
 \end{align}
 with $Y_2^j(x_2,0)=-i$,
 so that the $(\Delta s^j,\Delta u^j)$-subsystem becomes diagonalized
 \begin{equation}\label{eq:uvsystem}
 \begin{aligned}
  (\Delta s^j)' &= \left(1-\frac{i}{2} \epsilon^{\kappa -1} {Q}'(x_2)+\mathcal O(\epsilon^{2(\kappa -1)})\right)\Delta s^j,\\
  (\Delta u^j)'&= \left(-1-\frac{i}{2} \epsilon^{\kappa -1} {Q}'(x_2)+\mathcal O(\epsilon^{2(\kappa -1)})\right)\Delta u^j.
 \end{aligned}
 \end{equation}
 These equations together with \eqref{thisfuck} constitute a Fenichel normal form for the slow-fast system, see \cite{jones_1995}.
  All $\mathcal O(\epsilon^{2(\kappa -1)})$-terms are $C^n$ (for any $n\in \mathbb N$) with respect to $x_2\in \mathcal H^j \cap B_{\varrho}$ and $\epsilon \in [0,\epsilon_0)$, $0<\epsilon_0=\epsilon_0(n,\varrho)\ll 1$. To obtain the form \eqref{uvsystem} we have fixed a stable manifold of $(\Delta y^j,\Delta z^j)=(0,0)$ and then from here obtained the unstable manifold through the time-reversal symmetry. The transformed system given by \eqref{uvsystem} and $x_2'=i\epsilon^{\kappa -1}{Q}(x_2)$, $x_2\in \mathcal H^j$, is then time reversible with respect to the symmetry $(\Delta s^j,\Delta u^j,x_2)\mapsto (\overline{\Delta u^j},\overline{\Delta s^j},\overline x_2)$.

  In order to ``globalize the diagonalization'' to $x=\mathcal O(1)$, $x\in \widetilde{\mathcal H}^j$, we use the following directional blow-up
 \begin{align}\label{eq:rho1thetaw1}
  \begin{cases} x = \rho_1 \e^{i\theta},\\
   \epsilon = \rho_1w_1,
  \end{cases}
 \end{align}
%  
% and 
% \begin{align}\label{eq:x2r2}
%   \begin{cases} x &= \rho_2 x_2,\\
%    \epsilon &= \rho_2,
%   \end{cases}
%  \end{align}
%  respectively. 
Notice that $x_2=w_1^{-1} \e^{i\theta}$ as in \eqref{x2w1}. Since we are only interested in $\mathcal H^j$, we restrict attention to either $\theta=H^j(w_1)$ or $\theta = H^{2\kappa -1-j}(w_1)$, $w_1\in [0,\mu]$, cf. Lemma \ref{lem:imagtime}, see \eqref{conjhl} and assertion \ref{it:imag8}. %\fbox{check index}
% By the symmetry, we consider the former.
This leads to the following expansion of $A_{\alpha\beta}$:
 \begin{align*}
 A_{11} (\rho_1 \e^{i\theta},\rho_1 w_1)&= -\frac12 \rho_1^{\kappa -1} {P}_x'(\e^{i\theta},w_1)+\mathcal O(\rho_1^{2(\kappa -1)}),\\
 A_{12} (\rho_1 \e^{i\theta},\rho_1 w_1)&=1+\mathcal O(\rho_1^{2(\kappa -1)}),\\
 A_{21}(\rho_1 \e^{i\theta},\rho_1 w_1) &=-1+\mathcal O(\rho_1^{2(\kappa -1)}),\\
 A_{22}(\rho_1 \e^{i\theta},\rho_1 w_1) &=-\frac12 \rho_1^{\kappa -1} {P}_x'(\e^{i\theta},w_1)+\mathcal O(\rho_1^{2(\kappa -1)}),
\end{align*}
uniformly with respect to $w_1\in [0,\mu]$, and
\begin{align*}
\rho_1' &=\rho_1^\kappa  \operatorname{Re}\left(i\e^{-i\theta}{P}(\e^{i\theta},w_1)\right),\\
%  \dot \theta &=\rho_1^\kappa  \operatorname{Im}\left(i\e^{-i\theta}{P}(\e^{i\theta},w_1)\right),\\
 w_1' &=-\rho_1^{\kappa -1}w_1 \operatorname{Re}\left(i\e^{-i\theta}{P}(\e^{i\theta},w_1)\right).
 \end{align*}
Therefore the set $\rho_1=0$, $(\Delta y^j,\Delta z^j)=(0,0)$, $w_1\in [0,\mu]$, with either $\theta=H^j(w_1)$ or $\theta = H^{2\kappa -1-j}(w_1)$, defines a partially hyperbolic set of equilibria of saddle type. %\fbox{check $2k-j$?}
Indeed, the eigenvalues are $\pm 1$, $0$, $0$. In projective variables $Y^j:=(\Delta z^j)^{-1} \Delta y^j$, the stable and unstable manifolds of $(\Delta y^j,\Delta z^j)=(0,0)$ become slow manifolds in the $(Y^j,x_2,\epsilon)$-coordinates and center manifolds in the $(Y^j,\rho_1,w_1)$-coordinates for $\mu>0$ small enough. This follows from simple calculations, see \cite[Proposition 5.2]{bkt}. We can therefore extend the slow manifolds by using the center manifolds in the $(Y^j,\rho_1,w_1)$-coordinates (upon taking $\varrho>\mu^{-1}>0$ so that the sets overlap upon change of coordinates), see \cite[Proposition 5.2]{bkt} for further details. This leads to ``global'' stable and unstable manifolds of ($\Delta y^j,\Delta z^j)=(0,0)$ in an $\mathcal O(1)$ neighborhood of $x=0$ in $\widetilde{\mathcal H}^j$ for all $0<\epsilon\ll 1$. As above, these manifolds are not unique, but we fix a stable manifold of $(\Delta y^j,\Delta z^j)=(0,0)$ and then obtain the unstable one through the time-reversal symmetry. This leads to the following result: %see \cite[Proposition 5.2]{bkt} for further details.
% This then leads to the following result.
%  The following proposition shows that the diagonalization can be ``globalized'' (using blow-up).

% We use blow-up
% \begin{align*}
%  \rho\ge 0,\,(\breve w,\breve x)\in \mathbb S^2,\,\mapsto \begin{cases}
%                               x &=\rho \breve x,\\
%                               \epsilon &=\rho \breve w.
%                              \end{cases}
% \end{align*}

\begin{proposition}
% \fbox{Would be nice if this could be done in a way that respect the symmetry/conjugation}
 Fix any $n\in \mathbb N$ and $\varrho>\mu^{-1}>0$ with $\mu>0$ small enough. Then there exists an $\epsilon_0>0$ small enough, such that the system \eqref{Deltayz} can be diagonalized for $x\in \widetilde{\mathcal H}^j$ for all $\epsilon \in (0,\epsilon_0)$ by a transformation of the form
 \begin{align}\label{eq:uv}
  \begin{pmatrix}
   \Delta y^j\\
   \Delta z^j
  \end{pmatrix} = \begin{pmatrix}
  Y^j(x,\epsilon) & \overline{Y^j}(\overline x,\epsilon)\\
  1 & 1
  \end{pmatrix}\begin{pmatrix}
   \Delta s^j\\
   \Delta u^j
  \end{pmatrix},
 \end{align}
 where
 \begin{align*}
   Y^j(x,\epsilon) = -i +\mathcal O(\vert(x,\epsilon)\vert^{2(\kappa -1)}), \end{align*}
%  \note{maybe rethink this: better to think of $\mathcal O(2(k-1))$ in charts}
 so that
\begin{equation}\label{eq:Deltauv}
\begin{aligned}
 (\Delta s^j)' &= \left(1-\frac{i}{2} {P}'_x(x,\epsilon)+\mathcal O(\vert(x,\epsilon)\vert^{2(\kappa -1)})\right)\Delta s^j,\\
 (\Delta u^j)' &= \left(-1-\frac{i}{2} {P}'_x(x,\epsilon)+\mathcal O(\vert(x,\epsilon)\vert^{2(\kappa -1)})\right)\Delta u^j,\\
 x' &=i{P}(x,\epsilon),\\
 \epsilon'&=0.
\end{aligned}
\end{equation}
The system \eqref{Deltauv} with $x\in \widetilde{\mathcal H}^j$ is time reversible with respect to the symmetry
\begin{align}\label{eq:symmetry}
(\Delta s^j,\Delta u^j,x,\epsilon)\mapsto (\overline{\Delta u^j},\overline{\Delta s^j},\overline x,\epsilon).
\end{align}
Moreover, all $\mathcal O(\vert(x,\epsilon)\vert^{2(\kappa -1)})$-terms are $C^n$-smooth and the order is understood in the usual sense with respect to the dual limit $(x,\epsilon)\rightarrow 0$ (along $\widetilde{\mathcal H}^j \times \mathbb R_+$). In particular, in the $(x_2,\epsilon)$-coordinates all $\mathcal O(\vert(x,\epsilon)\vert^{2(\kappa -1)})$-terms are $\mathcal O(\epsilon^{2(\kappa -1)})$ uniformly with respect to $x_2\in \mathcal H^j\cap \mathcal B_\varrho$, whereas in the $(\rho_1,w_1)$-coordinates the same functions are $\mathcal O(\rho_1^{2(\kappa -1)})$ uniformly with respect to $w_1\in [0,\mu]$ (for both $\theta = H^j(w_1)$ and $\theta=H^{2\kappa -1-j}(w_1)$).
\end{proposition}
 In the following, we fix any $n\in \mathbb N$ and define
 \begin{align*}
 p^j:=\mathcal H^j\cap \{\operatorname{Im}(x_2)=0\}\in \mathbb R\quad \forall\,j\in \{1,\ldots,\kappa -1\}.
 \end{align*}
 Moreover, we fix $\rho_1>0$ small enough and put
 \begin{align*}
  x^j(\epsilon) &=  \rho_1 \e^{iH^j(\epsilon \rho_1^{-1})} \in \widetilde{\mathcal H}^j,
 \end{align*}
 for all $0<\epsilon\ll 1$.
 Here we have used assertion \ref{it:imag8} of Lemma \ref{lem:imagtime} with $w_1=\epsilon \rho_1^{-1}$, recall \eqref{rho1thetaw1}.
 We then consider the system \eqref{Deltauv} in the form
 \begin{equation}\label{eq:Deltauvj}
 \begin{aligned}
  i {P}(x,\epsilon) \frac{d\Delta s^j}{dx} &= (1-\frac{i}{2}{P}_x'(x,\epsilon)+\mathcal O(\vert(x,\epsilon)\vert^{2(\kappa -1)}))\Delta s^j,\\
  i {P}(x,\epsilon) \frac{d\Delta u^j}{dx} &= (-1-\frac{i}{2}{P}_x'(x,\epsilon)+\mathcal O(\vert(x,\epsilon)\vert^{2(\kappa -1)}))\Delta u^j,
 \end{aligned}
 \end{equation}
for $(\Delta s^j,\Delta u^j)(x,\epsilon)$, $x\in \widetilde{\mathcal H}^j$, $0<\epsilon\ll 1$, using initial conditions at $x=x^j(\epsilon)$ and $x=\overline{x^j(\epsilon)}=x^{2\kappa -1-j}(\epsilon)$. \eqref{Deltauv} is time reversible with respect to \eqref{symmetry}. By assertion \ref{it:imag8} of Lemma \ref{lem:imagtime}, we have
\begin{align}\label{eq:argxj}
 \operatorname{Arg}(x^j) = H^j(\epsilon\rho_1^{-1}) = \frac{\pi}{\kappa -1}(j-1)+\frac{\pi}{2(\kappa -1)}+\mathcal O(\epsilon)\in (0,\pi),
\end{align}
for all $0\le \epsilon\ll 1$. Here $\mathcal O(\epsilon)$ is a $C^\infty$-smooth function. In the statements below, we take the usual branch cut of the square root along the negative real axis.

%By working in the coordinates \eqref{} and \eqref{}, we obtain the following
 \begin{lemma}\label{lem:final}
 Let $j\in \{1,\ldots,\kappa -1\}$ and
 define $\mathcal I^j=\mathcal I^j(\epsilon)$ by
 \begin{align*}
 \mathcal I^j &= \epsilon^{\kappa -1} \int_{x^j}^{\epsilon p^j} \frac{-i}{{P}(x,\epsilon)}dx,\\
%    I_2 &= -\frac12 \log \frac{{Q}(p^j)}{{Q}(\epsilon^{-1} x^j)}-\log \epsilon^{-\frac{\kappa }{2}},
   \end{align*}
   with the complex integration along $\widetilde{\mathcal H}^j\cap\{\operatorname{Im}(x)\ge 0\}$, recall \eqref{tildeHj}.
   Then $\mathcal I^j$ has the following expansion
   \begin{align}
   \mathcal I^j &=  \sum_{l=1}^j \frac{\pi}{Q'(q^l)} + \mathcal O(\epsilon^{\kappa-1}),\label{eq:I1}
%     \begin{align*}
%  I_2 &= \mathcal O(1),\label{eq:I2}
% \end{align*}
\end{align}
with $\mathcal O(\epsilon^{\kappa-1})$ being complex valued and $C^\infty$-smooth with respect to $\epsilon\in [0,\epsilon_0)$, $0<\epsilon_0\ll 1$.
%    \end{align*}
% 

Next, consider \eqref{Deltauvj} with the following initial conditions:
\begin{align*}
 (\Delta s^j(x^j,\epsilon),\overline{\Delta s^j(x^j,\epsilon)}),
\end{align*}
at $x=x^j(\epsilon)$ for $j$ odd, and
\begin{align*}
 (\overline{\Delta u^j(\overline{x^j},\epsilon)},\Delta u^j(\overline{x^j},\epsilon)),
\end{align*}
at $x=\overline{x^j(\epsilon)}=x^{2\kappa -1-j}(\epsilon)$ for $j$ even. Then the following holds true (with $0<\epsilon_0\ll 1$):
% %  Moreover, the following holds
% \note{this should be divided into $j$ odd and $j$ even cf. \eqref{sumQkj}}
\begin{enumerate}
 \item Suppose that $j$ is odd such that
 \begin{align*}
  \sum_{l=1}^j \frac{1}{{Q}'(q^l)}<0,
 \end{align*}
 cf. \eqref{sumQkj}.
%  
%  \fbox{need to use the symmetry!}
Then there is a $C^n$-smooth function $\mathcal J^j:[0,\epsilon_0)\rightarrow \mathbb C$, such that
\begin{equation}\label{eq:Deltauvjfinal}
 \begin{aligned}
 \Delta s^j(\epsilon p^j,\epsilon) &= \epsilon^{-\frac{\kappa }{2}}\left(\frac{{P}(x^j,\epsilon)}{{Q}(p^j)}\right)^{\frac12} \exp\left(\epsilon^{1-\kappa } \mathcal I^j+\mathcal J^j\right)\Delta s^j(x^j,\epsilon),\\
%  &=
%  \epsilon^{-\frac{k}{2}} \exp\bigg({\epsilon^{1-k} \sum_{i=1}^j \frac{\pi}{Q'_k(q^l)} \left(1+\mathcal O(\epsilon)\right)} \\
%  &\quad -{i\epsilon^{1-k} i \sum_{l=1}^k\left( \frac{1}{Q'_k(q^l)} \log |p^j-q^l| + \mathcal O(\epsilon)\right)}\bigg) \Delta s^j(x^j(\epsilon),\epsilon),\\
 \Delta u^j(\epsilon p^j,\epsilon) &=\epsilon^{-\frac{\kappa }{2}} \left(\frac{{P}(\overline{x^j},\epsilon)}{{Q}(p^j)}\right)^{\frac12} \exp\left(\epsilon^{1-\kappa } \overline{\mathcal I^j}+\overline{ \mathcal J^j}\right)\overline{\Delta s^j( x^j,\epsilon)}.%\exp\left( -\overline I_1+\overline I_2+\mathcal O(1)\right)\Delta u^j(\overline{x^j}(\epsilon),\epsilon).
%  &=\epsilon^{-\frac{k}{2}} \exp \bigg({\epsilon^{1-k} \sum_{i=1}^j \frac{\pi}{Q'_k(q^l)} \left(1+\mathcal O(\epsilon)\right)}\\
%  &\quad +{i\epsilon^{1-k} \sum_{l=1}^k\left(\frac{1}{Q'_k(q^l)} \log |p^j-q^l| + \mathcal O(\epsilon)\right)}\bigg) \Delta u^j(\overline{x^j}(\epsilon),\epsilon).
%  \Delta u^j(\epsilon p^j,\epsilon) &= \left(\frac{{Q}(\epsilon p^j ) }{{Q}(\epsilon^{-1} x )}\right)^{\frac12} \e^{+\mathcal O(1)}\Delta s^j(x^j(\epsilon),\epsilon).
 \end{aligned}
 \end{equation}
\item  Suppose next that $j$ is even such that
 \begin{align*}
  \sum_{l=1}^j \frac{1}{{Q}'(q^l)}>0,
 \end{align*}
 cf. \eqref{sumQkj}.
Then there is a $C^n$-smooth function $\mathcal K^j:[0,\epsilon_0)\rightarrow \mathbb C$, such that
%  \begin{align*}
 \begin{align*}
 \Delta s^j(\epsilon p^j,\epsilon) &=\epsilon^{-\frac{\kappa }{2}} \left(\frac{{P}(\overline{x^j},\epsilon)}{{Q}(p^j)}\right)^{\frac12}  \exp\left( -\epsilon^{1-\kappa }\overline{\mathcal I^j}+\overline{\mathcal K^j}\right)\overline{\Delta u^j(\overline{x^j},\epsilon)},\\
 \Delta u^j(\epsilon p^j,\epsilon) &=\epsilon^{-\frac{\kappa }{2}} \left(\frac{{P}(x^j,\epsilon)}{{Q}(p^j)}\right)^{\frac12} \exp\left( -\epsilon^{1-\kappa } {\mathcal I^j}+ \mathcal K^j\right)\Delta u^j(\overline{x^j},\epsilon).
 \end{align*}
%  \note{double check these expressions}
 \end{enumerate}
 
%   \begin{pmatrix} \Delta s^j\\
%      \Delta u^j
%   \end{pmatrix}(\epsilon p^j,\epsilon) =
  
%   
% \left(\frac{{Q}(\epsilon p^j ) }{{Q}(\epsilon^{-1} x )}\right)^{\frac12} \exp \left(\Omega \left(i s + \mathcal O(\epsilon)\right)\right).
%  \end{align*}

 \end{lemma}
 \begin{proof}
To prove \eqref{I1}, we
first notice that by the substitution
$x=\epsilon x_2$ and \eqref{PkQk}, we have that:
\begin{align*}
 \mathcal I^j& = (-i) \int_{\epsilon^{-1} x^j}^{p^j} \frac{1}{{Q}(x_2)}dx_2.
%  &=(-i) \sum_{l=1}^\kappa  \frac{1}{Q'(q^l)}\log \frac{p^j -q^l}{\epsilon^{-1} x^j-q^l}.
\end{align*}
We therefore consider the closed positively oriented curve $\widetilde \gamma^j(\epsilon)$ in the complex $x$-plane, defined as the union of $\widetilde{\mathcal H}^j$ and the circle arc from $\overline x^j(\epsilon)=\rho_1 \e^{iH^{2\kappa-1-j}(\epsilon \rho_1^{-1})}$ to $x^j(\epsilon)=\rho_1 \e^{iH^{j}(\epsilon \rho_1^{-1})}$ with radius $\rho_1>0$. Then by the residue theorem and the fact that $\widetilde{\mathcal H}^j$ is invariant with respect to conjugation, see assertion \ref{it:imag6} of Lemma \ref{lem:imagtime}, we find that
\begin{align*}
 \epsilon^{\kappa-1}\int_{\widetilde \gamma^j} \frac{1}{{P}(x,\epsilon)}dx &=2 \int_{\epsilon^{-1} x^j}^{p^j} \frac{1}{{Q}(x_2)}dx_2+\epsilon^{\kappa-1}\int_{H^{2\kappa-1-j}(\epsilon \rho_1^{-1})}^{H^{j}(\epsilon \rho_1^{-1})} \frac{i\rho_1 \e^{i\theta}}{P(\rho_1 \e^{i\theta},\epsilon)}d\theta\\
 &=2\pi i \sum_{l=1}^j \frac{1}{Q'(q^l)},
\end{align*}
for all $\epsilon>0$ small enough. Hence
\begin{align*}
  2\mathcal I^j &= (-i)\left(2\pi i \sum_{l=1}^j \frac{1}{Q'(q^l)} - \epsilon^{\kappa-1}\int_{H^{2\kappa-1-j}(\epsilon \rho_1^{-1})}^{H^{j}(\epsilon \rho_1^{-1})} \frac {i\rho_1 \e^{i\theta} }{P(\rho_1 \e^{i\theta};\epsilon)}d\theta\right)\\
  &=2 \sum_{l=1}^j \frac{\pi}{Q'(q^l) } +\mathcal O(\epsilon^{\kappa-1}),
\end{align*}
% recall \eqref{PkQk}.
%
% We then expand the right hand side with respect to $\epsilon$. For this purpose, we notice that
% since $q^{j+1}<p^j<q^j$ for all $j\in \{1,\ldots,\kappa -1\}$, we have that
% \begin{align}\label{eq:pjql}
%  \begin{cases}
%           p^j-q^l < 0 \mbox{ for $l\le j$},\\
%           p^j-q^l > 0 \mbox{ for $l>j$}.
%          \end{cases}
% \end{align}
% This then leads to
% \begin{align*}
%  \log \frac{p^j-q^l}{\epsilon^{-1} x^j - q^l} &=  \log \frac{p^j-q^l}{\epsilon^{-1} x^j} +\mathcal O(\epsilon) \\
%  &= \log |p^j-q^l| + \log \epsilon \rho_1^{-1} + i \begin{cases}
%                                        \pi-H^j(0) & l\le j\\
%                                        -H^j(0) & l> j
%                                      \end{cases}+\mathcal O(\epsilon).\end{align*}
% %                                      \note{Maybe improve the expansion}
%                                      Here we have used that $H^j\in (0,\pi)$ for all $j\in \{1,\ldots,\kappa -1\}$ together with \eqref{argxj}.
%                                      {Notice that we take $\operatorname{arg} (p^j-q^l)=\pi$ for $l\le j$ to ensure that the argument is in the principal domain.}
% %                                      Here we have used \eqref{argxj}.
% We therefore conclude
% \begin{align*}
%  \mathcal I^j&=(-i) \sum_{l=1}^\kappa  \frac{1}{Q'(q^l)} \log |p^j-q^l| +  \sum_{l=1}^j \frac{\pi}{Q'(q^l)}+\mathcal O(\epsilon),
% \end{align*}
with $\mathcal O(\epsilon^{\kappa-1})$ being $C^\infty$-smooth with respect to $\epsilon\in [0,\epsilon_0)$, $0<\epsilon_0\ll 1$. This completes the proof of \eqref{I1}.

%  \note{proof the expression for $I$ first!}
For the remainder of the proof, we suppose that $j$ is odd so that
%  For simplicity, we only focus
 \begin{align*}
  \sum_{l=1}^j \frac{1}{{Q}(q^l)}<0,
 \end{align*}
 recall \eqref{sumQkj}.
 The case with $j$ even can be treated in the same way. Since the system \eqref{Deltauvj} decouples we can treat $\Delta s^j(x,\epsilon)$ and $\Delta u^j(x,\epsilon)$ separately. Moreover, once we have determined $\Delta s^j(\epsilon p^j,\epsilon)$, then we obtain the expression for $\Delta u^j(\epsilon p^j,\epsilon)$ by applying the symmetry \eqref{symmetry}. We therefore solve for $\Delta s^j$ using an initial condition at $x=x^j$.
%  For $\Delta u^j(x,\epsilon)$, which we will treat subsequently, we will use an initial condition at $x=\overline{x^j}$.
Since the system \eqref{Deltauvj} is linear, we can directly compute
\begin{equation}\label{eq:ujI}
   \begin{aligned}
  \Delta s^j(\epsilon p^j,\epsilon) &= \left(\frac{{Q}(\epsilon^{-1} x^j)}{{Q}(p^j)}\right)^{\frac12} \e^{\epsilon^{1-\kappa }\mathcal I^j+\mathcal J^j}\Delta s^j(x^j,\epsilon)\\
  &=  \epsilon^{-\frac{\kappa }{2}}\left(\frac{{P}(x^j,\epsilon)}{{Q}(p^j)}\right)^{\frac12} \e^{\epsilon^{1-\kappa }\mathcal I^j+\mathcal J^j}\Delta s^j(x^j,\epsilon)
 \end{aligned}
 \end{equation}
 where
 \begin{align*}
%    I_1 &= \int_{x^j}^{\epsilon p^j} \frac{-i}{{P}(x,\epsilon)}dx,\\
%    I_2 &= -\frac12 \log \frac{{Q}(p^j)}{{Q}(\epsilon^{-1} x^j)},\\
   \mathcal J^j &= \int_{x^j}^{\epsilon p^j} \frac{1}{{P}(x,\epsilon)} \mathcal O(\vert(x,\epsilon)\vert^{2(\kappa -1)}) dx,
 \end{align*}
for $x^j=x^j(\epsilon)$. In deriving \eqref{ujI}, we have used \eqref{PkQk} and \eqref{logQk}. %Here the complex integration is along $x\in \widetilde{\mathcal H}^j$.
By working in the coordinates \eqref{x2r2} and \eqref{rho1thetaw1}, we can easily obtain that $\mathcal J^j=\mathcal O(1)$, being $C^n$-smooth with respect to $\epsilon$. This completes the proof.

 \end{proof}
\section{Completing the proof of Theorem \ref{thm:main}}\label{sec:complete}
We consider $j$ odd. The case $j$ even is similar. We then use \eqref{yl} and the assumptions of the theorem, to conclude that
\begin{align*}
(\Delta y^j,
\Delta z^j)(x^j(0),0) \ne (0,0),
\end{align*}
and therefore also
\begin{align*}
(\Delta s^j,
\Delta u^j)(x^j(0),0)\ne (0,0),
\end{align*}
by \eqref{uv}. However,
$\Delta s^j(x^j(0),0)=0$ leads to a contradiction. Indeed, by proceeding as in Lemma \ref{lem:final}, we obtain from $\Delta u^j(x^j(0),0)\ne 0$ that $\Delta u^j(\epsilon p^j,\epsilon)$ is exponentially large and unbounded as $\epsilon\rightarrow 0$. This contradicts Proposition \ref{prop:Wuspj2}. We are therefore left with $\Delta s^j(x^j(0),0)\ne 0$. We then obtain $(\Delta s^j,\Delta u^j)(\epsilon p^j,\epsilon)$ from \eqref{Deltauvjfinal} and in turn $(\Delta y^j,\Delta z^j)(\epsilon p^j,\epsilon)$ by using \eqref{uv}:
\begin{align*}
\begin{pmatrix}
\Delta y^j\\
\Delta z^j
\end{pmatrix}(\epsilon p^j,\epsilon) = 2\operatorname{Re} \begin{pmatrix}
Y^j \Delta s^j\\
\Delta s^j
\end{pmatrix}(\epsilon p^j,\epsilon).
\end{align*}
The result then follows from a simple calculation using \eqref{sumQkj} and \eqref{I1}, upon setting $(\Delta y_2^j,\Delta z_2^j)(x_2,\epsilon):=\epsilon^{-\kappa}(\Delta y^j,\Delta z^j)(\epsilon^{-1} x_2,\epsilon))$ (cf. \eqref{xyzx2y2z21}). Notice, from the construction above, that the function $\Xi^j$ is apriori only $C^n$.  However, the manifolds are unique and $C^\infty$-smooth for $\epsilon>0$ (in particular, they are independent of our choice of $n$ and $\rho_1>0$). We therefore conclude that the function $\Xi^j$ is indeed $C^\infty$ with respect $\epsilon\in [0,\epsilon_0)$.
% \section{Discussion}\label{sec:conclusion}
% \end{lemma}
\subsection*{Acknowledgement}
The author was funded by Danish Research Council (DFF) grant 4283-00014B.

% \newpage
\bibliography{refs}
\bibliographystyle{plain}
\newpage
\appendix

\section{On writing \eqref{3rd} in the form \eqref{x2y2z2new0}}\label{app:A}
In this appendix, we consider \eqref{3rd} repeated here for convinience:
\begin{align}
\epsilon^{2(\kappa-1)} f''' +f' = Q(f),\quad f=f(t).\eqlab{3rdapp}
\end{align}
We will bring it into the normal form \eqref{x2y2z2new0}.
First, by defining
\begin{align}\label{eq:x2y2z2f}
\begin{cases} x_2: = f+\epsilon^{2(\kappa -1)} f'',\\
y_2:=f',\\
z_2:=\epsilon^{\kappa -1} f'',
\end{cases}
\end{align}
we can easily write \eqref{3rd} as the first order slow-fast system:
\begin{equation}\label{eq:x2y2z2}
\begin{aligned}
 \dot x_2&=\epsilon^{\kappa -1} {Q}(x_2-\epsilon^{\kappa -1} z_2),\\
 \dot y_2&=z_2,\\
 \dot z_2 &=-y_2 + {Q}(x_2-\epsilon^{\kappa -1}z_2),
\end{aligned}
\end{equation}
where $\dot{()}=\frac{d}{d\tau}$ with $\tau = \epsilon^{1-\kappa }t$. For future reference, we notice that the change of coordinates
\begin{align}\label{xyzx2y2z2app}
(x_2,y_2,z_2)\mapsto  \begin{cases}
  x = \epsilon x_2,\\
  y =\epsilon^\kappa  y_2,\\
  z=\epsilon^\kappa  z_2,
 \end{cases}
\end{align}
brings \eqref{x2y2z2} into the following form:
\begin{equation}\label{eq:xyzP}
\begin{aligned}
 \dot x &=P(x-z,\epsilon),\\
 \dot y &=z,\\
 \dot z &=-y+P(x-z,\epsilon),
\end{aligned}
\end{equation}
where we have defined
% with $P$ given by
\begin{align}\label{eq:Pdefn}
 P(x,\epsilon) := \epsilon^\kappa  Q(\epsilon^{-1}x) = -x^\kappa +\sum_{\alpha=0}^{\kappa -1} \epsilon^{\kappa -\alpha }a_\alpha x^\alpha.
\end{align}
For $\epsilon=0$, \eqref{xyzP} reduces to
\begin{equation}\label{eq:xyzP0app}
\begin{aligned}
 \dot x &=-(x-z)^\kappa ,\\
 \dot y &=z,\\
 \dot z &=-y-(x-z)^\kappa .
\end{aligned}
\end{equation}
\begin{remark}\remlab{apprem}
 If we perturb \eqref{3rdapp} by functions of the form \eqref{Wpert}:
 \begin{align}\eqlab{3rdapppert}
\epsilon^{2(\kappa-1)} f''' +f' = Q(f)+\epsilon^{-\kappa} W(\epsilon f,\epsilon^{\kappa}f',\epsilon^{2\kappa-1}f'',\epsilon),
\end{align}
 with $W$ satisfying \eqref{Xcond}, then \eqref{xyzP0app} becomes
\begin{equation}\nonumber
\begin{aligned}
 \dot x &=-(x-z)^\kappa + W(x-z,y,z,\epsilon),\\
 \dot y &=z,\\
 \dot z &=-y-(x-z)^\kappa + W(x-z,y,z,\epsilon).
\end{aligned}
\end{equation}
The transformations below also bring \eqref{3rdapppert} into the normal form \eqref{x2y2z2new0}. This should be clear enough and we leave out further details.
\end{remark}

We now apply an $\epsilon$-dependent normal form transformation of the form
 $$(x_2,y_2,z_2)(\mapsto (\widetilde x_2,\widetilde y_2,\widetilde z_2),$$ given by
 \begin{equation}\label{eq:nft}
 \begin{aligned}
 \widetilde x_2 &= x_2+\epsilon^{2(\kappa -1)}{Q}'(x_2) y_2\\
 \widetilde y_2 &= y_2-{Q}(x_2)+\frac12 \epsilon^{\kappa -1} Q'(x_2) z_2,\\
  \widetilde z_2 &= z_2-\epsilon^{\kappa -1} {Q}'(x_2)Q(x_2).%+\frac14 \epsilon^{k-1} Q'(x_2) y_2.
%  \widetilde z &=z -{P}'_x(x,\epsilon){P}(x,\epsilon)+\frac14 {P}'_x(x,\epsilon)y,
 \end{aligned}
 \end{equation}
%   that (a) fixes $\epsilon$, (b) is $\mathcal O(k)$-close to the identity (with respect to the weights $(1,1,1,1)$), and finally (c) that
This transformation conjugates \eqref{x2y2z2} with the following set of equations
 \begin{equation}\label{eq:tildexuv}
 \begin{aligned}
 \dot{\widetilde x}_2  &= \epsilon^{\kappa -1}\left(Q(\widetilde x_2) + \epsilon^{-\kappa }\widetilde F(\epsilon \widetilde x_2,\epsilon^\kappa  \widetilde y_2,\epsilon^\kappa  \widetilde z_2,\epsilon)\right),\\%\sum_{n=0}^{\lfloor \frac{\kappa }{2}\rfloor} \frac{1}{2^{2n} n!^2} {P}^{(2n)}(\widetilde x,\epsilon) (\widetilde y^2+\widetilde z^2)^n+\widetilde R_{2\kappa -1},\\
 \dot{\widetilde y}_2  &= \widetilde z_2 -\frac12 \epsilon^{\kappa -1} {\widetilde y_2} {Q}'(\widetilde x_2) +\epsilon^{-\kappa }\widetilde G(\epsilon \widetilde x_2,\epsilon^\kappa  \widetilde y_2,\epsilon^\kappa  \widetilde z_2,\epsilon),\\
  \dot{\widetilde z}_2  &= -\widetilde y_2 -\frac12 \epsilon^{\kappa -1} {\widetilde z_2} {Q}'(\widetilde x_2) +\epsilon^{-\kappa }\widetilde H(\epsilon \widetilde x_2,\epsilon^\kappa  \widetilde y_2,\epsilon^\kappa  \widetilde z_2,\epsilon),
%   \dot \epsilon &=0,
%  \end{align*}
\end{aligned}
\end{equation}
where
\begin{align*}
\epsilon^{-\kappa }\widetilde F(\epsilon \widetilde x_2,\epsilon^\kappa  \widetilde y_2,\epsilon^\kappa  \widetilde z_2,\epsilon):=&Q(x_2-\epsilon^{\kappa -1} z_2)-Q(\widetilde x_2)-\epsilon^{\kappa -1}Q'(x_2) z_2\\
&+\epsilon^{2(\kappa -1)}Q''(x_2)Q(x_2)y_2,\\
\epsilon^{-\kappa }\widetilde G(\epsilon \widetilde x_2,\epsilon^\kappa  \widetilde y_2,\epsilon^\kappa  \widetilde z_2,\epsilon):=&\frac12 \epsilon^{\kappa -1}Q'(x_2)\left(Q(x_2-\epsilon^{\kappa -1}z_2)-Q(x_2)+\frac12 \epsilon^{\kappa -1} Q'(x_2)z_2\right)\\
&+\frac12 \epsilon^{\kappa -1}\left(Q'(\widetilde x_2)-Q'(x_2)\right)\widetilde y_2,\\
\epsilon^{-\kappa }\widetilde H(\epsilon \widetilde x_2,\epsilon^\kappa  \widetilde y_2,\epsilon^\kappa  \widetilde z_2,\epsilon):=&Q(x_2-\epsilon^{\kappa -1}z_2)-Q(x_2)+\frac12 \epsilon^{\kappa -1}\left(Q'(\widetilde x_2)\widetilde z_2+Q'(x_2)z_2\right)\\
&-\epsilon^{2(\kappa -1)} Q(x_2-\epsilon^{\kappa -1}z_2)\left(Q''(x_2)Q(x_2)+Q'(x_2)^2\right).
% \frac12 \epsilon^{\kappa -1}\widetilde y_2 (Q'(\widetilde x_2)-Q'(x_2))\\
% &+\frac34 \epsilon^{\kappa -1}Q'(x_2)(Q(x_2)-Q(x_2-\epsilon^{\kappa -1}z_2))\\
% &+\frac14 \epsilon^{2(\kappa -1)} Q''(x_2)Q(x_2-\epsilon^{\kappa -1}z_2)\\
% &+\frac{1}{16}\epsilon^{2(\kappa -1)}(Q'(x_2))^2 z_2,\\
\end{align*}
This follows from a simple calculation. Notice that $(x_2,y_2,z_2,\epsilon)$ on the right hand side are given by the inverse of \eqref{nft}. It is easy to see that
\begin{align}
 \widetilde W(\epsilon \widetilde x_2,\epsilon^\kappa  \widetilde y_2,\epsilon^\kappa  \widetilde z_2,\epsilon)=\mathcal O(\epsilon^{3\kappa -2}),\quad W=F,G,H,\label{eq:X2exp}
\end{align}
uniformly on compact sets $(\widetilde x,\widetilde y_2,\widetilde z_2)\in \mathbb C^3$.

Next, we turn to proving the property \eqref{Xcond}. For this, we define $(\widetilde x,\widetilde y,\widetilde z)$ and $(x,y,z)$ by \eqref{xyzx2y2z21} (with and without tildes).
Then \eqref{nft} takes the following form
 \begin{equation}\nonumber
 \begin{aligned}
 \widetilde x &= x+P_x'(x,\epsilon) y\\
 \widetilde y &= y-P(x,\epsilon)+\frac12 P_x'(x,\epsilon) z,\\
  \widetilde z &= z- P_x'(x,\epsilon)P(x,\epsilon),%+\frac14 \epsilon^{k-1} Q'(x_2) y_2.
%  \widetilde z &=z -{P}'_x(x,\epsilon){P}(x,\epsilon)+\frac14 {P}'_x(x,\epsilon)y,
 \end{aligned}
 \end{equation}
 recall \eqref{Pdefn}.
Moreover, it is then a simple calculation to show that $\widetilde F, \widetilde G$ and $\widetilde H$ can be written in the following form:
\begin{align*}
\widetilde F(\widetilde x,\widetilde y,\widetilde z,\epsilon):=&P(x-z,\epsilon)-P(\widetilde x,\epsilon)-P_x'(x,\epsilon) z\\
&+P_{xx}''(x,\epsilon)P(x,\epsilon)y,\\
\widetilde G(\widetilde x,\widetilde y,\widetilde z,\epsilon):=&\frac12 P_x'(x,\epsilon)\left(P(x-z,\epsilon)-P(x,\epsilon)+\frac12 P_x'(x)z\right)\\
&+\frac12 \left(P_x'(\widetilde x,\epsilon)-P_x'(x,\epsilon)\right)\widetilde y,\\
\widetilde H(\widetilde x,\widetilde y,\widetilde z,\epsilon):=&P(x-z,\epsilon)-P(x,\epsilon)+\frac12 \left(P_x'(\widetilde x,\epsilon)\widetilde z+P_x'(x,\epsilon)z\right)\\
&-P(x-z,\epsilon)\left(P_{xx}''(x,\epsilon)P(x,\epsilon)+P_x'(x,\epsilon)^2\right).
% \frac12 \epsilon^{k-1}\widetilde y_2 (Q'(\widetilde x_2)-Q'(x_2))\\
% &+\frac34 \epsilon^{k-1}Q'(x_2)(Q(x_2)-Q(x_2-\epsilon^{k-1}z_2))\\
% &+\frac14 \epsilon^{2(k-1)} Q''(x_2)Q(x_2-\epsilon^{k-1}z_2)\\
% &+\frac{1}{16}\epsilon^{2(k-1)}(Q'(x_2))^2 z_2,\\
\end{align*}
% For reference, we notice that
% \begin{equation}\label{eq:RST0}
% \begin{aligned}
% \widetilde F(\widetilde x,\widetilde y,\widetilde z,0)=\\
% \widetilde G(\widetilde x,\widetilde y,\widetilde z,0)=\\
% \widetilde H(\widetilde x,\widetilde y,\widetilde z,0)=
% \end{aligned}
% \end{equation}
% This follows from a simple calculation.
% where
% This gives the following system
% \begin{align*}
% \dot{\widetilde x} &= P(\widetilde x,\epsilon)+\widetilde F(\widetilde x,\widetilde y,\widetilde z,\epsilon)
% \end{align*}
%
We then prove \eqref{Xcond} in each of the directional charts corresponding to $\breve x=\pm 1$, $\breve y=\pm 1$, $\breve z=\pm 1$, and $\breve \epsilon=\pm 1$. The details of $\breve \epsilon=1$ follows from \eqref{X2exp}. Consider the $\breve x=1$-chart with chart-specific coordinates $(r_1,y_1,z_1,\epsilon_1)$, defined by
\begin{align}\label{eq:thiss}
 \begin{cases}
  x = r_1,\\
  y =r_1^\kappa  y_1,\\
  z = r_1^\kappa  z_1,\\
  \epsilon = r_1 \epsilon_1.
 \end{cases}
\end{align}
We define $(\widetilde r_1,\widetilde y_1,\widetilde z_1,\widetilde \epsilon_1)$ similarly by adding tildes to \eqref{thiss} ($\widetilde \epsilon=\epsilon$), so that $\widetilde x = \widetilde r_1$, $\widetilde y = \widetilde r_1^\kappa  \widetilde y_1$, etc. Then
\begin{align*}
 P(r_1,r_1\epsilon_1) = r_1^\kappa  P(1,\epsilon_1),
\end{align*} and \eqref{nft} takes the following form
 \begin{equation}\nonumber
 \begin{aligned}
 \widetilde r_1 &= r_1\left(1 +r_1^{2(\kappa -1)} P_x'(1,\epsilon_1) y_1\right)\\
 \widetilde y_1 &= y_1-P(1,\epsilon_1)+\frac12 r_1^{\kappa -1} P_x'(1,\epsilon_1) z_1,\\
  \widetilde z_1 &= z_1- r_1^{\kappa -1}P_x'(1,\epsilon_1)P(1,\epsilon_1).%+\frac14 \epsilon^{\kappa -1} Q'(x_2) y_2.
%  \widetilde z &=z -{P}'_x(x,\epsilon){P}(x,\epsilon)+\frac14 {P}'_x(x,\epsilon)y,
 \end{aligned}
 \end{equation}
 with
 \begin{align*}
  \widetilde \epsilon_1 = \epsilon_1 \left(1 +r_1^{2(\kappa -1)} P_x'(1,\epsilon_1) y_1\right)^{-1}.
 \end{align*}
Consequently, we have
\begin{align*}
\widetilde F(\widetilde r_1,\widetilde r_1\widetilde y_1,\widetilde r_1 \widetilde z_1, \widetilde r_1 \widetilde \epsilon_1):=&r_1^{\kappa } P(1-r_1^{\kappa -1} z_1,\epsilon_1)-\widetilde r_1^\kappa  P(1,\widetilde \epsilon_1)-r_1^{2\kappa -1}P_x'(1,\epsilon_1) z_1\\
&+r_1^{3\kappa -2} P_{xx}''(1,\epsilon_1)P(1,\epsilon_1)y_1.
\end{align*}
By Taylor-expanding the right hand side with respect to $\widetilde r_1\rightarrow 0$, we find that
\begin{align*}
 \widetilde F(\widetilde r_1,\widetilde r_1\widetilde y_1,\widetilde r_1 \widetilde z_1, \widetilde r_1 \widetilde \epsilon_1) = \mathcal O(\widetilde r_1^{3\kappa -2}),
\end{align*}
uniformly on compact sets $(\widetilde y_1,\widetilde z_1,\widetilde \epsilon_1)\in \mathbb C^3$.
 The details of $\widetilde W(\widetilde r_1,\widetilde r_1\widetilde y_1,\widetilde r_1 \widetilde z_1, \widetilde r_1 \widetilde \epsilon_1)$, $W=G,H$, as well as the details of the remaining charts, are similar and therefore left out for simplicity.

\end{document}